\documentclass[11pt,reqno]{amsart}

\usepackage[utf8]{inputenc}
\usepackage{enumerate}
\usepackage{dsfont}
\usepackage{color}
\usepackage{a4wide}
\usepackage[all]{xy}
\usepackage[colorlinks=true]{hyperref}
\usepackage{bbm}

\hypersetup{%
    ,urlcolor=blue
    ,citecolor=red
    ,linkcolor=blue
    }
\usepackage{amssymb}
\usepackage{parskip}
\usepackage[framemethod=tikz]{mdframed}
\usepackage{pifont}
\usepackage[geometry]{ifsym}

\allowdisplaybreaks

\definecolor{mycolor}{rgb}{0.122, 0.435, 0.698}

\newmdenv[innerlinewidth=0.5pt, roundcorner=4pt,linecolor=mycolor,innerleftmargin=6pt,
innerrightmargin=6pt,innertopmargin=6pt,innerbottommargin=6pt]{mybox}

\makeatletter
\renewcommand*{\eqref}[1]{%
  \hyperref[{#1}]{\textup{\tagform@{\ref*{#1}}}}%
}
\makeatother

\newcommand{\RN}[1]{%
  \textup{\uppercase\expandafter{\romannumeral#1}}%
}

\def\R{\mathbb{R}}

\newcommand{\blackdiamond}{\textrm{\FilledSmallDiamondshape}}
\newcommand{\whitediamond}{\textrm{\small{\SmallDiamondshape}}}
\newcommand{\Id}{\rm Id}
\newcommand{\Iq}{I_q}

\newtheorem{theorem}{Theorem}[section]
\newtheorem{lemma}[theorem]{Lemma}

\theoremstyle{definition}

\theoremstyle{Assumption}
\newtheorem*{assumption}{Assumption}

\theoremstyle{remark}
\newtheorem{remark}[theorem]{Remark}

\theoremstyle{Proposition}
\newtheorem{proposition}[theorem]{Proposition}

\theoremstyle{Corollary}
\newtheorem{corollary}[theorem]{Corollary}

\numberwithin{equation}{section}

\begin{document}
	\newpage
	\title[$\sigma$-Antithetic Multilevel Monte Carlo method]{Central Limit Theorem for the $\sigma$-antithetic multilevel Monte Carlo method}
	
	%    Remove any unused author tags.
	
	%    author one information
	\author{Mohamed BEN ALAYA}
	\address{Mohamed Ben Alaya, Laboratoire De Math\'ematiques Rapha{\"{e}}l Salem, UMR 6085, Universit\'e De Rouen, Avenue de L'Universit\'e Technop\^ole du Madrillet, 76801 Saint-Etienne-Du-Rouvray, France}
	\curraddr{}
	\email{mohamed.ben-alaya@univ-rouen.fr}
	\thanks{}
	
	%    author two information
	\author{Ahmed KEBAIER}
	\address{Ahmed, Kebaier, Université Sorbonne Paris Nord, LAGA, CNRS, UMR 7539,  F-93430, Villetaneuse, France}
	\curraddr{}
	\email{kebaier@math.univ-paris13.fr}
	\thanks{This research is supported by Laboratory of Excellence MME-DII, Grant no. ANR11LBX-0023-01 (\url{http://labex-mme-dii.u-cergy.fr/}).  Ahmed Kebaier benefited from the support of the chair Risques Financiers, Fondation du Risque.}
	
	%    author two information
	\author{Thi Bao Tram NGO}
	\address{Thi Bao Tram NGO, Université Sorbonne Paris Nord, LAGA, CNRS, UMR 7539,  F-93430, Villetaneuse, France}
	\curraddr{}
	\email{ngo@math.univ-paris13.fr}
	\thanks{}
	
	
	
	\date{\today}
	
	\maketitle
	\begin{abstract}
	In this paper, we introduce the $\sigma$-antithetic multilevel Monte Carlo (MLMC) estimator for a multi-dimensional diffusion which is an extended version of the original antithetic MLMC one introduced by Giles and Szpruch \cite{a}. 
	Our aim is to study the asymptotic behavior of the weak errors involved in this new algorithm. 
	Among the obtained results, we prove that the error between on the one hand the average of the Milstein scheme without Lévy area and its $\sigma$-antithetic version build on the finer grid and on the other hand the  coarse approximation stably converges in distribution with a rate of order 1. 	
	We also prove that the error between the Milstein scheme without  L\'evy area and its $\sigma$-antithetic version stably converges in distribution with a rate of order $1/2$. More precisely, we have a functional limit theorem on the asymptotic behavior of the joined distribution of these errors based on a triangular array approach (see e.g. Jacod \cite{c}). 
Thanks to this result, we establish a central limit theorem of Lindeberg-Feller type for the $\sigma$-antithetic MLMC estimator. 
The time complexity of the algorithm is carried out.
	
	\end{abstract}
	
	\section{Introduction}
	In recent years, the  multilevel Monte Carlo (MLMC) algorithm, used to approximate $\mathbb E[\varphi(X_t,0\le t\le T)]$ for a given functional $\varphi$ and a stochastic process $(X_t)_{0\le t\le T}$,  have become a hot topic. This method introduced by Giles \cite{Gil},  that may  be seen as an extension of the works of Heinrich \cite{heinrich01} and Kebaier \cite{k},  is well known for reducing  significantly the approximation time complexity compared to a classical Monte Carlo method. Many authors have since been interested in the study of a central limit theorem associated to the MLMC estimator that can be found in the recent works by Ben Alaya and Kebaier \cite{h,b}, Dereich and Li \cite{Dereich_2016}, Giorgi et al. \cite{Gior},   H\"oel and Krumscheid \cite{HOEL} and Kebaier and Lelong \cite{KebLel}.  Like for the classical Monte Carlo method, obtaining a central limit theorem is important for the practical implementation of the MLMC method (see e.g. H\"oel  et al. \cite{Temp}). More recently, Giles and Szpruch \cite{a} introduced an antithetic version of the Milstein MLMC estimator without L\'evy area  that achieves the optimal complexity $O(\Delta_n^{-2})$ for a given precision $\Delta_n$ as for an unbiased Monte Carlo estimator. The efficiency of the antithetic MLMC estimator  was validated through a broad array of applications that can be found in  \cite{GilSzp1,GilSzp2}.	
	  Since then, many new studies were interested on  several types of use of the antithetic MLMC estimator ( see e.g. Debrabant and R\"ossler \cite{Debrabant_2015}, Debrabant et al. \cite{Debrabant_2019},  Al Gerbi et al.  \cite{Algerbi1,Algerbi2}). 
	  However,  the problem of studying the validity of  the central limit theorem for the antithetic  MLMC algorithm   has not been addressed in previous research.  In the present paper, we first introduce an extended version of this antithetic MLMC method  that we call  $\sigma$-antithetic MLMC estimator which allows permutations between the finer $m$ brownian increments associated to each corse increment with $m\ge2$. Let us emphasize that the original antithetic MLMC method  introduced in \cite{a} corresponds to $m=2$.
Then, we establish a central limit theorem on the $\sigma$-antithetic MLMC algorithm. This new result fills the gap in the literature for MLMC methods and yields new insights on the practical implementation  of the antithetic MLMC algorithm.
In order to establish this result, we prove a functional limit theorem for the normalized error on two consecutive levels for the joined distribution of  the couple 
\begin{equation}\label{joinederrors}
\big(\sqrt{n}(X^{nm}-X^{\sigma,nm}), n((X^{nm}+X^{\sigma,nm})/2-X^n)\big),
\end{equation}
where  $X^{nm}$ denotes the Milstein scheme with time step $T/mn$ without L\'evy area and $X^{\sigma,nm}$ is its $\sigma$-antithetic version.  This result extends the stable convergence limit theorem obtained by Ben Alaya and Kebaier \cite{b} for the normalized error on two consecutive levels  $\sqrt{n}(\tilde X^{mn} -\tilde X^n)$ where $\tilde X^{n}$ denotes the Euler scheme with time step $T/n$.  The proof of this result, written in a multidimensional setting,  relies on combining the limit theorems on martingale triangular arrays in Jacod \cite{c} with technics used in Jacod \cite{Jacod2004} and Jacod and Protter~\cite{d}. 

The rest of this paper is organized as follows. In Section 2, we recall from Giles and Szpruch \cite{a} the Milstein scheme without L\'evy area using our own notations and we introduce our assumptions. In Section 3, we introduce  the $\sigma$-antithetic scheme \eqref{eq:2.6.2} as well as the $\sigma$-antithetic MLMC estimator \eqref{eq:4.4.1} and  prove our main results namely Theorem \ref{thm:main}  a functional limit theorem for the couple of normalized errors \eqref{joinederrors} and Theorem \ref{thm:clt} the central limit theorem for the $\sigma$-antithetic MLMC estimator.  Section 4 gives the details of the error expansion needed to prove Theorem \ref{thm:main} with specifying the main and rest terms.  Based on these expansion, we study in Section 5 the asymptotic behaviors of the joined distribution of the main terms. The rest terms are treated in appendices \ref{app:B} and \ref{app:A}. Appendix \ref{app:C} is dedicated to recall some theoretical tools that we use throughout the paper.  

	\section{General framework}
	\subsection{Milstein scheme without L\'evy area}
	We consider the $d$-dimensional SDE  driven by a $q$-dimensional Brownian  motion $W=(W^1,\hdots,W^q)^{\top}$, $q\geq 1$, solution to
	\begin{align}\label{eq:2.1}
		X_t=&x_0+\int_0^tf(X_s)ds+\int_0^tg(X_s)dW_s, \mbox{ for } t\in [0,T], \, T>0,
	\end{align}
		where $x_0\in\mathbb{R}^d$, $f\in\mathcal{C}^2(\mathbb{R}^d,\mathbb{R}^d)$ and $g\in\mathcal{C}^2(\mathbb{R}^d,\mathbb{R}^{d\times q})$. 
		In what follows, we assume that $g$ does not have a commutativity property (see assumption (\nameref{Assume}) below). Without loss of generality we will take the solution of \eqref{eq:2.1} on the interval $[0,1]$ rather than $[0,T]$, $T>0$. We will consider a time grid on $[0,1]$ with a uniform time step  $\Delta_n=\frac{1}{n}$, $n\in\mathbb N$.
\paragraph{\bf{Notations}} Throughout this paper, we will  use the following notations:
\begin{itemize}
\item For $g\in\mathcal{C}^2(\mathbb{R}^d,\mathbb{R}^{d\times q})$, we introduce the tensor function $\{ h_{\ell jj'}, \, 1\le \ell\le d, 1\le j,j'\le q\}$ defined by
	$$h_{\ell jj'}(x)=\frac{1}{2}\nabla g_{\ell j}^{\top}(x)g_{\bullet j'}(x)= \frac{1}{2} \sum_{\ell'=1}^d \frac{\partial g_{\ell j}}{\partial{x_{\ell'}}}(x)g_{\ell'j'}(x),\quad  x\in\mathbb{R}^d$$
	with  $\nabla g_{\ell j}=(\frac{\partial g_{\ell j}}{\partial x_1},\cdots,\frac{\partial g_{\ell j}}{\partial x_d})^{\top}\in \mathbb R^d$ and $g_{\bullet j'}=(g_{1j'},\dots,g_{dj'})^{\top}\in\mathbb{R}^d$ is the ${j'}^{\rm th}$-column of $g$ and analogously  we also introduce the ${\ell}^{\rm th}$-row of $g$ given by  $g_{\ell\bullet }=(g_{\ell 1},\dots,g_{\ell q})$.
	The notation  $A^\top$ stands for the  transpose of the given matrix $A$. 
	\item For  $\ell\in\{1,\dots,d\}$, we denote the $q\times q$-matrix
	 $h_{\ell\bullet\bullet}=\left(\begin{array}{lll}h_{\ell11}&\hdots&h_{\ell1q}\\\vdots&\ddots&\vdots\\h_{\ell q1}&\hdots&h_{\ell qq}\end{array}\right)\in\mathbb{R}^{q\times q}$.
	 \item For more convenience, we set
	 $\mathbb{H}=(h_{1\bullet\bullet},\hdots,h_{d\bullet\bullet})^{\top}.$
	\item For any function $\psi: \mathbb R^d \rightarrow \mathbb R$, we denote
	$\nabla^2\psi=\left(\begin{array}{lll}\frac{\partial^2\psi}{\partial x_{1}\partial x_{1}}&\hdots&\frac{\partial^2\psi}{\partial x_{1}\partial x_{d}}\\\vdots&\ddots&\vdots\\\frac{\partial^2\psi}{\partial x_{d}\partial x_{1}}&\hdots&\frac{\partial^2\psi}{\partial x_{d}\partial x_{d}}\end{array}\right)$ the Hessian $d\times d$ matrix of $\psi$.
	\item For any $d$-dimensional function $f$, we denote its Jacobian matrix as $\nabla f =(\nabla f_1 ,\dots,\nabla f_d )^{\top}$.
	\item Let $\whitediamond$ denotes  the Frobenius inner products that is for any matrices  $A$ and $B\in \mathbb \mathcal M_{p\times q}(\mathbb R)$ 
$$
A\whitediamond B=\sum_{j=1}^p\sum_{j'=1}^q A_{jj'}B_{jj'} \in \mathbb R.
$$
Moreover, we introduce the operator $\blackdiamond$ defined by:  for any $A_{\ell\ell'}\in M_{p\times q}(\mathbb{R})$, $\ell\in\{1,\dots,r\}$ and $\ell'\in\{1,\dots,s\}$ with $r,s\in\mathbb N\setminus\{0\}$
$$
\left(\begin{array}{ccc}
A_{11}&\hdots&A_{1s}\\\vdots&\ddots&\vdots\\A_{r1}&\hdots&A_{rs}
\end{array}\right)\blackdiamond B=\left(\begin{array}{ccc}
A_{11}\whitediamond  B&\hdots&A_{1s}\whitediamond  B\\\vdots&\ddots&\vdots\\A_{r1}\whitediamond  B&\hdots&A_{rs}\whitediamond B
\end{array}\right)\in\mathbb R^{r\times s}.$$
\item We have the following property for any matrices $U$ and $A$ respectively in ${M}_{p\times 1}(\mathbb{R})$ and ${M}_{p\times p}(\mathbb{R})$
\begin{equation}\label{prop:bd}
 U^{\top} AU= A\blackdiamond (UU^\top).
\end{equation}
\item We denote $\eta_n(t)=\frac{[nt]}{n}$ for $t\in[0,1]$. For $i\in\{1,\dots,n\}$, $k\in\{1,\dots,m\}$, $n,m\in\mathbb{N}\backslash\{0,1\}$, we denote $\Delta W_i=W_{\frac{i}{n}}-W_{\frac{i-1}{n}}$ and $\delta W_{ik}=W_{\frac{m(i-1)+k}{nm}}-W_{\frac{m(i-1)+k-1}{nm}}$.
\item $\mathcal{S}_m$ stands for the set of all permutations of order $m$.  	
\item For $i\in\{1,\dots,n\}$, $k\in\{1,\dots,m\}$, $m\in\mathbb{N}\backslash\{0,1\}$, and $\tilde \sigma \in \mathcal{S}_m$ we denote the $\sigma$-algebra $\mathcal{F}_{\frac{i-1}{n}} ^{k,\tilde \sigma}=\mathcal{F}_{\frac{i-1}{n}}\bigvee \sigma(\delta W_{i\tilde\sigma(k')}:1\leq k'\leq k)$, where $(\mathcal{F}_t)_{t\in[0,1]}$ denotes the natural filtration of the brownian motion $W$ .
\item For $p>0$, let $(\Gamma^n)_{n\in\mathbb N}$ be a sequence of processes  in $L^p$. By $\Gamma^n \stackrel{L^p}\rightarrow0$  (resp. $\Gamma^n \stackrel{\mathbb P}\rightarrow0$) as  $n$ tends to infinity, we mean that $\sup_{s\le t} |\Gamma_s^n| \stackrel{L^p}{\rightarrow}0$  (resp. $\sup_{s\le t} |\Gamma_s^n| \stackrel{\mathbb P}{\rightarrow}0$) for all $t\in[0,1]$ as $n$ tends to infinity.  
\item For  any  block matrix $A=(A_{ij})$, the notation $|A|$ stands for the $L^1$-matrix norm, that satisfies  $|A|=\sum_{ij}|A_{ij}|$.
\end{itemize}

Thanks to the above  notations, the original Milstein scheme introduced in \cite{j} starting at $x_0$ can be rewritten in a compact form given by the following induction on the integer $i\in\{1,\dots,n\}$
	\begin{multline*}
		X^{{\rm Mil},n}_{\frac{i}{n}}=X^{{\rm Mil},n}_{\frac{i-1}{n}}+f(X^{{\rm Mil},n}_{\frac{i-1}{n}})\Delta_n+g(X^{{\rm Mil},n}_{\frac{i-1}{n}})\Delta W_{i}+\mathbb{H}(X^{{\rm Mil},n}_{\frac{i-1}{n}})\blackdiamond(\Delta W_{i}\Delta W_{i}^{\top}-{\Iq}\Delta_n -\mathcal{A}_i),
	\end{multline*} 
where   $\Delta W_{i}=W_\frac{i}{n}-W_{\frac{i-1}{n}}$ is the increment on the coarser partition, ${\Iq}=(\delta_{jj'})_{1\leq j,j'\leq q}$ is the correlation matrix for the driving Brownian paths  and $\mathcal{A}_{i}\in\mathbb{R}^{q\times q}$ is the L\'evy area defined by
	$$\mathcal{A}_{ijj'}=\int_{\frac{i-1}{n}}^{\frac{i}{n}}(W^{j}_s-W^{j}_{\frac{i-1}{n}})dW^{j'}_s-\int_{\frac{i-1}{n}}^{\frac{i}{n}}(W^{j'}_s-W^{j'}_{\frac{i-1}{n}})dW^{j}_s,\hskip 1cm j,j'\in\{1,\hdots,q\}.$$
In many applications, the simulation of L\'evy areas are very complicated. Recently,  Giles and Szpruch \cite{a} proposed to build a suitable antithetic MLMC estimator based on the Milstein scheme without the L\'evy area that achieves the optimal complexity $O(\Delta_n^{-2})$ for a given precision $\Delta_n$ as for an unbiased Monte Carlo estimator. Therefore, let us introduce the so called truncated Milstein scheme starting at $x_0$ defined by induction on the integer $i\in\{1,\dots,n\}$
	\begin{align}\label{eq:2.3}
		X^n_{\frac{i}{n}}=&X^n_{\frac{i-1}{n}}+f(X^n_{\frac{i-1}{n}})\Delta_n+g(X^n_{\frac{i-1}{n}})\Delta W_{i}+\mathbb{H}(X^{n}_{\frac{i-1}{n}})\blackdiamond(\Delta W_{i}\Delta W_{i}^{\top}-{\Iq}\Delta_n ).
		\end{align} 
\subsection{Settings and some standard results}
		
In what follows we introduce our assumption (\nameref{Assume}) on coefficients $f$ and $g$ in the spirit of Giles and Szpruch \cite{a}. Our condition is stricter than the one in  \cite{a} as we aim to prove functional limit theorems for this method. We also recall some standard results on the moment properties of \eqref{eq:2.3} (see Lemma 4.2, Corollary 4.3  and Lemma 4.4 of \cite{a}).
\begin{assumption}[{\textbf{H$_{f,g}$}}]\label{Assume}
	Let $f\in\mathcal{C}^3(\mathbb{R}^d,\mathbb{R}^d)$ and $g\in\mathcal{C}^3(\mathbb{R}^d,\mathbb{R}^{d\times q})$. We assume that 
	\begin{itemize}
		\item   there exists a positive constant $L$ such that 
		%%%multi-index notation
	\begin{align*}
		&\left|\frac{\partial^{|\alpha|} f}{\partial x^{\alpha}}\right|\leq L, \hskip 1cm	\left|\frac{\partial^{|\alpha|} g}{\partial x^{\alpha}}\right|\leq L,\hskip 1cm	\left|\frac{\partial ^{|\beta|}h}{\partial x^{\beta}}\right|\leq L
	\end{align*}
where $\alpha,\beta\in\mathbb{N}^d$, $\alpha=(\alpha_1,\dots,\alpha_d)^{\top}$, $\beta=(\beta_1,\dots,\beta_d)^{\top}$ are two multi-indices  such that $|\alpha|=\sum_{i=1}^d\alpha_i\le 3$, $|\beta|=\sum_{i=1}^d\beta_i\le 2$.
\item the diffusion coefficient $g$ does not have a commutativity property which gives $h_{\ell jj'}= h_{\ell j'j}$ for all $\ell\in\{1,\dots,d\}$ and $j,j'\in\{1,\dots,q\}$.
\end{itemize}	
\end{assumption}
\begin{lemma}\label{lem:1}
        Under (\nameref{Assume}), for $p\geq2$  there exists a constant $C_p$, independent of $n$, such that
	$$\mathbb{E}\left(\max_{0\leq i\leq n}|X^n_{\frac{i}{n}}|^p\right)\leq C_p,\;\;\mbox{and }\;\;
	\mathbb{E}\left(\max_{0\leq i\leq n}|X^n_{\frac{i}{n}}-X_{\frac{i}{n}}|^p\right)\leq C_p\Delta_n^{p/2}.$$
\end{lemma}
\begin{corollary}\label{cor:1}
	Under (\nameref{Assume}), for $p\geq2$ there exists a constant $C_p$, independent of $n$, such that 
	$$\mathbb{E}\left(\max_{0\leq i\leq n}|f_\ell(X^n_{\frac{i}{n}})|^p\right)\leq C_p,\hskip 1cm \mathbb{E}\left(\max_{0\leq i\leq n}|g_{\ell j}(X^n_{\frac{i}{n}})|^p\right)\leq C_p,$$
	and $$\mathbb{E}\left(\max_{0\leq i\leq n}|h_{\ell jj'}(X^n_{\frac{i}{n}})|^p\right)\leq C_p$$
	for all $1\leq \ell \leq d$ and $1\leq j,j'\leq q$.
\end{corollary}
\begin{lemma}\label{lem:2}
Under (\nameref{Assume}), for $p\geq2$, there exists a constant $C_p$, independent of $n$, such that
	$$\max_{1\leq i\leq n}\mathbb{E}\left(|X^n_{\frac{i}{n}}-X^n_{\frac{i-1}{n}}|^p\right)\leq C_p\Delta_n^{p/2}.$$
\end{lemma}
	\section{Main results}
\paragraph{\bf{The $\sigma$-antithetic scheme}} In view of running a MLMC method, we consider two types of schemes, a coarser one and a finer one. The antithetic MLMC estimator was  introduced in \cite{a} for $m=2$.  For each level, the main idea consists in switching the two finer Brownian increments to obtain an antithetic version of the approximation scheme. In order to extend this idea for a general $m\in\mathbb N^*\setminus{\{1\}}$, we consider $\sigma\in\mathcal{S}_m\setminus\{\Id\}$  and for each level  $\ell\in\{1,\dots,L\}$, we introduce the $\sigma$-antithetic  scheme  $X^{m^\ell,\sigma}$ obtained by permuting the $m$ finer Brownian increments lying in each of the coarse intervals with length $1/m^{\ell-1}$. Based on this, we will also introduce the $\sigma$-antithetic MLMC estimator. To do so, we  set the scheme given by the equation \eqref{eq:2.3} as the coarser approximation with time step $1/m^{\ell-1}$.
The  finer scheme with time step $1/m^{\ell}$  can be rewritten as follows : for $i\in\{1,\dots,m^{\ell-1}\}$ and $k\in\{1,\dots ,m\}$, 
\begin{multline}\label{eq:2.6.1}
	X^{m^{\ell}}_{\frac{m(i-1)+k}{m^{\ell}}}=X^{m^{\ell}}_{\frac{m(i-1)+k-1}{m^{\ell}}}+f(X^{m^{\ell}}_{\frac{m(i-1)+k-1}{m^{\ell}}})\frac{\Delta_{m^{\ell-1}}}{m}+g(X^{m^{\ell}}_{\frac{m(i-1)+k-1}{m^{\ell}}})\delta W_{ik}\\+\mathbb{H}(X^{m^{\ell}}_{\frac{m(i-1)+k-1}{m^{\ell}}})\blackdiamond (\delta W_{ik}\delta W_{ik}^{\top}-{\Iq}\frac{\Delta_{m^{\ell-1}}}{m}),
\end{multline}
where $\delta W_{ik}=W_{\frac{m(i-1)+k}{m^{\ell}}}-W_{\frac{m(i-1)+k-1}{m^{\ell}}}\in\mathbb{R}^q$. Now, for a given $\sigma\in \mathcal{S}_m\setminus\{{\Id}\}$ our $\sigma$-antithetic scheme is defined by  
\begin{multline}\label{eq:2.6.2}
	X^{m^{\ell},\sigma}_{\frac{m(i-1)+k}{m^{\ell}}}=X^{m^{\ell},\sigma}_{\frac{m(i-1)+k-1}{m^{\ell}}}+f(X^{m^{\ell},\sigma}_{\frac{m(i-1)+k-1}{m^{\ell}}})\frac{\Delta_{m^{\ell-1}}}{m}+g(X^{m^{\ell},\sigma}_{\frac{m(i-1)+k-1}{m^{\ell}}})\delta W_{i\sigma(k)}\\+\mathbb{H}(X^{m^{\ell},\sigma}_{\frac{m(i-1)+k-1}{m^{\ell}}}) \blackdiamond(\delta W_{i\sigma(k)}\delta W_{i\sigma(k)}^{\top}-{\Iq}\frac{\Delta_{m^{\ell-1}}}{m}).
\end{multline}
When $\sigma=\Id$, we clearly have $X^{m^{\ell},{\Id}}=X^{m^{\ell}}$.
Throughout the paper we take  $\boxed{\sigma(k)=m-k+1}$. The reason for fixing  $\sigma$ in this way is explained  in Remark \ref{rk:fix_sigma}. 
	 Since  the increments $(\delta W_{ik})_{1\le i\le m^{\ell-1}, 1\le k \le m}$ are independent and identically distributed, it is obvious  that $X^{m^{\ell},\sigma}\stackrel{\rm Law}{=}X^{m^{\ell}}$ and for any $i\in\{1,\dots,m^{\ell-1}\}$ and $k\in\{1,\dots,m\}$,  $X^{m^{\ell},\sigma}_{\frac{m(i-1)+k}{m^{\ell}}}$ is $\mathcal{F}_{\frac{i-1}{m^{\ell-1}}}^{k,\sigma}$-measurable.
\paragraph{\bf{The $\sigma$-antithetic MLMC method}}	
	Recall that  the idea of the original multilevel Monte Carlo method (MLMC)  is based on writing $\mathbb{E}(\varphi(X^{m^L}_1))$  using the following  telescoping summation
	\begin{align}\label{eq:3.3}
	\mathbb{E}(\varphi(X^{m^L}_1))=\mathbb{E}(\varphi(X^1_1))+\sum_{\ell=1}^L\mathbb{E}(\varphi(X^{m^{\ell}}_1)-\varphi(X^{m^{\ell-1}}_1)).
	\end{align}
 As $X^{m^\ell,\sigma}\stackrel{ law}{=}X^{m^\ell}$, we rewrite the above telescoping sum as follows

	$$\mathbb{E}(\varphi(X^{m^L}_1))=\mathbb{E}(\varphi(X^1_1))+\sum_{\ell=1}^L\mathbb{E}\left(\frac{\varphi(X^{m^{\ell}}_1)+\varphi(X^{m^{\ell,\sigma}}_1)}{2}-\varphi(X^{m^{\ell-1}}_1)\right).$$
Then we estimate independently each  expectation using an empirical mean. Thus, the $\sigma$-antithetic MLMC estimator  ${\hat Q}$  approximates $\mathbb{E}(\varphi(X^{m^L}_1))$  by
\begin{align}\label{eq:4.4.1}
		\left\{\begin{array}{l}	
		{\hat Q}=\hat{{Q}}_{0}+\sum_{\ell=1}^L\hat{ Q}_{\ell}, \mbox{ with }\\
		\hat{{Q}}_{0}=\frac{1}{N_0}\sum_{k=1}^{N_0}\varphi(X^{1}_{1,k})\mbox{ and }
		\hat{ Q}_{\ell}=\frac{1}{N_{\ell}}\sum_{k=1}^{N_{\ell}}\left(\frac{\varphi(X^{m^{\ell}}_{1,k})+\varphi(X^{m^{\ell,\sigma}}_{1,k})}{2}-\varphi(X^{m^{\ell-1}}_{1,k})\right)\end{array}\right.
	\end{align}
	where for each level  $\ell\in\{1,\dots,L\}$, $(X^{m^{\ell}}_{1,k},X^{m^{\ell,\sigma}}_{1,k},X^{m^{{\ell}-1}}_{1,k})_{1\leq k\leq N_{\ell}}$ are independent copies of $(X^{m^{\ell}}_{1},X^{m^{\ell,\sigma}}_{1},X^{m^{{\ell}-1}}_{1})$ whose components are simulated using the same Brownian path and $(X^{1}_{1,k})_{1\leq k\leq N_0}$ are independent copies of $X^1_1$. 
In order to study the error of the $\sigma$-antithetic MLMC method, we assume that $\varphi\in\mathcal{C}^2(\mathbb{R}^d,\mathbb{R})$ and introduce $\bar{X}^{m^\ell,\sigma}_1= \frac{1}{2}(X^{m^\ell,\sigma}_1+X^{m^\ell}_1),$  for $\ell\in\{1,\dots,L\}$ and use a Taylor expansion  to write 
\begin{align}\label{eq:2.13}
\nonumber	\frac{1}{2}(\varphi(X^{m^\ell}_1)+\varphi(X^{m^\ell,\sigma}_1))-\varphi(X^{m^{\ell-1}}_1)
	=& \nabla \varphi^{\top}(\xi_1)(\bar{X}^{m^\ell,\sigma}_1-X^{m^{\ell-1}}_1)\\&+\frac{1}{8}(X^{m^\ell}_1-X^{m^{\ell},\sigma}_1)^{\top} {\nabla^2}{\varphi}(\xi_2)(X^{m^{\ell}}_1-X^{m^\ell,\sigma}_1),
\end{align}
where $\xi_1$ is a point lying between $\bar{X}^{m^\ell,\sigma}_1$ and $X^{m^{\ell-1}}_1$,  $\xi_2$ is a point lying between ${X}^{m^\ell,\sigma}_1$ and $X^{m^{\ell}}_1$ and ${\nabla^2}{\varphi}$ denotes the Hessian matrix of $\varphi$. More generally, if we consider the $\sigma$-antithetic MLMC method on the coarse time grid  we have to introduce the error process
$${\mathcal E}^{m^{\ell-1},m^\ell}_t=	\frac{1}{2}(\varphi(X^{m^{\ell}}_{\eta_{m^{\ell-1}}(t)})+\varphi(X^{m^{\ell},\sigma}_{\eta_{m^{\ell-1}}(t)}))-\varphi(X^{m^{\ell-1}}_{\eta_{m^{\ell-1}}(t)}), \;\; t\in[0,1].
$$
 The work of Giles and Szpruch in \cite{a} corresponds to $m=2$ and in this case they proved the $L^p$ boundedness of the process $m^\ell {\mathcal E}^{m^{\ell-1},m^\ell}$. 
In this paper, we establish this result   for a general setting with $m\in\mathbb N\backslash\{0,1\}$ and we further study its asymptotic  distribution behavior.  To do so and in view of the decomposition $\eqref{eq:2.13}$, we study the couple of two errors $\bar{X}^{m^{\ell}}_{\eta_{m^{\ell-1}}(t)}-X^{m^{\ell-1}}_{\eta_{m^{\ell-1}}(t)}$ and 
$X^{m^{\ell}}_{\eta_{m^{\ell-1}}(t)}-X^{m^{\ell},\sigma}_{\eta_{m^{\ell-1}}(t)}$, where $\bar{X}^{m^{\ell},\sigma}_{\eta_{m^{\ell-1}}(t)}=\frac{1}{2}(X^{m^{\ell},\sigma}_{\eta_{m^{\ell-1}}(t)}+X^{m^{\ell}}_{\eta_{m^{\ell-1}}(t)})$,
$t\in[0,1]$. 

At first we reduce the problem to the study of the error given by the process $(\bar{X}^{nm,\sigma}_{\eta_{n}(t)}-X^{n}_{\eta_{n}(t)}, X^{nm}_{\eta_{n}(t)}-X^{mn,\sigma}_{\eta_{n}(t)})_{0\leq t \leq 1}$, where  
$\bar{X}^{nm,\sigma}_{\eta_{n}(t)}=\frac{1}{2}(X^{nm,\sigma}_{\eta_{n}(t)}+X^{nm}_{\eta_{n}(t)})$ and
	$X^{nm}$ , $X^{nm,\sigma}$ and  $X^n$ respectively stand for the finer approximation scheme with time step $1/nm$, its antithetic version and the coarser approximation scheme with time step $1/n$, 
	with $n\in\mathbb N\backslash\{0\}$  and $m\in\mathbb N\backslash\{0,1\}$. All these approximation schemes are constructed using the same Brownian path. Second, we extend Theorem 4.10. and Lemma 4.6.   in \cite{a} to get.
\begin{lemma}\label{Lp}
	Under (\nameref{Assume}), for $p\geq2$, $\tilde{\sigma}\in\mathcal S_m$,  there exists a constant $C_p>0$, independent of the time step, such that 
 \begin{align*}
	&\mathbb{E}(\max_{0\leq i\leq n}|X^{nm,\tilde\sigma}_{\frac{i}{n}}|^p)\leq C_p, \quad\mathbb{E}(\max_{0\leq i\leq n}|X^{nm}_{\frac{i}{n}}-X^{nm,\tilde{\sigma}}_{\frac{i}{n}}|^{p}) \leq C_p\Delta_n^{p/2} \mbox{ and}\\
			&\max_{1\leq i\leq n}\max_{1\leq k\leq m}\mathbb{E}(|X^{nm,\tilde\sigma}_{\frac{m(i-1)+k-1}{nm}}-X^{nm,\tilde\sigma}_{\frac{i-1}{n}}|^p)\leq C_p{\Delta_{n}}^{p/2}.
	\end{align*}
\end{lemma}
\begin{proof}
The first and the third inequalities are  straightforward consequences of Lemma \ref{lem:1} and Lemma \ref{lem:2}. Next, we prove the second  inequality following similar arguments as in Lemma 4.6 and Theorem 4.10 in \cite{a}. As  
$X^{nm}_{\frac{i}{n}}-X^{n}_{\frac{i}{n}}\stackrel{\rm law}{=}X^{nm,\tilde{\sigma}}_{\frac{i}{n}}-X^{n}_{\frac{i}{n}}$, by Jensen inequality and Lemma 2.1, we have
\begin{align*}
	\mathbb{E}(\max_{0\leq i\leq n}|X^{nm}_{\frac{i}{n}}-X^{nm,\tilde{\sigma}}_{\frac{i}{n}}|^{p})
	\le& C_p (\mathbb{E}(\max_{0\leq i\leq n}|X^{nm}_{\frac{i}{n}}-X^{n}_{\frac{i}{n}}|^{p})+\mathbb{E}(\max_{0\leq i\leq n}|X^{nm,\tilde{\sigma}}_{\frac{i}{n}}-X^{n}_{\frac{i}{n}}|^{p}))\\
	\leq &C_p\mathbb{E}(\max_{0\leq i\leq n}|X^{nm}_{\frac{i}{n}}-X^{n}_{\frac{i}{n}}|^{p})\\
	\le& C_p (\mathbb{E}(\max_{0\leq i\leq n}|X^{nm}_{\frac{i}{n}}-X_{\frac{i}{n}}|^{p})+\mathbb{E}(\max_{0\leq i\leq n}|X^{n}_{\frac{i}{n}}-X_{\frac{i}{n}}|^{p}))
	\leq C_p \Delta_n^{p/2},
	\end{align*}
where $C_p$ is a generic positive constant.
	\end{proof}
\subsection{Functional limit theorem for the errors}
As we have  the uniform  $L^p$-boundedness  of  $(\sqrt{n}(X^{nm}_{\eta_n(t)}-X^{nm,\sigma}_{\eta_n(t)}))_{t\in[0,1]}$ (see Lemma \ref{Lp}) and $(n(\bar{X}^{nm,\sigma}_{\eta_n(t)}-X^{n}_{\eta_n(t)}))_{t\in[0,1]}$ (see Corollary \ref{bar_LP}), we get the tightness of these quantities
(see e.g.  \cite{f}).  Then, it  is natural to study the weak convergence of the couple $(\sqrt{n}(X^{nm}_{\eta_n(t)})-X^{nm,\sigma}_{\eta_n(t)})),n(\bar{X}^{nm,\sigma}_{\eta_n(t)}-X^{n}_{\eta_n(t)}))_{t\in[0,1]}$. The following theorem is our main result.
\begin{theorem}\label{thm:main}
	Under the assumption  (\nameref{Assume}),  let us denote $U^n_t=X^{nm}_{\eta_n(t)}-X^{nm,\sigma}_{\eta_n(t)}$ and $V^n_t=\bar{X}^{nm,\sigma}_{\eta_n(t)}-X^n_{\eta_n(t)}$, $t\in[0,1]$. Then  we have 
	\begin{align}\label{eq:2.10}
		(\sqrt{n}U^n,nV^n)\stackrel{\rm stably}{\Rightarrow} (U,V),\hskip 1cm \textrm{ as }n\rightarrow\infty,
	\end{align}
	with $U$ and $V$ are  solutions to
	\begin{align}
		U_t&=\sum_{j=0}^q\int_0^t \dot{F}^j_sU_sdY_s^j+\mathcal{M}_{1,t},\label{eq:2.11}\\
		V_t&=\sum_{j=0}^q\int_0^t \dot{F}^j_sV_sdY_s^j+\mathcal{M}_{2,t},\label{eq:2.12}
	\end{align}
where for $\ell\in\{1,\dots,d\}$, the $\ell$-th component of $\mathcal{M}_{1,t}$ and $\mathcal{M}_{2,t}$ are  given by  
	\begin{align*}
	&\mathcal{M}_{1,t}^{\ell}=-2\int_0^th_{\ell\bullet\bullet}(X_s)\blackdiamond dZ_{2,s}\\
	&\mathcal{M}_{2,t}^{\ell}=\\&\sum_{j=0}^q\int_0^t\bigg[\frac{m-1}{2m}\bigg(\nabla f_{\ell}(X_{s})^{\top}g_{\bullet j}(X_{s})\mathbf 1_{j\neq0}+\nabla g_{\ell j}(X_s)^{\top}f(X_s)+
	\frac{1}{2}g(X_s)^{\top}\nabla^2 g_{\ell j}(X_s)g(X_s)\blackdiamond {\Iq^j} \bigg)\\&+\frac{1}{8}U^{\top}_s\nabla^2 g_{\ell j}(X_s)U_s\bigg]dY^{j}_s+\frac{1}{2}\sum_{j=1}^q\int_0^t\bigg[\nabla g_{\ell j}(X_{s})^{\top}\mathbb{H}(X_{s})+\frac{1}{2}g(X_s)^{\top}\nabla^2g_{\ell j}(X_s)g(X_s)\\&+{\dot h}_{\ell\bullet\bullet}^s\blackdiamond g_{\bullet j}(X_s)\bigg]\blackdiamond dZ_{1,s}^{\bullet\bullet j}-\frac{1}{2}\int_0^t({\dot h}_{\ell\bullet\bullet}^s\blackdiamond U_s)\blackdiamond dZ_{2,s}+\frac{1}{2}\sum_{j,j'=1}^q\int_0^t\nabla g_{\ell j}(X_s)^{\top}[{\dot{g}^s}\blackdiamond g_{\bullet j''}(X_s)]dZ^{j\bullet j'}_{3,s}\\&+\frac{1}{2}\sum_{j=1}^q\int_0^t[g(X_s)^{\top}\nabla^2 g_{\ell  j}(X_{s})g(X_s)]\blackdiamond dZ_{3,s}^{j \bullet\bullet},
	\end{align*}
with $Y_t:=(t, W^1_t,\dots,W^q_t)^{\top}$,  $I_q^0=\mathbf 1_{q\times q}$ is the $\mathbb R^{q\times q}$ matrix with all its elements equal to 1, $\dot{F}^0=\nabla f$ and for $j\neq 0$, $I_q^j=I_q$  and $\dot{F}^j=\nabla g_{\bullet j}$,  $\dot{g}^s\in(\R^{d\times1})^{d\times q}$ is a block matrix such that for $\ell\in\{1,\dots,d\}$, $j\in\{1,\dots,q\}$, the $\ell j$-th block is given by $(\dot{g}^s)_{\ell j}=\nabla g_{\ell j}(X_s)$, $s\in[0,t]$ and  the  ${\dot  h}_{\ell \bullet\bullet}^s\in(\R^{d\times 1})^{q\times q}$ is a block matrix such that for $j$ and $j'\in\{1,\dots,q\}$, the $jj'$-th block is given by  $({\dot  h}_{\ell \bullet\bullet}^s)_{j j'}=\nabla h_{\ell jj'}( X_s)\in \R^{d\times 1}$, $s\in[0,t]$. 
Here, $Z_{1}$, $Z_{3}$ are $\mathbb{R}^{q^3}$-dimensional processes and  $Z_{2}$ is a $\mathbb{R}^{q\times q}$-dimensional process given by: for $j,j',j''\in\{1,\dots,q\}$,   
	\begin{align*}
	Z_{1,t}^{jj'j''}&=\left\{\begin{array}{lll}
	\frac{\sqrt{m-1}}{m}B^{jj'j''}_{1,t}&,&j> j'\\\frac{\sqrt{2(m-1)}}{m}B^{jjj''}_{1,t}&,&j=j'\\
	\frac{\sqrt{m-1}}{m}B^{j'jj''}_{1,t}&,&j< j'
	\end{array}\right.,
	\quad
	Z_{2,t}^{jj'}=\left\{\begin{array}{lll}
	\sqrt{\frac{m-1}{m}}B^{jj'}_{2,t}&,&j>j'\\0&,&j=j'\\-\sqrt{\frac{m-1}{m}}B^{j'j}_{2,t}&,&j< j'\end{array}\right.
	\\\\
	\mbox{ and }\quad 
	 Z_{3,t}^{jj'j''}&=\left\{\begin{array}{lll}
	\sqrt{\frac{(m-1)(m-2)}{3m^2}}B^{jj'j''}_{3,t}&,&j> j''\\ \sqrt{\frac{2(m-1)(m-2)}{3m^2}}B^{jj'j''}_{3,t}&,&j= j''\\ \sqrt{\frac{(m-1)(m-2)}{3m^2}}B^{j''j'j}_{3,t}&,&j< j''\end{array}\right.
	\end{align*}
	with $(B^{jj'j''}_1)_{\substack{1\leq j,j',j''\leq q\\j\geq j'}}$ and $(B^{jj'j''}_3)_{\substack{1\leq j,j',j''\leq q\\j\geq j''}}$   are two standard $q^2(q+1)/2$-dimensional Brownian motions and $(B^{jj'}_2)_{1\leq j'<j\leq q}$ is a standard $q(q-1)/2$-dimensional Brownian motion. Moreover, we have $B_1$,  $B_2$ and $B_3$ are independent of the original $q$-dimensional Brownian motion $W$ and also independent of each other.
	These processes are defined on an extension $(\tilde{{\Omega}},\tilde{\mathcal{F}},(\tilde{\mathcal{F}}_t)_{t\geq0},\tilde{\mathbb{P}})$ of the space $({{\Omega}},{\mathcal{F}},({\mathcal{F}}_t)_{t\geq0},{\mathbb{P}})$.
	Here we use that for $r\in\{1,3\}$ we have $$Z_{r,s}^{\bullet\bullet j}=\left(\begin{array}{ccc}Z_{r,s}^{11 j}&\hdots&Z_{r,s}^{1q j}\\\vdots&\ddots&\vdots\\Z_{r,s}^{q1 j}&\hdots&Z_{r,s}^{qq j}\end{array}\right)\in\mathbb R^{q\times q}\quad \textrm{and} \quad Z^{j\bullet j''}_{r,t}=(Z^{j1j''}_{r,t},\dots,Z^{jq j''}_{r,t})^{\top}\in\mathbb R^{q\times 1}.$$
\end{theorem}
\begin{proof}
	From section 4, equations $\eqref{eq:3.2}$ and $\eqref{eq:3.1}$, we can rewrite $U^n$ and $V^n$ as follows
	\begin{align*}
		U^n_t=&\sum_{j=0}^q\int_0^{t} (\dot{F}^{n,j}_{\eta_n(s)}\blackdiamond U^n_{\eta_n(s)})\mathbbm{1}_{s\leq \eta_n(t)}dY^{j}_s+J_t^{n,1},\\
		V^n_t=&\sum_{j=0}^q\int_0^{t} (\bar{\dot{F}}^{n,j}_{\eta_n(s)}\blackdiamond V^n_{\eta_n(s)})\mathbbm{1}_{ s\leq \eta_n(t)}dY^{j}_s+J_t^{n,2},
	\end{align*}	
	where  $Y_t:=(t,W^1_t,\dots,W^q_t)^{\top}$,
	$J_t^{n,1}=\mathcal{M}^{n,1}_t+\mathcal{R}^{n,1}_t$, $J_t^{n,2}=\mathcal{M}^{n,2}_t+\mathcal{R}^{n,2}_t$ and for $i\in\{1,\dots,n\}$ we denote
	\begin{align}\label{eq:coef}
\dot{F}^{n,j}_{\frac{i-1}{n}}=	\left\{\begin{array}{ll}\dot{f}^{n}_i,&j=0\\(\dot{g}^{n}_i)_{\bullet j},&j\in\{1,\dots,q\}\end{array}\right., \hskip 0.1cm \textrm{ where }(\dot{g}^{n}_i)_{\bullet j}=((\dot{g}^{n}_i)_{1 j},\hdots,(\dot{g}^{n}_i)_{d j})^{\top},
	\end{align}
	and 
	\begin{align}\label{eq:coefbar}
\bar{\dot{F}}^{n,j}_{\frac{i-1}{n}}=	\left\{\begin{array}{ll}\bar{\dot{f}}^{n}_i,&j=0\\(\bar{\dot{g}}^{n}_i)_{\bullet j},&j\in\{1,\dots,q\}\end{array}\right., \hskip 0.1cm \textrm{ where }(\bar{\dot{g}}^{n}_i)_{\bullet j}=((\bar{\dot{g}}^{n}_i)_{1 j},\hdots,(\bar{\dot{g}}^{n}_i)_{d j})^{\top},
	\end{align}
	with $\dot{f}^n_i\in(\R^{d\times1})^{d\times1}$ and $\dot{g}^n_i\in(\R^{d\times 1})^{d\times q}$  are block matrices  such that for  $\ell\in\{1,\dots,d\}$ the $\ell$-th block of $\dot{f}^n_i$ is given by $(\dot{f}^n_i)_{\ell}=\nabla f_{\ell}(\xi_{\frac{i-1}{n}}^{1,n})$ and for $\ell,j\in\{1,\dots,d\}$ the $\ell j$-th block of $\dot{g}^n_i$ is given by $(\dot{g}^n_i)_{\ell j}=\nabla g_{\ell j}(\xi_{\frac{i-1}{n}}^{2,n})$ with $\xi_{\frac{i-1}{n}}^{1,n}$ and $\xi_{\frac{i-1}{n}}^{2,n}$  are some vector points lying between $X^{nm}_{\frac{i-1}{n}}$ and $X^{nm,\sigma}_{\frac{i-1}{n}}$. In the same way,  $\bar{\dot f}^n_i\in(\R^{d\times1})^{d\times1}$ and $\bar{\dot g}^n_i\in(\R^{d\times 1})^{d\times q}$  are block matrices such that for $\ell\in\{1,\dots,d\}$ the $\ell$-th block of $\bar{\dot f}^n_i$ is given by $(\bar{\dot f}^n_i)_{\ell}=\nabla f_{\ell}({\bar\xi}_{\frac{i-1}{n}}^{1,n})$ and 
 $\ell,j\in\{1,\dots,d\}$   the  $\ell j$-th block of  $\bar{\dot g}^n_i$  is given by $(\bar{\dot g}^n_i)_{\ell j}=\nabla g_{\ell j}(\bar{\xi}^{2,n}_{\frac{i-1}{n}})$ with ${\bar\xi}_{\frac{i-1}{n}}^{1,n}$ and $\bar{\xi}^{2,n}_{\frac{i-1}{n}}$
 are some vector points lying between $X^{n}_{\frac{i-1}{n}}$ and $\bar{X}^{nm,\sigma}_{\frac{i-1}{n}}$. The aim now is to use  Theorem \ref{thm:2.5 1998} to get the joined convergence of our couple of errors. To do so, let us introduce the processes $Z^n_t=\sum_{j=0}^q\int_0^{t} \dot{F}^{n,j}_{\eta_n(s)}\blackdiamond \mathds{1}_{d}\mathbbm{1}_{ s\leq \eta_n(t)}(s)dY^{j}_s$, $\bar{Z}^n_t=\sum_{j=0}^q\int_0^{t} \bar{\dot{F}}^{n,j}_{\eta_n(s)}\blackdiamond \mathds{1}_{d}\mathbbm{1}_{ s\leq \eta_n(t)}(s)dY^{j}_s$ and $Z_t=\sum_{j=0}^q\int_0^{t} \dot{F}^j_{s}\blackdiamond \mathds{1}_ddY^{j}_s$, where $\mathds{1}_{d}=(1,\hdots,1)^{\top}\in\mathbb R^{d\times 1}$. Thanks to Lemma \ref{lem:1} and assumption (\nameref{Assume}), using the Burkholder-Davis-Gundy (BDG) inequality  with $p\ge 2$, there is a generic constant $C_p>0$ such that
 \begin{align*}
 	\mathbb{E}(\sup_{0\leq t\leq 1}|Z^n_t-Z_t|^p)&\le C_p\mathbb{E}(|\sum_{j=1}^q\int_0^{1} (\dot{F}^{n,j}_{\eta_n(s)}-\dot{F}^j_{s})^2\blackdiamond \mathds{1}_{d}ds|^{p/2})\\
 	&\le C_p\sum_{\ell=1}^d\sum_{j=1}^q\mathbb{E}(|\int_0^{1} (\nabla g_{\ell j}(\xi^{2,n}_{\frac{i-1}{n}})-\nabla g_{\ell j}(X_{\frac{i-1}{n}}))^2\blackdiamond \mathds{1}_{d}ds|^{p/2})\\
 	&\le C_p \Delta_n^{p/2}.
 \end{align*}
   Similarly, $\mathbb{E}(\sup_{0\leq t\leq 1}|\bar{Z}^n_t-Z_t|^p)$ is also bounded by $C_p \Delta_n^{p/2}$. Therefore, we have $Z^n-Z\stackrel{L^p}{\rightarrow}0$ and $\bar{Z}^n-Z\stackrel{L^p}{\rightarrow}0$ as $n\to\infty$. By Lemma \ref{lem:R1} and Lemma \ref{lem:R2} and Proposition \ref{prop:M}, we deduce that $(\sqrt{n}J^{n,1},nJ^{n,2})\stackrel{\rm stably}{\Rightarrow} (\mathcal{M}_{1},\mathcal{M}_{2})$ as $n\to\infty$, where the limit processes are defined on an extension  $(\tilde{{\Omega}},\tilde{\mathcal{F}},(\tilde{\mathcal{F}}_t)_{t\geq0},\tilde{\mathbb{P}})$ of the original space $({{\Omega}},{\mathcal{F}},({\mathcal{F}}_t)_{t\geq0},{\mathbb{P}})$. By Lemma \ref{lm:d5}, we get that $(Y,\sqrt{n}J^{n,1},nJ^{n,2},Z^n)$  stably converges to the limit $(Y,\mathcal{M}_{1},\mathcal{M}_{2},Z)$ as $n\to\infty$
 Finally, by Theorem \ref{thm:2.5 1998} , we have $(Y,\sqrt{n}J^{n,1},nJ^{n,2},Z^n, \sqrt{n}U^n,nV^n)$  stably converges to the limit $(Y,\mathcal{M}_{1},\mathcal{M}_{2},Z,U,V)$  as $n\to\infty$,  where $U$ and $V$ respectively satisfy $\eqref{eq:2.11}$ and $\eqref{eq:2.12}$.  
\end{proof}	
\subsection{Central limit theorem}
The $\sigma$-antithetic Multilevel Monte Carlo method uses information from a sequence of computations with increasing step sizes and approximates the quantity  of interest $\mathbb{E}(\varphi(X_1))$ by
$$\hat{\mathcal{Q}}_n=\frac{1}{N_0}\sum_{k=1}^{N_0}\varphi(X^1_{1,k})+\sum_{\ell=1}^L\frac{1}{N_{\ell}}\sum_{k=1}^{N_{\ell}}[\frac{1}{2}(\varphi(X^{{\ell},m^{\ell}}_{1,k})+\varphi(X^{{\ell},m^{\ell},\sigma}_{1,k}))-\varphi(X^{{\ell},m^{{\ell}-1}}_{1,k})],$$
$m\in\mathbb{N}\backslash\{0,1\}$, and $L=\frac{\log{n}}{\log{m}}$.
We denote the weak error
$\epsilon_n=\mathbb{E}(\varphi(X^n_1))-\varphi(X_1))$. In the spirit of Kebaier \cite{k}, we assume that $\epsilon_n$ is of order $1/n^{\alpha}$, for any $\alpha\in[1/2,1]$. Taking advantage from Theorem \ref{central}, we are now able to establish a central limit theorem of Lindeberg Feller type on the error $\hat{\mathcal{Q}}_n-\mathbb{E}(\varphi(X_1))$. To do so, we introduce a real sequence $(a_\ell)_{\ell\in\mathbb{N}}$ of positive weights such that   
\begin{equation}\label{W}\tag{\textbf{W}} 
	\lim_{L\uparrow\infty}\sum_{\ell=1}^{L}a_{\ell}=\infty, \textrm{ for }p>2,\textrm{ and }	\lim_{L\uparrow\infty}\frac{1}{\left(\sum_{\ell=1}^{L}a_{\ell}\right)^{p/2}}	\sum_{\ell=1}^{L}a_{\ell}^{p/2}=0
\end{equation}
and we choose the same form of $N_{\ell}$ as in \cite{b}, namely
\begin{align}\label{eq:3.7}
	N_{\ell}=\frac{n^{2\alpha}}{m^{2(\ell-1)}a_{\ell}}\sum_{\ell=1}^{L}a_{\ell},\hskip 1cm \ell\in\{0,\dots,L\}\textrm{ and }L=\frac{\log{n}}{\log{m}}.
\end{align}	
This generic form for the sample size  allows us a straightforward use of Theorem \ref{thm:main} to prove a central limit theorem for the $\sigma$-antithetic MLMC estimator. 
In the sequel, we denote by $\tilde{E}$ and $\tilde{\rm Var}$ the expectation and the variance respectively defined on the probability space $(\tilde{\Omega},\tilde{\mathcal{F}},(\tilde{\mathcal{F}}_t)_{t\geq0},\tilde{\mathbb{P}})$ introduced in Theorem \ref{thm:main}. 
\begin{theorem}\label{thm:clt}
	Assume that $f$ and $g$ satisfy assumption (\nameref{Assume}). Let  $\varphi\in\mathcal{C}^2(\mathbb{R}^d,\mathbb{R})$ satisfying
	\begin{equation}\label{phi}\tag{\textbf{H$_\varphi$}}
		\begin{array}{l}	|\varphi(x)-\varphi(y)|\leq C(1+|x|^p+|y|^p)|x-y|, \textrm{ for some constant $C$ and $p>0$}\\\textrm{ with bounded second derivatives}.\end{array}
	\end{equation}
	Assume that for some $\alpha\in[1/2,1]$ we have 
	\begin{equation}\label{epsilon}\tag{\textbf{H$_{\epsilon_n}$}}
		\lim_{n\uparrow\infty}n^{\alpha}\epsilon_n=C_{\varphi}(\alpha).
	\end{equation}
	Then, for the choice of $N_{\ell}$, $\ell\in\{0,\dots,L\}$ given by the equation $\eqref{eq:3.7}$, we have 
	\begin{align*}
		n^{\alpha}(\hat{\mathcal{Q}}_n-\mathbb{E}(\varphi(X_1))){\Rightarrow} \mathcal{N}(C_{\varphi}(\alpha),\mathcal{V}), \hskip 1cm \textrm{as } n\rightarrow\infty
	\end{align*}
	with $\mathcal{V}= \tilde{\rm Var}(\nabla\varphi^{\top}(X_1)V_1+\frac{1}{8}U_1^{\top}\nabla^2\varphi(X_1)U_1),$ where the limit processes
	$U$ and $V$ are explicitly given in Theorem \ref{thm:main}.
\end{theorem}
\begin{proof}
	To simplify our notation, we give the proof for $\alpha=1$, the case $\alpha\in[1/2,1)$ is straightforward by similar arguments. At first, we rewrite the error term as follows
	$$
	\hat{\mathcal{Q}}_n-\mathbb{E}(\varphi(X_1))=\hat{\mathcal{Q}}_n^1+\hat{\mathcal{Q}}_n^2+\epsilon_n,\;\mbox{ where }
	$$
	\begin{align*}
		\hat{\mathcal{Q}}^1_n=& \frac{1}{N_0}\sum_{k=1}^{N_0}(\varphi(X^1_{1,k})-\mathbb{E}(\varphi(X^1_{1}))),\\
		\hat{\mathcal{Q}}^2_n=&\sum_{\ell=1}^L\frac{1}{N_{\ell}}\sum_{k=1}^{N_{\ell}}[\frac{1}{2}(\varphi(X^{{\ell},m^{\ell}}_{1,k})+\varphi(X^{{\ell},m^{\ell},\sigma}_{1,k}))-\varphi(X^{\ell,m^{\ell-1}}_{1,k})-\mathbb{E}(\varphi(X^{\ell,m^{\ell}}_{1})-\varphi(X^{\ell,m^{\ell-1}}_{1}))].
	\end{align*}
	 For $N_0=\frac{n^2}{a_0}\sum_{\ell=1}^{L}a_{\ell}$  we simply apply the classical central limit theorem to get 
	 	$$	n\hat{\mathcal{Q}}^1_n=\sqrt{\frac{a_0}{\sum_{\ell=1}^{L}a_{\ell}}}\sqrt{N_0}\hat{\mathcal{Q}^1_n}\stackrel{\mathbb{P}}{\rightarrow}0 \quad{as }\; n\to \infty. $$ 
	Finally, we only need to study the convergence of $n\hat{\mathcal{Q}}_n^2$ and the proof is completed by assumption (\ref{epsilon}).
	To do so,  we use Theorem \ref {central} and set
	\begin{align*}
		X_{n,\ell}:=&\frac{n}{N_{\ell}}\sum_{k=1}^{N_{\ell}}Z^{m^{\ell},m^{\ell-1}}_{1,k},\textrm{ where } (Z^{m^{\ell},m^{\ell-1}}_{1,k})_{1\le k\le N_\ell} \mbox{ are independent copies of}\\
		Z^{m^{\ell},m^{\ell-1}}_{1}:=&	\frac{1}{2}(\varphi(X^{{\ell},m^{\ell}}_{1})+\varphi(X^{{\ell},m^{\ell},\sigma}_{1}))-\varphi(X^{\ell,m^{\ell-1}}_{1})-\mathbb{E}(\varphi(X^{\ell,m^{\ell}}_{1})-\varphi(X^{\ell,m^{\ell-1}}_{1})).
	\end{align*} 
	First, we check the limit variance of $n\hat{\mathcal{Q}}^2_n$. We have
	\begin{align}\label{eq:3.8}
		\sum_{\ell=1}^{L}\mathbb{E}(X_{n,\ell})^2
		=\sum_{\ell=1}^{L}\frac{n^2}{N_{\ell}}{\rm Var}(Z^{m^{\ell},m^{\ell-1}}_{1})
		=\sum_{\ell=1}^{L}	\frac{1}{\sum_{\ell=1}^{L}a_{\ell}}a_{\ell}m^{2(\ell-1)}{\rm Var}(Z^{m^{\ell},m^{\ell-1}}_{1}).
	\end{align}
	Besides, since $\varphi\in\mathcal{C}^2(\mathbb{R}^d,\mathbb{R})$, applying Taylor expansion  twice we get
	\begin{align*}
		\frac{1}{2}(\varphi(X^{{\ell},m^{\ell}}_{1})&+\varphi(X^{{\ell},m^{\ell},\sigma}_{1}))-\varphi(X^{\ell,m^{\ell-1}}_{1})\\&=\nabla \varphi^{\top}	(\xi_1)(\bar{X}^{\ell,m^{\ell},\sigma}_{1}-X^{\ell,m^{\ell-1}}_{1})+\frac{1}{8}(X^{\ell,m^{\ell}}_{1}-X^{\ell,m^{\ell},\sigma}_{1})^{\top}\nabla^2\varphi(\xi_2)(X^{\ell,m^{\ell}}_{1}-X^{\ell,m^{\ell},\sigma}_{1}),
	\end{align*}
	for some $\xi_1$ a vector point lying between $X^{\ell,m^{\ell}}_{1}$ and $X^{\ell,m^{\ell},\sigma}_{1}$ and $\xi_2$ a vector point lying between $\bar{X}^{\ell,m^{\ell},\sigma}_{1}$ and $X^{\ell,m^{\ell-1}}_{1}$. Thus, 
	under assumption (\ref{phi}), thanks to Theorem \ref{thm:main}  we get  as $\ell\rightarrow\infty$
		\begin{align*}
		n\left[\frac{1}{2}(\varphi(X^{{\ell},m^{\ell}}_{1})+\varphi(X^{{\ell},m^{\ell},\sigma}_{1}))-\varphi(X^{\ell,m^{\ell-1}}_{1})\right]\stackrel{\rm stably}{\Rightarrow}\nabla\varphi^{\top}(X_1)V_1+\frac{1}{8}U_1^{\top}\nabla^2\varphi(X_1)U_1.	\end{align*}
	From the uniform integrability obtained by combining  (\ref{phi}) and Lemma \ref{Lp},  we get for $k\in\{1,2\}$ 
	\begin{align*}
		\mathbb{E}\left(n\left[\frac{1}{2}(\varphi(X^{{\ell},m^{\ell}}_{1})+\varphi(X^{{\ell},m^{\ell},\sigma}_{1}))-\varphi(X^{\ell,m^{\ell-1}}_{1})\right]\right)^k\underset{\ell\rightarrow\infty}{\longrightarrow}\tilde{\mathbb{E}}\left(\nabla\varphi^{\top}(X_1)V_1+\frac{1}{8}U_1^{\top}\nabla^2\varphi(X_1)U_1\right)^k.
	\end{align*}
	Consequently, 
	$
		m^{2(\ell-1)}{\rm Var}(Z^{m^{\ell},m^{\ell-1}}_{1,1}){\longrightarrow}\mathcal{V},
	$
	as ${\ell\rightarrow\infty}$. 
	Thus,  by $\eqref{eq:3.8}$ and Toeplitz lemma we get
	$
		\lim_{L\uparrow\infty}\sum_{\ell=1}^{L}\mathbb{E}(X_{n,\ell})^2=\mathcal{V}.
	$
	Finally, we only need to check the Lyapunov condition. By Burkholder's inequality and Jensen's inequality, we get for $p>2$,
	\begin{align*}
		\mathbb{E}|X_{n,\ell}|^p=\frac{n^p}{N_{\ell}^{p/2}}\mathbb{E}|\sum_{\ell=1}^{L}Z^{m^{\ell},m^{\ell-1}}_{1}|^p{\leq}C_p\frac{n^p}{N_{\ell}^p}\mathbb{E}|Z^{m^{\ell},m^{\ell-1}}_{1}|^p,
	\end{align*}
	where $C_p$ is a generic positive constant depending on $p$. Besides, Lemma \ref{Lp} ensures that there is a constant $K_p>0$ such that 
	$
		\mathbb{E}|Z^{m^{\ell},m^{\ell-1}}_{1}|^p\leq \frac{K_p}{m^{p(\ell-1)}}.
	$
	Therefore, 
	\begin{align*}
		\sum_{\ell=1}^{L}\mathbb{E}|X_{n,\ell}|^p&\leq{C}_p	\sum_{\ell=1}^{L} \frac{n^p}{N_{\ell}^{p/2}}\frac{1}{m^{p(\ell-1)}}\leq {C}_p \frac{1}{\left(\sum_{\ell=1}^{L}a_{\ell}\right)^{p/2}}	\sum_{\ell=1}^{L}a_{\ell}^{p/2}\stackrel{n\rightarrow\infty}{\rightarrow}0
	\end{align*}
	which complets the proof.
\end{proof}
\subsection{The time complexity}
The time complexity in the $\sigma$-antithetic MLMC method is given by
\begin{align*}
	C_{\sigma \mbox{-}{\rm MLMC}}&=C\times\sum_{\ell=1}^{L}N_{\ell}(2m^{\ell}+m^{\ell-1})	\textrm{ with }C>0\\
	&=C\times\sum_{\ell=1}^{L} \frac{n^2}{m^{2(\ell-1)}a_{\ell}}(2m^{\ell}+m^{\ell-1})\sum_{\ell=1}^{L}a_{\ell}\\
	&=C\times n^2(2m^2+m)\sum_{\ell=1}^{L} \frac{1}{m^{\ell}a_{\ell}}\sum_{\ell=1}^{L}a_{\ell}.
\end{align*}
This analysis  is online with the one obtained by Ben Alaya and Kebaier \cite{h} in the context of  pricing Asian options using numerical schemes with a strong convergence order  equal to 1. The optimal choice corresponding to $a^*_{\ell}=m^{-\ell/2}$, $\ell\in\{1,\dots,L\}$ leads to the optimal time complexity   $C^{*}_{\sigma \mbox{-}{\rm MLMC}}=O(n^2)$ the same one as for an unbiased Monte Carlo method having the same precision. However, this optimal weight $a^*_{\ell}$ does not satisfy (\ref{W}) which ensures Theorem \ref{thm:clt}. In what follows, we recall from \cite{h},  three examples of weights  $(a_{\ell})_{1\leq\ell\leq L}$ satisfying (\ref{W}) and for which the time complexity gets closer and closer to $C^{*}_{\sigma \mbox{-}{\rm MLMC}}$:
\begin{enumerate}[i)]
	\item The choice $a_{\ell}=1$, corresponding to $N_{\ell}=\frac{n^2}{m^{2(\ell-1)}}L$, $\ell\in\{1,\dots,L\}$ leads to the complexity $C_{\sigma \mbox{-}{\rm MLMC}}=O(n^2\log{n})$.
	\item The choice $a_{\ell}=\frac{1}{\ell}$, corresponding to $N_{\ell}=\frac{n^2\ell}{m^{2(\ell-1)}}\sum_{\ell=1}^{L}\frac{1}{\ell}$, $\ell\in\{1,\dots,L\}$ leads to the complexity  $C_{\sigma \mbox{-}{\rm MLMC}}=O(n^2\log{\log{n}})$.
	\item The choice $a_{\ell}=\frac{1}{\ell\log{\ell}}$, corresponding to $N_{\ell}=\frac{n^2\ell\log{\ell}}{m^{2(\ell-1)}}\sum_{\ell=1}^{L}\frac{1}{\ell\log{\ell}}$, $\ell\in\{1,\dots,L\}$ leads to the complexity $C_{\sigma \mbox{-}{\rm MLMC}}=O(n^2\log{\log{\log{n}}})$.
\end{enumerate} 
\section{Expanding analysis of $\sigma$-antithetic scheme}
In this section, we have two main purposes. Firstly, for $\tilde{\sigma}\in\{\Id,\sigma\}$, we give the expansion of two error terms $X^{nm,\tilde{\sigma}}_{\frac{i}{n}}-X^{nm,\tilde{\sigma}}_{\frac{i-1}{n}}$ and $\bar{X}^{nm,\sigma}_{\frac{i}{n}}-\bar{X}^{nm,\sigma}_{\frac{i-1}{n}}$ together with some related  $L^p$ estimates. Secondly, we give the expansions of the errors $U^n$ and $V^n$ with specifying the main and the rest terms. From now on we assume that  assumption (\nameref{Assume}) is satisfied.
\subsection{Expansion of the error $X^{nm,\tilde{\sigma}}_{\frac{i}{n}}-X^{nm,\tilde{\sigma}}_{\frac{i-1}{n}}$, with $\tilde{\sigma}\in\{\Id,\sigma\}$}
By $\eqref{eq:2.6.1}$ and $\eqref{eq:2.6.2}$, we have for all $k\in\{1,\dots,m\}$
\begin{multline}\label{eq:iter1}
	X^{nm,\tilde{\sigma}}_{{\frac{m(i-1)+k}{nm}}}-X^{nm,\tilde{\sigma}}_{\frac{i-1}{n}}=\sum_{k'=1}^k f(X^{nm,\tilde{\sigma}}_{\frac{m(i-1)+k'-1}{nm}})\frac{\Delta_n}{m}+\sum_{k'=1}^kg(X^{nm,\tilde{\sigma}}_{\frac{m(i-1)+k'-1}{nm}})\delta W_{i\tilde{\sigma}(k')}\\
	+\sum_{k'=1}^k\mathbb{H}(X^{nm,\tilde{\sigma}}_{\frac{m(i-1)+k'-1}{nm}})\blackdiamond(\delta W_{i\tilde{\sigma}(k')}\delta W_{i\tilde{\sigma}(k')}^{\top}-{\Iq}\frac{\Delta_n}{m}).
\end{multline}
In particular, we have 
\begin{multline}\label{eq:iter2}
	X^{nm,\tilde{\sigma}}_{\frac{i}{n}}-X^{nm,\tilde{\sigma}}_{\frac{i-1}{n}}=\sum_{k=1}^m f(X^{nm,\tilde{\sigma}}_{\frac{m(i-1)+k-1}{nm}})\frac{\Delta_n}{m}+\sum_{k=1}^mg(X^{nm,\tilde{\sigma}}_{\frac{m(i-1)+k-1}{nm}})\delta W_{i\tilde{\sigma}(k)}\\
	+\sum_{k=1}^m\mathbb{H}(X^{nm,\tilde{\sigma}}_{\frac{m(i-1)+k-1}{nm}})\blackdiamond(\delta W_{i\tilde{\sigma}(k)}\delta W_{i\tilde{\sigma}(k)}^{\top}-{\Iq}\frac{\Delta_n}{m}).
\end{multline}
This last equation can be rewritten as follows
\begin{align}\label{eq:2.15}
	\nonumber X^{nm,\tilde{\sigma}}_{\frac{i}{n}}&-X^{nm,\tilde{\sigma}}_{\frac{i-1}{n}}=f(X^{nm,\tilde{\sigma}}_{\frac{i-1}{n}})\Delta_n+g(X^{nm,\tilde{\sigma}}_{\frac{i-1}{n}})\Delta W_{i}+\mathbb{H}(X^{nm,\tilde{\sigma}}_{\frac{i-1}{n}})\blackdiamond(\Delta W_{i}\Delta W_{i}^{\top}-{\Iq}\Delta_n)\\\nonumber&+\sum_{k=1}^m \left[f(X^{nm,\tilde{\sigma}}_{\frac{m(i-1)+k-1}{nm}})-f(X^{nm,\tilde{\sigma}}_{\frac{i-1}{n}})\right]\frac{\Delta_n}{m}+\sum_{k=1}^m\left[g(X^{nm,\tilde{\sigma}}_{\frac{m(i-1)+k-1}{nm}})-g(X^{nm,\tilde{\sigma}}_{\frac{i-1}{n}})\right]\delta W_{i\tilde{\sigma}(k)}\\
	&+\sum_{k=1}^m\mathbb{H}(X^{nm,\tilde{\sigma}}_{\frac{m(i-1)+k-1}{nm}})\blackdiamond(\delta W_{i\tilde{\sigma}(k)}\delta W_{i\tilde{\sigma}(k)}^{\top}-{\Iq}\frac{\Delta_n}{m})-\mathbb{H}(X^{nm,\tilde{\sigma}}_{\frac{i-1}{n}})\blackdiamond(\Delta W_{i}\Delta W_{i}^{\top}-{\Iq}\Delta_n).
\end{align}
Let us start dealing with the last four  terms in the right-hand side (r.h.s.) of the above equality. By a Taylor expansion,  we have for any fixed index component  $\ell\in\{1,\dots,d\}$,
\begin{multline}\label{eq:xi1n}
	f_{\ell}(X^{nm,\tilde{\sigma}}_{\frac{m(i-1)+k-1}{nm}})-f_{\ell}(X^{nm,\tilde{\sigma}}_{\frac{i-1}{n}})
	=	{\nabla f_{\ell}}(X^{nm,\tilde{\sigma}}_{\frac{i-1}{n}})^{\top}(X^{nm,\tilde{\sigma}}_{\frac{m(i-1)+k-1}{nm}}-X^{nm,\tilde{\sigma}}_{\frac{i-1}{n}})\\+\frac{1}{2}(X^{nm,\tilde{\sigma}}_{\frac{m(i-1)+k-1}{nm}}-X^{nm,\tilde{\sigma}}_{\frac{i-1}{n}})^\top{\nabla^2}{f_{\ell}}(\xi^{1,n}_{ik})(X^{nm,\tilde{\sigma}}_{\frac{m(i-1)+k-1}{nm}}-X^{nm,\tilde{\sigma}}_{\frac{i-1}{n}}),
\end{multline}	
for some vector point $\xi^{1,n}_{ik}$  lying between $X^{ nm,\tilde{\sigma}}_{\frac{m(i-1)+k-1}{nm}}$ and $X^{nm,\tilde{\sigma}}_{\frac{i-1}{n}}$.
Then, using \eqref{eq:iter1}
\begin{equation}\label{eq:decf}
\sum_{k=1}^m \left[f(X^{nm,\tilde{\sigma}}_{\frac{m(i-1)+k-1}{nm}})-f(X^{nm,\tilde{\sigma}}_{\frac{i-1}{n}})\right]\frac{\Delta_n}{m}=:
M^{nm,\tilde{\sigma},1}_{\frac{i-1}{n}}+N^{nm,\tilde{\sigma}}_{\frac{i-1}{n}},
\end{equation}
where for $\ell\in\{1,\dots,d\}$ the $\ell^{ th}$-component of  $N^{nm,\tilde{\sigma}}_{\frac{i-1}{n}}$ and $M^{nm,\tilde{\sigma},1}_{\frac{i-1}{n}}$  are  given by 
\begin{align}\label{eq:N}
N^{nm,\tilde{\sigma}}_{\ell,\frac{i-1}{n}}&=\sum_{k=2}^m\nabla f^{\top}_{\ell}(X^{nm,\tilde{\sigma}}_{\frac{i-1}{n}})\\\times&\sum_{k'=1}^{k-1}\left[f(X^{nm,\tilde{\sigma}}_{\frac{m(i-1)+k'-1}{nm}})\frac{\Delta_n}{m}+\mathbb{H}(X^{nm,\tilde{\sigma}}_{\frac{m(i-1)+k'-1}{nm}})\blackdiamond(\delta W_{i\tilde{\sigma}(k')}\delta W_{i\tilde{\sigma}(k')}^{\top}-{\Iq} \frac{\Delta_n}{m})\right]\frac{\Delta_n}{m}\nonumber\\
	&+\frac{1}{2}\sum_{k=2}^m (X^{nm,\tilde{\sigma}}_{\frac{m(i-1)+k-1}{nm}}-X^{nm,\tilde{\sigma}}_{\frac{i-1}{n}})^{\top}\nabla^2f_{\ell}(\xi^{1,n}_{ik})(X^{nm,\tilde{\sigma}}_{\frac{m(i-1)+k-1}{nm}}-X^{nm,\tilde{\sigma}}_{\frac{i-1}{n}})\frac{\Delta_n}{m}\nonumber
\\\label{eq:M1}
M^{nm,\tilde{\sigma},1}_{\ell,\frac{i-1}{n}}&=\sum_{k=2}^m{\nabla f_{\ell}}(X^{nm,\tilde{\sigma}}_{\frac{i-1}{n}})^\top\sum_{k'=1}^{k-1}g(X^{nm,\tilde{\sigma}}_{\frac{m(i-1)+k'-1}{nm}})\delta W_{i\tilde{\sigma}(k')}\frac{\Delta_n}{m}.
\end{align}	
For the last two terms in the r.h.s. of  $\eqref{eq:2.15}$ by $\Delta W_{i}\Delta W_{i}^{\top}=\sum_{k,k'=1}^m\delta W_{i,\tilde{\sigma}(k)}\delta W_{i\tilde{\sigma}(k')}^{\top},$ we get
\begin{align}\label{eq:2.17}
	\nonumber&\sum_{k=1}^m\mathbb{H}(X^{nm,\tilde{\sigma}}_{\frac{m(i-1)+k-1}{nm}})\blackdiamond(\delta W_{i\tilde{\sigma}(k)}\delta W_{i\tilde{\sigma}(k)}^{\top}-{\Iq}\frac{\Delta_n}{m})-\mathbb{H}(X^{nm,\tilde{\sigma}}_{\frac{i-1}{n}})\blackdiamond(\Delta W_{i}\Delta W_{i}^{\top}-{\Iq}{\Delta_n})\\
	=&\sum_{k=2}^m\left[\mathbb{H}(X^{nm,\tilde{\sigma}}_{\frac{m(i-1)+k-1}{nm}})-\mathbb{H}(X^{nm,\tilde{\sigma}}_{\frac{i-1}{n}})\right]\blackdiamond(\delta W_{i\tilde{\sigma}(k)}\delta W_{i\tilde{\sigma}(k)}^{\top}-{\Iq} \frac{\Delta_n}{m})\\\nonumber-&2\hspace{-0.4cm}\sum_{\substack{1\le k'<k\le m}}\hspace{-0.4cm}\mathbb{H}(X^{nm,\tilde{\sigma}}_{\frac{i-1}{n}})\blackdiamond\delta W_{i\tilde{\sigma}(k)}\delta W_{i\tilde{\sigma}(k')}^{\top}-\hspace{-0.4cm}\sum_{\substack{1\le k<k'\le m}}\hspace{-0.4cm}\mathbb{H}(X^{nm,\tilde{\sigma}}_{\frac{i-1}{n}})\blackdiamond(\delta W_{i\tilde{\sigma}(k)}\delta W_{i\tilde{\sigma}(k')}^{\top}-\delta W_{i\tilde{\sigma}(k')}\delta W_{i\tilde{\sigma}(k)}^{\top}).
\end{align}
From $\eqref{eq:2.15}$ and $\eqref{eq:2.17}$, if we denote 
\begin{align*}
	M^{nm,\tilde{\sigma},2}_{\frac{i-1}{n}}=&\sum_{k=2}^m\left[g(X^{nm,\tilde{\sigma}}_{\frac{m(i-1)+k-1}{nm}})-g(X^{nm,\tilde{\sigma}}_{\frac{i-1}{n}})\right]\delta W_{i\tilde{\sigma}(k)}-2\hspace{-0.4cm}\sum_{\substack{1\le k<k'\le m}}\hspace{-0.4cm}\mathbb{H}(X^{nm,\tilde{\sigma}}_{\frac{i-1}{n}})\blackdiamond\delta W_{i\tilde{\sigma}(k')}\delta W_{i\tilde{\sigma}(k)}^{\top},\\
	M^{nm,\tilde{\sigma},3}_{\frac{i-1}{n}}=&\sum_{k=2}^m\left[\mathbb{H}(X^{nm,\tilde{\sigma}}_{\frac{m(i-1)+k-1}{nm}})-\mathbb{H}(X^{nm,\tilde{\sigma}}_{\frac{i-1}{n}})\right]\blackdiamond(\delta W_{i\tilde{\sigma}(k)}\delta W_{i\tilde{\sigma}(k)}^{\top}-{\Iq} \frac{\Delta_n}{m}).
\end{align*} 
Let us set 
\begin{equation}\label{eq:decM}
M^{nm,\tilde{\sigma}}=M^{nm,\tilde{\sigma},1}+M^{nm,\tilde{\sigma},2}+M^{nm,\tilde{\sigma},3}.
\end{equation}
Then combining \eqref{eq:2.15}, \eqref{eq:decf} and \eqref{eq:decM} we obtain the first assertion of the following lemma. The proof of the remaining results are postponed to appendix \ref{app:A}
\begin{lemma}\label{finer}
	The difference equation for $X^{nm,\tilde{\sigma}}_{\frac{i}{n}}-X^{nm,\tilde{\sigma}}_{\frac{i-1}{n}}$, $i\in\{1,\dots,n\}$ is given by
	\begin{multline}\label{eq:2.6.3}
		X^{nm,\tilde{\sigma}}_{\frac{i}{n}}-X^{nm,\tilde{\sigma}}_{\frac{i-1}{n}}=f(X^{nm,\tilde{\sigma}}_{\frac{i-1}{n}})\Delta_n+g(X^{nm,\tilde{\sigma}}_{\frac{i-1}{n}})\Delta W_{i}+\mathbb{H}(X^{nm,\tilde{\sigma}}_{\frac{i-1}{n}})\blackdiamond(\Delta W_{i}\Delta W_{i}^{\top}-{\Iq}\Delta_n)\\-\mathbb{H}(X^{nm,\tilde{\sigma}}_{\frac{i-1}{n}})\blackdiamond\hspace{-0.5cm}\sum_{\substack{1\le k<k'\le m}}\hspace{-0.5cm}(\delta W_{i\tilde{\sigma}(k)}\delta W_{i\tilde{\sigma}(k')}^{\top}-\delta W_{i\tilde{\sigma}(k')}\delta W_{i\tilde{\sigma}(k)}^{\top})+M^{nm,\tilde{\sigma}}_{\frac{i-1}{n}}+N^{nm,\tilde{\sigma}}_{\frac{i-1}{n}},
	\end{multline}
	where $\mathbb{E}(M^{nm,\tilde{\sigma}}_{\frac{i-1}{n}}|\mathcal{F}_{\frac{i-1}{n}})=0$, and for any integer $p\geq2$ there exists a constant $K_p$ such that
	\begin{align}\label{eq:2.6.3a}
		\max_{0\leq i\leq n}\mathbb{E}(| M^{nm,\tilde{\sigma}}_{\frac{i-1}{n}}|^p)&\leq K_p\Delta_n^{3p/2},\\
\label{eq:2.6.3b}
		\max_{0\leq i\leq n}\mathbb{E}(| N^{nm,\tilde{\sigma}}_{\frac{i-1}{n}}|^p)&\leq K_p\Delta_n^{2p}.
	\end{align}
\end{lemma}
In what follows, we give further expansion studies for the terms  $N^{nm,\tilde{\sigma}}_{\frac{i-1}{n}}$ and $M^{nm,\tilde{\sigma},1}_{\frac{i-1}{n}}$, $M^{nm,\tilde{\sigma},2}$  and $M^{nm,\tilde{\sigma},3}$ defined above. 
\paragraph{\bf $\bullet$ The term $N^{nm,\tilde{\sigma}}_{\frac{i-1}{n}}$:} Starting from relation  \eqref{eq:N} we replace the increment $X^{nm,\tilde{\sigma}}_{\frac{m(i-1)+k-1}{nm}}-X^{nm,\tilde{\sigma}}_{\frac{i-1}{n}} $ using \eqref{eq:iter1} and we only freeze the coefficients of the contributing terms in the asymptotic behavior of the error at the limit point $X_{\frac{i-1}{n}}$. Then thanks to  \eqref{prop:bd} and using that
\begin{align*}
	\sum_{k=2}^m\sum_{k'=1}^{k-1}\delta W_{i\tilde{\sigma}(k')}\delta W_{i\tilde{\sigma}(k')}^{\top}&=\sum_{k=1}^{m-1}(m-k)\delta W_{i\tilde{\sigma}(k)}\delta W_{i\tilde{\sigma}(k)}^{\top},
\end{align*}
we get  the following result.
\begin{lemma}\label{lem:N}
 For $\ell\in\{1,\dots,d\}$ the $\ell^{ th}$-component of  $N^{nm,\tilde{\sigma}}_{\frac{i-1}{n}}$ has the following expansion 
 \begin{multline}\label{exp:N}
 N^{nm,\tilde{\sigma}}_{\ell,\frac{i-1}{n}}=\frac{(m-1)}{2m}\nabla f^{\top}_{\ell}(X_{\frac{i-1}{n}})f(X_{\frac{i-1}{n}}){\Delta^2_n}+\\
	\frac{1}{2}\Big[g(X_{\frac{i-1}{n}})^\top\nabla^2f_{\ell}(X_{\frac{i-1}{n}})g(X_{\frac{i-1}{n}})\Big]\blackdiamond \sum_{k=1}^{m-1}(m-k)\delta W_{i\tilde{\sigma}(k)}\delta W_{i\tilde{\sigma}(k)}^{\top}\frac{\Delta_n}{m}+R^{nm,\tilde{\sigma}}_{\ell,\frac{i-1}{n}}(0)+{\tilde R}^{nm,\tilde{\sigma}}_{\ell,\frac{i-1}{n}}(0),
\end{multline}
where 
 \begin{multline*}
	R^{nm,\tilde{\sigma}}_{\ell,\frac{i-1}{n}}(0)=\frac{\Delta_n}{m}\hspace{-0.3cm}\sum_{\substack{1\le k'<k\le m}}\hspace{-0.3cm}\nabla f^{\top}_{\ell}(X^{nm,\tilde{\sigma}}_{\frac{i-1}{n}})\mathbb{H}(X^{nm,\tilde{\sigma}}_{\frac{m(i-1)+k'-1}{nm}})\blackdiamond(\delta W_{i\tilde{\sigma}(k')}\delta W_{i\tilde{\sigma}(k')}^{\top}-{\Iq} \frac{\Delta_n}{m})\\+\Big[g(X_{\frac{i-1}{n}})^\top\nabla^2f_{\ell}(X_{\frac{i-1}{n}})g(X_{\frac{i-1}{n}})\Big]\blackdiamond\hspace{-0.5cm} \sum_{\substack{1\le k''<k'<k\le m}}\hspace{-0.5cm}\delta W_{i\tilde{\sigma}(k')}\delta W_{i\tilde{\sigma}(k'')}^{\top}\frac{\Delta_n}{m}
\end{multline*} 
satisfies $\mathbb{E}(R^{nm,\tilde{\sigma}}_{\ell,\frac{i-1}{n}}(0)|\mathcal{F}_{\frac{i-1}{n}})=0$. Moreover, for any integer $p\geq2$ there exists 
\begin{align}\label{eq:R0}
&\max_{0\leq i\leq n}\mathbb{E}(| R^{nm,\tilde{\sigma}}_{\ell,\frac{i-1}{n}}(0)|^p)=o\Big(\Delta_n^{3p/2}\Big),\\\label{eq:tilR0}
&\max_{0\leq i\leq n}\mathbb{E}(| {\tilde R}^{nm,\tilde{\sigma}}_{\ell,\frac{i-1}{n}}(0)|^p)=o\Big(\Delta_n^{2p}\Big).
\end{align}
\end{lemma}
The proof of the above lemma is postponed to appendix \ref{app:A}.
\paragraph{\bf $\bullet$ The term $M^{nm,\tilde{\sigma},1}_{\frac{i-1}{n}}$:} For this term we only need to freeze the coefficients in relation  \eqref{eq:M1} at the limit point $X_{\frac{i-1}{n}}$. Then using 
\begin{align*}
	\sum_{k=2}^m\sum_{k'=1}^{k-1}\delta W^j_{i\tilde{\sigma}(k')}&=\sum_{k=1}^{m-1}(m-k)\delta W^{j}_{i\tilde{\sigma}(k)},
	\end{align*} 
we get the following result.	
\begin{lemma}\label{lem:M1}
 For $\ell\in\{1,\dots,d\}$ the $\ell^{ th}$-component of the term $M^{nm,\tilde{\sigma},1}_{\frac{i-1}{n}}$ has the following expansion
 \begin{equation}\label{exp:M1}
  M^{nm,\tilde{\sigma},1}_{\ell,\frac{i-1}{n}}=\Big[{\nabla f_{\ell}}(X_{\frac{i-1}{n}})^\top g(X_{\frac{i-1}{n}})\Big]\sum_{k=1}^{m-1}(m-k)\delta W_{i\tilde{\sigma}(k)}\frac{\Delta_n}{m}+R^{nm,\tilde{\sigma}}_{\ell,\frac{i-1}{n}}(1),
 \end{equation}
 with $\mathbb{E}(R^{nm,\tilde{\sigma}}_{\ell,\frac{i-1}{n}}(1)|\mathcal{F}_{\frac{i-1}{n}})=0$. Moreover, for any integer $p\geq2$ there exists 
\begin{align}\label{eq:R1}
\max_{0\leq i\leq n}\mathbb{E}(| R^{nm,\tilde{\sigma}}_{\ell,\frac{i-1}{n}}(1)|^p)=o\Big(\Delta_n^{3p/2}\Big).
\end{align}
\end{lemma}
The proof of the above lemma is postponed to the appendix \ref{app:A}.
\paragraph{\bf $\bullet$ The term $M^{nm,\tilde{\sigma},2}_{\frac{i-1}{n}}$:}
For this term we first proceed similarly as in \eqref{eq:decf}  and we use a Taylor expansion to write for $\ell,j\in\{1,\dots,q\}$ 
\begin{multline}\label{eq:2.16}
	g_{\ell j}(X^{nm,\tilde{\sigma}}_{\frac{m(i-1)+k-1}{nm}})-g_{\ell j}(X^{nm,\tilde{\sigma}}_{\frac{i-1}{n}})
	=	\nabla g_{\ell j}(X^{nm,\tilde{\sigma}}_{\frac{i-1}{n}})^{\top}(X^{nm,\tilde{\sigma}}_{\frac{m(i-1)+k-1}{nm}}-X^{nm,\tilde{\sigma}}_{\frac{i-1}{n}})\\+\frac{1}{2}(X^{nm,\tilde{\sigma}}_{\frac{m(i-1)+k-1}{nm}}-X^{nm,\tilde{\sigma}}_{\frac{i-1}{n}})^{\top}\nabla^2g_{\ell j}(\xi^{2,n}_{ik})(X^{nm,\tilde{\sigma}}_{\frac{m(i-1)+k-1}{nm}}-X^{nm,\tilde{\sigma}}_{\frac{i-1}{n}}),
\end{multline}	
for some vector point $\xi^{2,n}_{ik}$ lying between $X^{nm,\tilde{\sigma}}_{\frac{m(i-1)+k-1}{nm}}$ and $X^{nm,\tilde{\sigma}}_{\frac{i-1}{n}}$.  Once again by \eqref{eq:iter1} we get
\begin{align*}
	g_{\ell  j}(X^{nm,\tilde{\sigma}}_{\frac{m(i-1)+k-1}{nm}})&-g_{\ell j}(X^{nm,\tilde{\sigma}}_{\frac{i-1}{n}})
=\nabla g_{\ell  j}(X^{nm,\tilde{\sigma}}_{\frac{i-1}{n}})^{\top}\Bigg[\sum_{k'=1}^{k-1}\bigg(f(X^{nm,\tilde{\sigma}}_{\frac{m(i-1)+k'-1}{nm}})\frac{\Delta_n}{m}\\&+g(X^{nm,\tilde{\sigma}}_{\frac{m(i-1)+k'-1}{nm}})\delta W_{i\tilde{\sigma}(k')}+	\mathbb{H}(X^{nm,\tilde{\sigma}}_{\frac{m(i-1)+k'-1}{nm}})\blackdiamond(\delta W_{i\tilde{\sigma}(k')}\delta W_{i\tilde{\sigma}(k')}^{\top}-{\Iq} \frac{\Delta_n}{m})\bigg)\Bigg]
	\\&+\frac{1}{2}(X^{nm,\tilde{\sigma}}_{\frac{m(i-1)+k-1}{nm}}-X^{nm,\tilde{\sigma}}_{\frac{i-1}{n}})^{\top}\nabla^2g_{\ell j}(\xi^{2,n}_{ik})(X^{nm,\tilde{\sigma}}_{\frac{m(i-1)+k-1}{nm}}-X^{nm,\tilde{\sigma}}_{\frac{i-1}{n}}).
\end{align*}
Then we have
\begin{align*}
	M^{nm,\tilde{\sigma},2}_{\ell,\frac{i-1}{n}}=	&\sum_{k=2}^m\sum_{j=1}^q\Bigg[\nabla g_{\ell j}(X^{nm,\tilde{\sigma}}_{\frac{i-1}{n}})^{\top}\sum_{k'=1}^{k-1}\bigg(f(X^{nm,\tilde{\sigma}}_{\frac{m(i-1)+k'-1}{nm}})\frac{\Delta_n}{m}+g(X^{nm,\tilde{\sigma}}_{\frac{m(i-1)+k'-1}{nm}})\delta W_{i\tilde{\sigma}(k')}\\&+\mathbb{H}(X^{nm,\tilde{\sigma}}_{\frac{m(i-1)+k'-1}{nm}})\blackdiamond(\delta W_{i\tilde{\sigma}(k')}\delta W_{i\tilde{\sigma}(k')}^{\top}-{\Iq}\frac{\Delta_n}{m})\bigg)-2\sum_{k'=1}^{k-1}\sum_{j'=1}^qh_{\ell jj'}(X^{nm,\tilde{\sigma}}_{\frac{i-1}{n}})\delta W_{i\tilde{\sigma}(k')}^{j'}\\&+\frac{1}{2}(X^{nm,\tilde{\sigma}}_{\frac{m(i-1)+k-1}{nm}}-X^{nm,\tilde{\sigma}}_{\frac{i-1}{n}})^{\top}\nabla^2g_{\ell j}(\xi^{2,n}_{ik})(X^{nm,\tilde{\sigma}}_{\frac{m(i-1)+k-1}{nm}}-X^{nm,\tilde{\sigma}}_{\frac{i-1}{n}})\Bigg]\delta W_{i\tilde{\sigma}(k)}^{j}.
\end{align*}
Recalling that $h_{\ell jj'}=\frac{1}{2}\nabla g_{\ell j}^{\top}g_{\bullet j'}$ we obtain
\begin{align*}
	&M^{nm,\tilde{\sigma},2}_{\ell,\frac{i-1}{n}}=\sum_{k=2}^m\sum_{j=1}^q\Bigg[\nabla g_{\ell j}(X^{nm,\tilde{\sigma}}_{\frac{i-1}{n}})^{\top}\sum_{k'=1}^{k-1}\bigg(f(X^{nm,\tilde{\sigma}}_{\frac{m(i-1)+k'-1}{nm}})\frac{\Delta_n}{m}+\\&\Big[g(X^{nm,\tilde{\sigma}}_{\frac{m(i-1)+k'-1}{nm}})-g(X^{nm,\tilde{\sigma}}_{\frac{i-1}{n}})\Big]\delta W_{i\tilde{\sigma}(k')}+\mathbb{H}(X^{nm,\tilde{\sigma}}_{\frac{m(i-1)+k'-1}{nm}})\blackdiamond(\delta W_{i\tilde{\sigma}(k')}\delta W_{i\tilde{\sigma}(k')}^{\top}-{\Iq}\frac{\Delta_n}{m})\bigg)\\&+\frac{1}{2}(X^{nm,\tilde{\sigma}}_{\frac{m(i-1)+k-1}{nm}}-X^{nm,\tilde{\sigma}}_{\frac{i-1}{n}})^{\top}\nabla^2g_{\ell j}(\xi^{2,n}_{ik})(X^{nm,\tilde{\sigma}}_{\frac{m(i-1)+k-1}{nm}}-X^{nm,\tilde{\sigma}}_{\frac{i-1}{n}})\Bigg]\delta W_{i\tilde{\sigma}(k)}^j.
\end{align*}
Again by applying Taylor expansion for each component of the matrix function $g$, we get
$
g(X^{nm,\tilde{\sigma}}_{\frac{m(i-1)+k'-1}{nm}})-g(X^{nm,\tilde{\sigma}}_{\frac{i-1}{n}})={\dot g}^n_{ik'}\blackdiamond(X^{nm,\tilde{\sigma}}_{\frac{m(i-1)+k'-1}{nm}}-X^{nm,\tilde{\sigma}}_{\frac{i-1}{n}})\in\R^{d\times q}$,
where  ${\dot g}^n_{ik'}\in (\R^{d\times 1})^{d\times q}$ is a block matrix such that for $\ell'\in\{1,\dots,d\},j'\in\{1,\dots,q\}$, the $\ell'j'$-th block is given by $({\dot g}^n_{ik'})_{\ell' j'}=\nabla g_{\ell' j'}( \xi'^{2,n}_{ik'})\in\R^{d\times 1}$ where  $ \xi'^{2,n}_{ik'}$ is a vector point lying between $X^{nm,\tilde{\sigma}}_{\frac{m(i-1)+k'-1}{nm}}$ and $X^{nm,\tilde{\sigma}}_{\frac{i-1}{n}}$.
Then, we have
\begin{align}\label{eq:M2}
	&M^{nm,\tilde{\sigma},2}_{\ell,\frac{i-1}{n}}=\sum_{k=2}^m\sum_{j=1}^q\Bigg[\nabla g_{\ell j}^{\top}(X^{nm,\tilde{\sigma}}_{\frac{i-1}{n}})\sum_{k'=1}^{k-1}\bigg(f(X^{nm,\tilde{\sigma}}_{\frac{m(i-1)+k'-1}{nm}})\frac{\Delta_n}{m}+\nonumber\\&\mathbb{H}(X^{nm,\tilde{\sigma}}_{\frac{m(i-1)+k'-1}{nm}})\blackdiamond(\delta W_{i\tilde{\sigma}(k')}\delta W_{i\tilde{\sigma}(k')}^{\top}-{\Iq}\frac{\Delta_n}{m})\bigg)\Bigg]\delta W_{i\tilde{\sigma}(k)}^j\nonumber\\&+\sum_{k=3}^m\sum_{j=1}^q\nabla g_{\ell j}^{\top}(X^{nm,\tilde{\sigma}}_{\frac{i-1}{n}})\sum_{k'=2}^{k-1}\Big[ {\dot g}^n_{ik'}\blackdiamond(X^{nm,\tilde{\sigma}}_{\frac{m(i-1)+k'-1}{nm}}-X^{nm,\tilde{\sigma}}_{\frac{i-1}{n}})\Big]\delta W_{i\tilde{\sigma}(k')}\delta W_{i\tilde{\sigma}(k)}^j\nonumber\\&+\sum_{k=2}^m\sum_{j=1}^q\frac{1}{2}(X^{nm,\tilde{\sigma}}_{\frac{m(i-1)+k-1}{nm}}-X^{nm,\tilde{\sigma}}_{\frac{i-1}{n}})^{\top}\nabla^2g_{\ell j}(\xi^{2,n}_{ik})(X^{nm,\tilde{\sigma}}_{\frac{m(i-1)+k-1}{nm}}-X^{nm,\tilde{\sigma}}_{\frac{i-1}{n}})\delta W_{i\tilde{\sigma}(k)}^j.
\end{align}
Now, we replace the increment $X^{nm,\tilde{\sigma}}_{\frac{m(i-1)+k-1}{nm}}-X^{nm,\tilde{\sigma}}_{\frac{i-1}{n}} $ using \eqref{eq:iter1} and we only freeze   the coefficients of the contributing terms in the asymptotic behavior of the error at the limit point $X_{\frac{i-1}{n}}$.
\begin{lemma}\label{lem:M2}
 For $\ell\in\{1,\dots,d\}$ the $\ell^{ th}$-component of the term $M^{nm,\tilde{\sigma},2}_{\frac{i-1}{n}}$ has the following expansion
 \begin{align}\label{exp:M2}
  &M^{nm,\tilde{\sigma},2}_{\ell,\frac{i-1}{n}}=\notag\\&\sum_{j=1}^q\nabla g_{\ell j}^{\top}(X_{\frac{i-1}{n}})\hspace{-0.5cm}\sum_{\substack{1\le k'<k\le m}}\hspace{-0.3cm}\bigg(f(X_{\frac{i-1}{n}})\frac{\Delta_n}{m}+\mathbb{H}(X_{\frac{i-1}{n}})\blackdiamond(\delta W_{i\tilde{\sigma}(k')}\delta W_{i\tilde{\sigma}(k')}^{\top}-{\Iq}\frac{\Delta_n}{m})\bigg)\delta W_{i\tilde{\sigma}(k)}^j\nonumber\\&+\sum_{j=1}^q \nabla g_{\ell j}^{\top}(X_{\frac{i-1}{n}}) \hspace{-0.5cm}\sum_{\substack{1\le k''< k'<k\le m}}\hspace{-0.3cm}\Big[ {\dot g}^n_{i}
  \blackdiamond\bigg(g(X_{\frac{i-1}{n}})\delta W_{i\tilde{\sigma}(k'')}\bigg)\Big]\delta W_{i\tilde{\sigma}(k')}\delta W_{i\tilde{\sigma}(k)}^j
  \nonumber\\&
+\sum_{j=1}^q\frac{1}{2}\Big[g(X_{\frac{i-1}{n}})^\top\nabla^2g_{\ell j}(X_{\frac{i-1}{n}})g(X_{\frac{i-1}{n}})\Big]\blackdiamond\hspace{-0.5cm}\sum_{\substack{1\le k''< k'<k\le m}}\hspace{-0.5cm}\delta W_{i\tilde{\sigma}(k'')}\delta W_{i\tilde{\sigma}(k')}^{\top}\delta W_{i\tilde{\sigma}(k)}^j+R^{nm,\tilde{\sigma}}_{\ell,\frac{i-1}{n}}(2)
 \end{align}
with $\mathbb{E}(R^{nm,\tilde{\sigma}}_{\ell,\frac{i-1}{n}}(2)|\mathcal{F}_{\frac{i-1}{n}})=0$ and  $\dot{g}^n_i\in(\R^{d\times1})^{d\times q}$ is a block matrix such that for $\ell\in\{1,\dots,d\}$, $j\in\{1,\dots,q\}$, the $\ell j$-th block is given by $(\dot{g}^n_i)_{\ell j}=\nabla g_{\ell j}(X_{\frac{i-1}{n}})$. Moreover,  for any integer $p\geq2$ 
\begin{align}\label{eq:R2}
\max_{0\leq i\leq n}\mathbb{E}(| R^{nm,\tilde{\sigma}}_{\ell,\frac{i-1}{n}}(2)|^p)=o\Big(\Delta_n^{3p/2}\Big).
\end{align}
\end{lemma}
The proof of the above lemma is postponed to the appendix \ref{app:A}.
\begin{remark}
The $\ell$-th component of $M^{nm,\tilde{\sigma},2}$ can be rewritten as follows:
\begin{align}\label{exp:M2a}
&M^{nm,\tilde{\sigma},2}_{\ell,\frac{i-1}{n}}=\notag\\&\sum_{j=1}^q\nabla g_{\ell j}^{\top}(X_{\frac{i-1}{n}})\hspace{-0.4cm}\sum_{\substack{1\le k'<k\le m}}\bigg[f(X_{\frac{i-1}{n}})\frac{\Delta_n}{m}+\mathbb{H}(X_{\frac{i-1}{n}})\blackdiamond(\delta W_{i\tilde{\sigma}(k')}\delta W_{i\tilde{\sigma}(k')}^{\top}-{\Iq}\frac{\Delta_n}{m})\bigg]\delta W_{i\tilde{\sigma}(k)}^j\nonumber\\&+\sum_{j=1}^q \nabla g_{\ell j}^{\top}(X_{\frac{i-1}{n}}) \hspace{-0.5cm}\sum_{\substack{1\le k''< k'<k\le m}}\hspace{-0.5cm}\Big[ {\dot g}^n_{i}
\blackdiamond\bigg(g(X_{\frac{i-1}{n}})\delta W_{i\tilde{\sigma}(k'')}\bigg)\Big]\delta W_{i\tilde{\sigma}(k')}\delta W_{i\tilde{\sigma}(k)}^j
\nonumber\\&
+\sum_{j=1}^q\frac{1}{2}\Big[g(X_{\frac{i-1}{n}})^\top\nabla^2g_{\ell j}(X_{\frac{i-1}{n}})g(X_{\frac{i-1}{n}})\Big]\blackdiamond\hspace{-0.4cm}\sum_{\substack{1\le  k'<k\le m}}\hspace{-0.4cm}(\delta W_{i\tilde{\sigma}(k')}\delta W_{i\tilde{\sigma}(k')}^{\top}-{\Iq}\frac{\Delta_n}{m})\delta W_{i\tilde{\sigma}(k)}^j\nonumber\\&+\sum_{j=1}^q\frac{1}{2}\Big[g(X_{\frac{i-1}{n}})^\top\nabla^2g_{\ell j}(X_{\frac{i-1}{n}})g(X_{\frac{i-1}{n}})\Big]\blackdiamond\sum_{k=2}^m(k-1){\Iq}\frac{\Delta_n}{m}\delta W_{i\tilde{\sigma}(k)}^j\nonumber\\&+\sum_{j=1}^q\Big[g(X_{\frac{i-1}{n}})^\top\nabla^2g_{\ell j}(X_{\frac{i-1}{n}})g(X_{\frac{i-1}{n}})\Big]\blackdiamond\hspace{-0.7cm}\sum_{\substack{1\le k''< k'<k\le m}}\hspace{-0.5cm}\delta W_{i\tilde{\sigma}(k')}\delta W_{i\tilde{\sigma}(k'')}^{\top}\delta W_{i\tilde{\sigma}(k)}^j+R^{nm,\tilde{\sigma}}_{\ell,\frac{i-1}{n}}(2).
\end{align}
\end{remark}

\paragraph{\bf $\bullet$ The term $M^{nm,\tilde{\sigma},3}_{\frac{i-1}{n}}$:}

Considering each component  of $M^{nm,\tilde{\sigma},3}$,   for $\ell\in\{1,\dots,d\}$ we can also consider a Taylor expansion for the components of the matrix $h_{\ell \bullet\bullet}\in\mathbb R^{q\times q}$ to get 
\begin{align}\label{eq:M3}	M^{nm,\tilde{\sigma},3}_{\ell,\frac{i-1}{n}}=&\sum_{k=2}^m\Big[ {\dot  h}_{\ell \bullet\bullet}^{n,ik}\blackdiamond(X^{nm,\tilde{\sigma}}_{\frac{m(i-1)+k-1}{nm}}-X^{nm,\tilde{\sigma}}_{\frac{i-1}{n}})\Big]\blackdiamond(\delta W_{i\tilde{\sigma}(k)}\delta W_{i\tilde{\sigma}(k)}^{\top}-{\Iq} \frac{\Delta_n}{m}),
\end{align}
where the  ${\dot  h}_{\ell \bullet\bullet}^{n,ik}\in(\R^{d\times 1})^{q\times q}$ is a random block matrix such that for $j$ and $j'\in\{1,\dots,q\}$, the $jj'$-th block is given by  $({\dot  h}_{\ell \bullet\bullet}^{n,ik})_{j j'}=\nabla h_{\ell jj'}( \xi^{3,n}_{ik})\in \R^{d\times 1}$ and  $ \xi^{3,n}_{ik}$ is a vector point lying between $X^{nm,\tilde{\sigma}}_{\frac{m(i-1)+k-1}{nm}}$ and $X^{nm,\tilde{\sigma}}_{\frac{i-1}{n}}$.
\begin{remark}
Concerning $M^{nm,\tilde{\sigma},3}$, the last formula can be written differently. In fact, as $\mathbb H \in(\R^{q\times q})^{d\times 1}=\mathbb R^{dq\times q}$, we proceed similarly as above using a Taylor expansion to get the existence of a random block matrix $\dot{\mathbb H}^n_{ik}$  such that 
\begin{align*}	M^{nm,\tilde{\sigma},3}_{\frac{i-1}{n}}=\sum_{k=2}^m\Big[\dot{\mathbb H}^n_{ik}\blackdiamond(X^{nm,\tilde{\sigma}}_{\frac{m(i-1)+k-1}{nm}}-X^{nm,\tilde{\sigma}}_{\frac{i-1}{n}})\Big]\blackdiamond(\delta W_{i\tilde{\sigma}(k)}\delta W_{i\tilde{\sigma}(k)}^{\top}-{\Iq} \frac{\Delta_n}{m}).
\end{align*}
More precisely, we have $\dot{\mathbb H}^n_{ik}\in(\R^{d\times 1})^{dq\times q}$ where for $\ell'\in\{1,\dots,dq\}$ and $j'\in\{1,\dots,q\}$, the $\ell' j'$-th block is given by  $(\dot{\mathbb H}^n_{ik})_{\ell' j'}=\nabla h_{\ell jj'}( \xi^{3,n}_{ik})\in\R^{d\times 1}$ where $\ell'=q(\ell-1)+j$ and  $ \xi^{3,n}_{ik}$ is a vector point lying between $X^{nm,\tilde{\sigma}}_{\frac{m(i-1)+k-1}{nm}}$ and $X^{nm,\tilde{\sigma}}_{\frac{i-1}{n}}$.
\end{remark}
Now, we replace the increment $X^{nm,\tilde{\sigma}}_{\frac{m(i-1)+k-1}{nm}}-X^{nm,\tilde{\sigma}}_{\frac{i-1}{n}} $ using \eqref{eq:iter1} and we only freeze   the coefficients of the contributing terms in the asymptotic behavior of the error at the limit point $X_{\frac{i-1}{n}}$.
\begin{lemma}\label{lem:M3}
	For $\ell\in\{1,\dots,d\}$ the $\ell^{ th}$ component of the term $M^{nm,\tilde{\sigma},3}_{\frac{i-1}{n}}$ has the following expansion
	\begin{align}\label{exp:M3}
		&M^{nm,\tilde{\sigma},3}_{\ell,\frac{i-1}{n}}=\sum_{k=2}^m\sum_{k'=1}^{k-1}\Big[ {\dot  h}_{\ell \bullet\bullet}^{n,i}\blackdiamond(g(X_{\frac{i-1}{n}})\delta W_{i\tilde{\sigma}(k')})\Big]\blackdiamond(\delta W_{i\tilde{\sigma}(k)}\delta W_{i\tilde{\sigma}(k)}^{\top}-{\Iq} \frac{\Delta_n}{m})+R^{nm,\tilde{\sigma}}_{\ell,\frac{i-1}{n}}(3)
	\end{align}
	with $\mathbb{E}(R^{nm,\tilde{\sigma}}_{\ell,\frac{i-1}{n}}(3)|\mathcal{F}_{\frac{i-1}{n}})=0$ and $\dot{h}^{n,i}_{\ell\bullet\bullet}\in(\R^{d\times1})^{q\times q}$ is a block matrix such that for $j$ and $j'\in\{1,\dots,q\}$, the $j j'$-th block is given by $(\dot{h}^{n,i}_{\ell\bullet\bullet})_{j j'}=\nabla h_{\ell j j'}(X_{\frac{i-1}{n}})$. Moreover, for any integer $p\geq2$ there exists 
	\begin{align}\label{eq:R3}
	\max_{0\leq i\leq n}\mathbb{E}(| R^{nm,\tilde{\sigma}}_{\ell,\frac{i-1}{n}}(3)|^p)=o\Big(\Delta_n^{3p/2}\Big).
	\end{align}
\end{lemma}
The proof of the above lemma is postponed to the appendix \ref{app:A} .

\subsection{Expansion of the error $\bar{X}^{nm,\sigma}_{\frac{i}{n}}-\bar{X}^{nm,\sigma}_{\frac{i-1}{n}}$}
We remind that $\bar{X}^{nm,\sigma}= \frac{1}{2}(X^{nm}+X^{nm,\sigma})$. By $\eqref{eq:iter2}$, we have
\begin{align*}
	&\bar{X}^{nm,\sigma}_{\frac{i}{n}}-\bar{X}^{nm,\sigma}_{\frac{i-1}{n}}=\frac{1}{2}\sum_{k=1}^m\left[f(X^{nm}_{\frac{m(i-1)+k-1}{nm}})+f(X^{nm,\sigma}_{\frac{m(i-1)+k-1}{nm}})\right]\frac{\Delta_n}{m}\\&+\frac{1}{2}\sum_{k=1}^m\left[\mathbb{H}(X^{nm}_{\frac{m(i-1)+k-1}{nm}})\blackdiamond(\delta W_{ik}\delta W_{ik}^{\top}-{\Iq}\frac{\Delta_n}{m})+\mathbb{H}(X^{nm,\sigma}_{\frac{m(i-1)+k-1}{nm}})\blackdiamond(\delta W_{i\sigma(k)}\delta W_{i\sigma(k)}^{\top}-{\Iq}\frac{\Delta_n}{m})\right]\\&+\frac{1}{2}\sum_{k=1}^m\left[g(X^{nm}_{\frac{m(i-1)+k-1}{nm}})\delta W_{ik}+g(X^{nm,\sigma}_{\frac{m(i-1)+k-1}{nm}})\delta W_{i\sigma(k)}\right].
\end{align*}
Then we rewrite it as follows
\begin{align*}
	\bar{X}^{nm,\sigma}_{\frac{i}{n}}-\bar{X}^{nm,\sigma}_{\frac{i-1}{n}}=&f(\bar{X}^{nm,\sigma}_{\frac{i-1}{n}})\Delta_n+g(\bar{X}^{nm,\sigma}_{\frac{i-1}{n}})\Delta W_{i}+\mathbb{H}(\bar{X}^{nm,\sigma}_{\frac{i-1}{n}})\blackdiamond(\Delta W_{i}\Delta W_{i}^{\top}-{\Iq}\Delta_n)\\
	&+A_{\frac{i-1}{n}}+B_{\frac{i-1}{n}}+C_{\frac{i-1}{n}},
\end{align*}
where 
\begin{align*}
	A_{\frac{i-1}{n}}=&\frac{1}{2}\sum_{k=1}^m\left[f(X^{nm}_{\frac{m(i-1)+k-1}{nm}})+f(X^{nm,\sigma}_{\frac{m(i-1)+k-1}{nm}})\right]\frac{\Delta_n}{m}-f(\bar{X}^{nm,\sigma}_{\frac{i-1}{n}})\Delta_n,\\
	B_{\frac{i-1}{n}}=&\frac{1}{2}\sum_{k=1}^m\left[g(X^{nm}_{\frac{m(i-1)+k-1}{nm}})\delta W_{i,k}+g(X^{nm,\sigma}_{\frac{m(i-1)+k-1}{nm}})\delta W_{i\sigma(k)}\right]-g(\bar{X}^{nm,\sigma}_{\frac{i-1}{n}})\Delta W_{i},\\
	C_{\frac{i-1}{n}}=&\frac{1}{2}\sum_{k=1}^m\left[\mathbb{H}(X^{nm}_{\frac{m(i-1)+k-1}{nm}})\blackdiamond(\delta W_{ik}\delta W_{ik}^{\top}-{\Iq}\frac{\Delta_n}{m})+\mathbb{H}(X^{nm,\sigma}_{\frac{m(i-1)+k-1}{nm}})\blackdiamond(\delta W_{i\sigma(k)}\delta W_{i\sigma(k)}^{\top}-{\Iq}\frac{\Delta_n}{m})\right]\\
	&-\mathbb{H}(\bar{X}^{nm,\sigma}_{\frac{i-1}{n}})\blackdiamond(\Delta W_{i}\Delta W_{i}^{\top}-{\Iq}\Delta_n).
\end{align*}
Now considering $A_{\frac{i-1}{n}}$, we use $\eqref{eq:decf}$ to get
\begin{align*}
	A_{\frac{i-1}{n}}=&\frac{1}{2}\sum_{k=1}^m\left[f(X^{nm}_{\frac{m(i-1)+k-1}{nm}})-f(X^{nm}_{\frac{i-1}{n}})\right]\frac{\Delta_n}{m}+\frac{1}{2}\sum_{k=1}^m\left[f(X^{nm,\sigma}_{\frac{m(i-1)+k-1}{nm}})-f(X^{nm,\sigma}_{\frac{i-1}{n}})\right]\frac{\Delta_n}{m}\\
	&+\frac{1}{2}\left(f(X^{nm}_{\frac{i-1}{n}})+f(X^{nm,\sigma}_{\frac{i-1}{n}})\right)\Delta_n-f(\bar{X}^{nm,\sigma}_{\frac{i-1}{n}})\Delta_n\\
	=&\frac{1}{2}(M^{nm,\Id,1}_{\frac{i-1}{n}}+M^{nm,\sigma,1}_{\frac{i-1}{n}}+N^{nm,\Id}_{\frac{i-1}{n}}+N^{nm,\sigma}_{\frac{i-1}{n}})+ {\tilde N}^{nm}_{\frac{i-1}{n}},
\end{align*}
where $\displaystyle {\tilde N}^{nm}_{\frac{i-1}{n}}=\frac{1}{2}\big(f(X^{nm}_{\frac{i-1}{n}})+f(X^{nm,\sigma}_{\frac{i-1}{n}})\big)\Delta_n-f(\bar{X}^{nm,\sigma}_{\frac{i-1}{n}})\Delta_n$.
Similarly,  we  have
\begin{align*}
&B_{\frac{i-1}{n}}+C_{\frac{i-1}{n}}=\left[\frac{1}{2}\left(g(X^{nm}_{\frac{i-1}{n}})+g(X^{nm,\sigma}_{\frac{i-1}{n}})\right)-g(\bar{X}^{nm,\sigma}_{\frac{i-1}{n}})\right]\Delta W_{i}
\\&+\frac{1}{2}\sum_{k=1}^m\left[g(X^{nm}_{\frac{m(i-1)+k-1}{nm}})-g(X^{nm}_{\frac{i-1}{n}})\right]\delta W_{ik}+\frac{1}{2}\sum_{k=1}^m\left[g(X^{nm,\sigma}_{\frac{m(i-1)+k-1}{nm}})-g(X^{nm,\sigma}_{\frac{i-1}{n}})\right]\delta W_{i\sigma(k)}
\\&+\frac{1}{2}\sum_{k=1}^m\left[\mathbb{H}(X^{nm}_{\frac{m(i-1)+k-1}{nm}})\blackdiamond(\delta W_{ik}\delta W_{ik}^{\top}-{\Iq}\frac{\Delta_n}{m})-\mathbb{H}(X^{nm}_{\frac{i-1}{n}})\blackdiamond(\Delta W_{i}\Delta W_{i}^{\top}-{\Iq}\Delta_n)\right]
\\
&+\frac{1}{2}\sum_{k=1}^m\left[\mathbb{H}(X^{nm,\sigma}_{\frac{m(i-1)+k-1}{nm}})\blackdiamond(\delta W_{i\sigma(k)}\delta W_{i\sigma(k)}^{\top}-{\Iq}\frac{\Delta_n}{m})-\mathbb{H}(X^{nm,\sigma}_{\frac{i-1}{n}})\blackdiamond(\Delta W_{i}\Delta W_{i}^{\top}-{\Iq}\Delta_n)\right]\\
&+\left[\frac{1}{2}\left(\mathbb{H}(X^{nm}_{\frac{i-1}{n}})+\mathbb{H}(X^{nm,\sigma}_{\frac{i-1}{n}})\right)-\mathbb{H}(\bar{X}^{nm,\sigma}_{\frac{i-1}{n}})\right]\blackdiamond(\Delta W_{i}\Delta W_{i}^{\top}-{\Iq}\Delta_n).
\end{align*}
Now, by  \eqref{eq:2.17} and the expressions of $M^{nm,\tilde{\sigma},2}$ and $M^{nm,\tilde{\sigma},3}$ given above relation \eqref{eq:decM} we rearrange our terms to get
$$
	B_{\frac{i-1}{n}}+C_{\frac{i-1}{n}}=\frac{1}{2}(M^{nm,\Id,2}_{\frac{i-1}{n}}+M^{nm,\sigma,2}_{\frac{i-1}{n}}+M^{nm,\Id,3}_{\frac{i-1}{n}}+M^{nm,\sigma,3}_{\frac{i-1}{n}})+{\tilde M}^{nm,1}_{\frac{i-1}{n}}+{\tilde M}^{nm,2}_{\frac{i-1}{n}}-\frac{1}{2}{\tilde M}^{nm,3}_{\frac{i-1}{n}},
$$
where
\begin{align*}
	{\tilde M}^{nm,1}_{\frac{i-1}{n}}=&\left[\frac{1}{2}\left(g(X^{nm}_{\frac{i-1}{n}})+g(X^{nm,\sigma}_{\frac{i-1}{n}})\right)-g(\bar{X}^{nm,\sigma}_{\frac{i-1}{n}})\right]\Delta W_{i},\\
	{\tilde M}^{nm,2}_{\frac{i-1}{n}}=&\left[\frac{1}{2}\left(\mathbb{H}(X^{nm}_{\frac{i-1}{n}})+\mathbb{H}(X^{nm,\sigma}_{\frac{i-1}{n}})\right)-\mathbb{H}(\bar{X}^{nm,\sigma}_{\frac{i-1}{n}})\right]\blackdiamond(\Delta W_{i}\Delta W_{i}^{\top}-{\Iq}\Delta_n),\\
	{\tilde M}^{nm,3}_{\frac{i-1}{n}}=&\sum_{\tilde \sigma\in\{\Id,\sigma\}}\sum_{\substack{1\le k<k'\le m}}\mathbb{H}(X^{nm,\tilde\sigma}_{\frac{i-1}{n}})\blackdiamond(\delta W_{i\tilde\sigma(k)}\delta W_{i\tilde\sigma(k')}^{\top}-\delta W_{i\tilde\sigma(k')}\delta W_{i\tilde\sigma(k)}^{\top})
\end{align*}
Now recalling  that $\sigma(k)=m-k+1$, for all $k\in\{1,\dots,m\}$, we get 
$$
	{\tilde M}^{nm,3}_{\frac{i-1}{n}}=\sum_{\substack{1\le k<k'\le m}}\left[\mathbb{H}(X^{nm}_{\frac{i-1}{n}})-\mathbb{H}(X^{nm,\sigma}_{\frac{i-1}{n}})\right]\blackdiamond(\delta W_{ik}\delta W_{ik'}^{\top}-\delta W_{ik'}\delta W_{ik}^{\top}).
$$
In what follows, by \eqref{eq:decM} we introduce for  $i\in\{1,\dots,n\}$
\begin{align}\label{eq:nbar}
	{\bar N}^{nm}_{\frac{i-1}{n}}=&\frac{1}{2}(N^{nm,\Id}_{\frac{i-1}{n}}+N^{nm,\sigma}_{\frac{i-1}{n}})+{\tilde N}^{nm}_{\frac{i-1}{n}},\\\label{eq:mbar}
	{\bar M}^{nm}_{\frac{i-1}{n}}=&\frac{1}{2}(M^{nm,\Id}_{\frac{i-1}{n}}+M^{nm,\sigma}_{\frac{i-1}{n}})+{\tilde M}_{\frac{i-1}{n}}^{nm,1}+{\tilde M}_{\frac{i-1}{n}}^{nm,2}-\frac{1}{2}{\tilde M}_{\frac{i-1}{n}}^{nm,3}.
\end{align}
 The proof of the  following lemma is postponed to the appendix \ref{app:A}.
\begin{lemma}\label{average}
	The error $\bar{X}^{nm,\sigma}_{\frac{i}{n}}-\bar{X}^{nm,\sigma}_{\frac{i-1}{n}}$, $i\in\{1,\dots,n\}$ can be expressed as follows
	\begin{multline}\label{eq:2.6a}
		\bar{X}^{nm,\sigma}_{\frac{i}{n}}-\bar{X}^{nm,\sigma}_{\frac{i-1}{n}}=f(\bar{X}^{nm,\sigma}_{\frac{i-1}{n}})\Delta_n+g(\bar{X}^{nm,\sigma}_{\frac{i-1}{n}})\Delta W_{i}\\+\mathbb{H}(\bar{X}^{nm,\sigma}_{\frac{i-1}{n}})\blackdiamond(\Delta W_{i}\Delta W_{i}^{\top}-{\Iq}\Delta_n)
		+{\bar M}^{mn}_{\frac{i-1}{n}}+{\bar N}^{nm}_{\frac{i-1}{n}},
	\end{multline}
	where $\mathbb{E}({\bar M}^{nm}_{\frac{i-1}{n}}|\mathcal{F}_{\frac{i-1}{n}})=0$ and for any integer $p\geq2$ there exists a constant $K_p$ such that
	\begin{align}\label{eq:2.6a1}
		\max_{0\leq i\leq n}\mathbb{E}(|{\bar M}^{nm}_{\frac{i-1}{n}}|^p)&\leq K_p\Delta_n^{3p/2},\\
	\label{eq:2.6a2}
		\max_{0\leq i\leq n}\mathbb{E}(|{\bar N}^{nm}_{\frac{i-1}{n}}|^p)&\leq K_p\Delta_n^{2p}.
	\end{align}
\end{lemma}
\begin{corollary}\label{bar_LP} We have 
$$
\mathbb{E}(\max_{0\leq i\leq n}|\bar{X}^{nm,\sigma}_{\frac{i}{n}}-X^{n}_{\frac{i}{n}}|^{p}) \leq C_p\Delta_n^p.
$$
\end{corollary}
\begin{proof}
 Let us define $S_k=\mathbb{E}(\max\limits_{0\leq k'\leq k}|\bar{X}^{nm,\sigma}_{\frac{k'}{n}}-X^{n}_{\frac{k'}{n}}|^{p})$, for any $0\le k\le n$.
For a fixed $k$, by summing \eqref{eq:2.6a} over the first $k'$ timesteps, we obtain
\begin{multline*}
	\bar{X}^{nm,\sigma}_{\frac{k'}{n}}-X^{n}_{\frac{k'}{n}}=\sum_{i=1}^{k'}(f(\bar{X}^{nm,\sigma}_{\frac{i-1}{n}})-f(X^{n}_{\frac{i-1}{n}}))\Delta_n+\sum_{i=1}^{k'}(g(\bar{X}^{nm,\sigma}_{\frac{i-1}{n}})-g(X^{n}_{\frac{i-1}{n}}))\Delta W_{i}\\+\sum_{i=1}^{k'}(\mathbb{H}(\bar{X}^{nm,\sigma}_{\frac{i-1}{n}})-\mathbb{H}(X^{n}_{\frac{i-1}{n}}))\blackdiamond (\Delta W_{i}\Delta W_{i}^{\top}-{\Iq}\Delta_n)
	+\sum_{i=1}^{k'}\bar{M}^{nm}_{\frac{i-1}{n}}+\sum_{i=1}^{k'}\bar{N}^{nm}_{\frac{i-1}{n}},
\end{multline*}
Then there is a generic constant $C_p>0$ such that 
\begin{align*}
	\mathbb{E}(\max_{0\leq k'\leq k}|\bar{X}^{nm,\sigma}_{\frac{k}{n}}-X^{n}_{\frac{k}{n}}|^{p})&\le C_p\mathbb{E}(\max_{0\leq k'\leq k}|\sum_{i=1}^{k'}(f(\bar{X}^{nm,\sigma}_{\frac{i-1}{n}})-f(X^{n}_{\frac{i-1}{n}}))\Delta_n|^{p})\\+&C_p\mathbb{E}(\max_{0\leq k\leq n}|\sum_{i=1}^{k'}(g(\bar{X}^{nm,\sigma}_{\frac{i-1}{n}})-g(X^{n}_{\frac{i-1}{n}}))\Delta W_{i}|^{p})\\+&C_p\mathbb{E}(\max_{0\leq k'\leq k}|\sum_{i=1}^{k'}(\mathbb{H}(\bar{X}^{nm,\sigma}_{\frac{i-1}{n}})-\mathbb{H}(X^{n}_{\frac{i-1}{n}}))\blackdiamond (\Delta W_{i}\Delta W_{i}^{\top}-{\Iq}\Delta_n)|^{p})\\
	+&C_p\mathbb{E}(\max_{0\leq k'\leq k}|\sum_{i=1}^{k'}\bar{M}^{nm}_{\frac{i-1}{n}}|^{p})+C_p\mathbb{E}(\max_{0\leq k'\leq k}|\sum_{i=1}^{k'}\bar{N}^{nm}_{\frac{i-1}{n}}|^{p}),
\end{align*}
By Jensen's  inequality and  (\nameref{Assume}), we have
\begin{align*}
&\mathbb{E}(\max_{0\leq k'\leq k}|\sum_{i=1}^{k'}(f(\bar{X}^{nm,\sigma}_{\frac{i-1}{n}})-f(X^{n}_{\frac{i-1}{n}}))\Delta_n|^{p})\le C_p\mathbb{E}(\max_{0\leq k\leq n}k^{p-1}\sum_{i=1}^{k}|(f(\bar{X}^{nm,\sigma}_{\frac{i-1}{n}})-f(X^{n}_{\frac{i-1}{n}}))\Delta_n|^{p}) \\
&\le C_p n^{p-1}\sum_{i=1}^{k}\mathbb{E}((f(\bar{X}^{nm,\sigma}_{\frac{i-1}{n}})-f(X^{n}_{\frac{i-1}{n}})|^{p})\Delta_n^p\le  C_p \sum_{i=1}^{k}\mathbb{E}(\max_{0\leq k\leq i-1}|\bar{X}^{nm,\sigma}_{\frac{k}{n}}-X^{n}_{\frac{k}{n}}|^{p})\Delta_n. 
\end{align*}
Similarly, by Jensen's inequality, the independence between $\Delta W_i$ and $\mathcal{F}_{\frac{i-1}{n}}$ and the assumption (\nameref{Assume}), $\mathbb{E}(\max\limits_{0\leq k'\leq k}|\sum_{i=1}^{k'}(\mathbb{H}(\bar{X}^{nm,\sigma}_{\frac{i-1}{n}})-\mathbb{H}(X^{n}_{\frac{i-1}{n}}))\blackdiamond (\Delta W_{i}\Delta W_{i}^{\top}-{\Iq}\Delta_n)|^{p})$ has an upper bound $ C_p \sum_{i=0}^{k-1}\mathbb{E}(\max_{0\leq j\leq i-1}|\bar{X}^{nm,\sigma}_{\frac{j}{n}}-X^{n}_{\frac{j}{n}}|^{p})\Delta_n$. 
Now, by Jensen's inequality and Lemma \ref{average}, $\mathbb{E}(\max\limits_{0\leq k'\leq k}|\sum_{i=1}^{k'}\bar{N}^{nm}_{\frac{i-1}{n}}|^{p})$ has an upper bound  $ C_p\Delta_n^p$.
Finally, by  the discrete BDG inequality in \cite{i} combined with Jensen's inequality, we have
\begin{align*}
&\mathbb{E}(\max_{0\leq k'\leq k}|\sum_{i=1}^{k'}(g(\bar{X}^{nm,\sigma}_{\frac{i-1}{n}})-g(X^{n}_{\frac{i-1}{n}}))\Delta W_{i}|^{p})\le
C_p\mathbb{E}(\sum_{i=1}^{k}|g(\bar{X}^{nm,\sigma}_{\frac{i-1}{n}})-g(X^{n}_{\frac{i-1}{n}}))\Delta W_{i}|^{2})^{p/2}\\&\le
C_pn^{p/2-1}\sum_{i=1}^{k}\mathbb{E}(|g(\bar{X}^{nm,\sigma}_{\frac{i-1}{n}})-g(X^{n}_{\frac{i-1}{n}}))|^{p})\mathbb{E}(|\Delta W_{i}|^{p})\le C_p\sum_{i=1}^{k}\mathbb{E}(\max_{0\leq k\leq i-1}|\bar{X}^{nm,\sigma}_{\frac{k}{n}}-X^{n}_{\frac{k}{n}}|^{p})\Delta_n.
\end{align*}
Similarly,  thanks to Lemma \ref{average}, $\mathbb{E}(\max\limits_{0\leq k'\leq k}|\sum_{i=1}^{k'}\bar{M}^{nm}_{\frac{i-1}{n}}|^{p})$ has an upper bound $C_p\Delta_n^p$.
Thus, it follows that
$$S_k\le C_p(\Delta_n^p+\sum_{i=0}^{k-1}S_i\Delta_n),\quad\textrm{for any }0\le k\le n. $$
By the discrete Gr\"onwal inequality, we have
$$S_n\le C_p\Delta_n^p+C\Delta_n^{p+1}\sum_{i=0}^{n-1}\exp\{(n-1-i)\Delta_n\}\le C_p\Delta_n^p+C_p\Delta_n^{p+1}\sum_{i=0}^{n-1}e\le  C_p\Delta_n^p. $$
\end{proof}
In what follows we give further expansions for the terms ${\tilde N}^{nm}_{\frac{i-1}{n}}$, ${\tilde M}^{nm,1}_{\frac{i-1}{n}}$,  ${\tilde M}^{nm,2}_{\frac{i-1}{n}}$ and 
${\tilde M}^{nm,3}_{\frac{i-1}{n}}$ defined above. These expansions will be useful later on. To do so, we apply twice the Taylor expansion until the second order, for each  $\ell\in\{1,\dots,d\}$, we get
\begin{align}\label{eq:Ntild}
	{\tilde N}^{nm}_{\ell,\frac{i-1}{n}}=&\frac{1}{16}(X^{nm}_{\frac{i-1}{n}}-X^{nm,\sigma}_{\frac{i-1}{n}})^{\top}\left(\nabla^2 f_{\ell}(\zeta^{n,1}_{\frac{i-1}{n}})+\nabla^2 f_{\ell}(\zeta^{n,2}_{\frac{i-1}{n}})\right)(X^{nm}_{\frac{i-1}{n}}-X^{nm,\sigma}_{\frac{i-1}{n}})\Delta_n,\\\label{eq:M1tild}
	{\tilde M}^{nm,1}_{\ell,\frac{i-1}{n}}=&\frac{1}{16}\sum_{j'=1}^q(X^{nm}_{\frac{i-1}{n}}-X^{nm,\sigma}_{\frac{i-1}{n}})^\top \left(\nabla^2g_{\ell j'}(\zeta^{n,3}_{\frac{i-1}{n}})+\nabla^2g_{\ell j'}(\zeta^{n,4}_{\frac{i-1}{n}})\right)(X^{nm}_{\frac{i-1}{n}}-X^{nm,\sigma}_{\frac{i-1}{n}})\Delta W^{j'}_{i}.
\end{align}	
Then using  twice the Taylor expansion until the first order we get
\begin{align}\label{eq:M2tild}
	{\tilde M}^{nm,2}_{\ell,\frac{i-1}{n}}=&\frac{1}{4}\Big[{{\dot h}}^{n,i,1}_{\ell\bullet\bullet}\blackdiamond(X^{nm}_{\frac{i-1}{n}}-X^{nm,\sigma}_{\frac{i-1}{n}})\Big] \blackdiamond  (\Delta W_{i}\Delta W_{i}-{\Iq}\Delta_n)
	\end{align}
and similarly
	\begin{align}\label{eq:M3tild}
	{\tilde M}^{nm,3}_{\ell,\frac{i-1}{n}}=&\sum_{\substack{1\le k<k'\le m}}\Big[{{\dot h}}^{n,i,2}_{\ell\bullet\bullet}\blackdiamond(X^{nm}_{\frac{i-1}{n}}-X^{nm,\sigma}_{\frac{i-1}{n}})\Big] \blackdiamond (\delta W_{ik}\delta W_{ik'}^{\top}-\delta W_{ik'}\delta W_{ik}^{\top}),
\end{align}
 where for $j$ and $j'\in\{1,\dots,q\}$, the $jj'$-th elements of the block matrices ${{\dot h}}^{n,i,1}_{\ell\bullet\bullet}$ and ${{\dot h}}^{n,i,2}_{\ell\bullet\bullet}$ are respectively given by  $({\dot  h}_{\ell \bullet\bullet}^{n,i,1})_{j j'}=\nabla h_{\ell jj'}( \zeta^{n,5}_{\frac{i-1}{n}})-\nabla h_{\ell jj'}( \zeta^{n,6}_{\frac{i-1}{n}})\in \R^{d\times 1}$ and $({\dot  h}_{\ell \bullet\bullet}^{n,i,2})_{j j'}=\nabla h_{\ell jj'}( \zeta^{n,7}_{\frac{i-1}{n}})\in \R^{d\times 1}$; for some vector point $\xi^{7,n}_{i}$ lying between ${X}^{nm,\sigma}_{\frac{i-1}{n}}$ and $X^{nm}_{\frac{i-1}{n}}$, some vector points  $\zeta^{n,1}_{\frac{i-1}{n}},\zeta^{n,3}_{\frac{i-1}{n}},\zeta^{n,5}_{\frac{i-1}{n}}$ lying between $\bar{X}^{nm,\sigma}_{\frac{i-1}{n}}$ and $X^{nm}_{\frac{i-1}{n}}$ and some vector points $\zeta^{n,2}_{\frac{i-1}{n}},\zeta^{n,4}_{\frac{i-1}{n}},\zeta^{n,6}_{\frac{i-1}{n}}$ lying  between $\bar{X}^{nm,\sigma}_{\frac{i-1}{n}}$ and $X^{nm,\sigma}_{\frac{i-1}{n}}$.
\begin{remark}\label{rk:fix_sigma}
In order to get the good rate of convergence, we need to assume that our $\sigma$ is strictly decreasing which leads us to take  the unique choice defined by  $\sigma(k)=m-k+1$. Otherwise, it is easy to check that  the term $n\sum_{i=1}^{[nt]}\tilde{M}^{nm,3}_{\frac{i-1}{n}}$ appearing  in the decomposition of the normalized error $n(\bar{X}_{\eta_n(t)}^{nm,\sigma}-X^n_{\eta_n(t)})$ is not tight.
\end{remark}
	
\subsection{Errors analysis of $U^n$ and $V^n$}
For $t\in[0,1]$ we have 
\begin{align}\label{eq:2.6}
	X^{nm,\sigma}_{\eta_n(t)}=&x_0+\sum_{i=1}^{[nt]}\sum_{k=1}^mf(X^{nm,\sigma}_{\frac{m(i-1)+k-1}{nm}})\frac{\Delta_n}{m}+\sum_{i=1}^{[nt]}\sum_{k=1}^mg(X^{nm,\sigma}_{\frac{m(i-1)+k-1}{nm}})\delta W_{i\sigma(k)}\nonumber\\&+\sum_{i=1}^{[nt]}\sum_{k=1}^m\mathbb{H}(X^{nm,\sigma}_{\frac{m(i-1)+k-1}{nm}})\blackdiamond (\delta W_{i\sigma(k)}\delta W_{i\sigma(k)}^{\top}-{\Iq}\frac{\Delta_n}{m}).
\end{align}

\paragraph{\bf Error analysis of $U^n$ }
At first, we consider the error $U^n_t=(U^{n,1}_t,\dots,U^{n,d}_t)^{\top}\in\mathbb{R}^d$  between the finer and the $\sigma$-antithetic Milstein approximations given by  $U^n_t=X^{nm}_{\eta_n(t)}-X^{nm,\sigma}_{\eta_n(t)}$. Then by $\eqref{eq:2.6.3}$, the expansion of $U^n$  takes the following form
\begin{align*}
	U^n_t=&\sum_{i=1}^{[nt]}(f(X^{nm}_{\frac{i-1}{n}})-f(X^{nm,\sigma}_{\frac{i-1}{n}}))\Delta_n+\sum_{i=1}^{[nt]}(g(X^{nm}_{\frac{i-1}{n}})-g(X^{nm,\sigma}_{\frac{i-1}{n}}))\Delta W_{i}\\&+\sum_{i=1}^{[nt]}(\mathbb{H}(X^{nm}_{\frac{i-1}{n}})-\mathbb{H}(X^{nm,\sigma}_{\frac{i-1}{n}}))\blackdiamond (\Delta W_{i}\Delta W_{i}^{\top}-{\Iq}\Delta_n)\nonumber\\
	&-\sum_{i=1}^{[nt]}\sum_{\substack{k,k'=1\\k<k'}}^m(\mathbb{H}(X^{nm}_{\frac{i-1}{n}})+\mathbb{H}(X^{nm,\sigma}_{\frac{i-1}{n}}))\blackdiamond\left(\delta W_{ik}\delta W_{ik'}^{\top}-\delta W_{ik'}\delta W_{ik}^{\top}\right)\\&+\sum_{i=1}^{[nt]}(M^{nm,\Id}_{\frac{i-1}{n}}-M^{nm,\sigma}_{\frac{i-1}{n}})+\sum_{i=1}^{[nt]}(N^{nm,\Id}_{\frac{i-1}{n}}-N^{nm,\sigma}_{\frac{i-1}{n}}).
\end{align*}
By Taylor's expansion, we rewrite $U^n$ as follows
\begin{align}\label{eq:4.4}
	U^n_t=&\sum_{i=1}^{[nt]}\dot{f}^{n}_i\blackdiamond U^n_{\frac{i-1}{n}}\Delta_n+\sum_{i=1}^{[nt]}\left(\dot{g}^n_i\blackdiamond U^{n}_{\frac{i-1}{n}}\right)\Delta W_{i}+\sum_{i=1}^{[nt]}(\mathbb{H}(X^{nm}_{\frac{i-1}{n}})-\mathbb{H}(X^{nm,\sigma}_{\frac{i-1}{n}}))\blackdiamond (\Delta W_{i}\Delta W_{i}^{\top}-{\Iq}\Delta_n)\nonumber\\
	&-\sum_{i=1}^{[nt]}\sum_{\substack{1\leq k<k'\le m}}(\mathbb{H}(X^{nm}_{\frac{i-1}{n}})+\mathbb{H}(X^{nm,\sigma}_{\frac{i-1}{n}}))\blackdiamond\left(\delta W_{ik}\delta W_{ik'}^{\top}-\delta W_{ik'}\delta W_{ik}^{\top}\right)\nonumber\\&+\sum_{i=1}^{[nt]}(M^{nm,\Id}_{\frac{i-1}{n}}-M^{nm,\sigma}_{\frac{i-1}{n}})+\sum_{i=1}^{[nt]}(N^{nm,\Id}_{\frac{i-1}{n}}-N^{nm,\sigma}_{\frac{i-1}{n}}),
\end{align}
where $\dot{f}^n_i\in(\R^{d\times1})^{d\times1}$ and $\dot{g}^n_i\in(\R^{d\times 1})^{d\times q}$  are block matrices  such that for  $\ell\in\{1,\dots,d\}$ the $\ell$-th block of $\dot{f}^n_i$ is given by $(\dot{f}^n_i)_{\ell}=\nabla f_{\ell}(\xi_{\frac{i-1}{n}}^{1,n})$, for $\ell\in\{1,\dots,d\}$ and $j\in\{1,\dots,q\}$ the $\ell j$-th block of $\dot{g}^n_i$ is given by $(\dot{g}^n_i)_{\ell j}=\nabla g_{\ell j}(\xi_{\frac{i-1}{n}}^{2,n})$ with $\xi_{\frac{i-1}{n}}^{1,n}$ and $\xi_{\frac{i-1}{n}}^{2,n}$  are some vector points lying between $X^{nm}_{\frac{i-1}{n}}$ and $X^{nm,\sigma}_{\frac{i-1}{n}}$ .
Now, the equation \eqref{eq:4.4} can be rewritten as 
\begin{align}\label{eq:3.2}
	U^n_{t}=&\sum_{i=1}^{[nt]}\dot{f}^n_i\blackdiamond U^n_{\frac{i-1}{n}}\Delta_n+\sum_{i=1}^{[nt]}\left(\dot{g}^n_i \blackdiamond U^{n}_{\frac{i-1}{n}}\right)\Delta W_{i}+\mathcal{M}^{n,1}_t+\mathcal{R}^{n,1}_{t}\nonumber\\
	=&\sum_{i=1}^{[nt]}\dot{f}^n_i\blackdiamond U^n_{\frac{i-1}{n}}\Delta_n+\sum_{j=1}^q\sum_{i=1}^{[nt]}\left((\dot{g}^n_i)_{\bullet j} \blackdiamond U^{n}_{\frac{i-1}{n}}\right)\Delta W_{i}^j+\mathcal{M}^{n,1}_t+\mathcal{R}^{n,1}_{t},
\end{align}
with $(\dot{g}^{n}_i)_{\bullet j}=((\dot{g}^{n}_i)_{1 j},\hdots,(\dot{g}^{n}_i)_{d j})^{\top}$,  $\mathcal{M}^{n,1}$ is the main term  and  $\mathcal{R}^{n,1}$ is the rest term given by
\begin{align*}
 \mathcal{M}^{n,1}_t=&-\sum_{i=1}^{[nt]}\sum_{\substack{k,k'=1\\ k<k'}}^m(\mathbb{H}(X^{nm}_{\frac{i-1}{n}})+\mathbb{H}(X^{nm,\sigma}_{\frac{i-1}{n}}))\blackdiamond\left(\delta W_{ik}\delta W_{ik'}^{\top}-\delta W_{ik'}\delta W_{ik}^{\top}\right),\\
 \mathcal{R}^{n,1}_{t}=&\sum_{i=1}^{[nt]}(\mathbb{H}(X^{nm}_{\frac{i-1}{n}})-\mathbb{H}(X^{nm,\sigma}_{\frac{i-1}{n}}))\blackdiamond (\Delta W_{i}\Delta W_{i}^{\top}-{\Iq}\Delta_n)+\sum_{i=1}^{[nt]}(M^{nm,\Id}_{\frac{i-1}{n}}-M^{nm,\sigma}_{\frac{i-1}{n}})\\&+\sum_{i=1}^{[nt]}(N^{nm,\Id}_{\frac{i-1}{n}}-N^{nm,\sigma}_{\frac{i-1}{n}}).
\end{align*}
The proof of the following lemma is postponed to appendix \ref{app:B}.
\begin{lemma}\label{lem:R1}
	Under the assumption (\nameref{Assume}), we have $\sqrt{n}\mathcal{R}^{n,1}\stackrel{L^p}{\rightarrow}0$ as $n\to\infty$.
\end{lemma}
\paragraph{\bf Error analysis of $V^n$ }
Now, we consider the error  $V^n_t=(V^{n,1}_t,\dots,V^{n,d}_t)^{\top}\in\mathbb{R}^d$ between the average of the finer and the coarser $\sigma$-antithetic Milstein approximations  given by $V^n_t=\bar{X}^{nm,\sigma}_{\eta_n(t)}-X^n_{\eta_n(t)}$. Similarly, by \eqref{eq:2.6a} and \eqref{eq:2.3}, we rewrite $V^n$ as follows
\begin{multline*}
	V^n_t=\sum_{i=1}^{[nt]}(f(\bar{X}^{nm,\sigma}_{\frac{i-1}{n}})-f(X^{n}_{\frac{i-1}{n}}))\Delta_n+\sum_{i=1}^{[nt]}(g(\bar{X}^{nm,\sigma}_{\frac{i-1}{n}})-g(X^{n}_{\frac{i-1}{n}}))\Delta W_{i}\\+\sum_{i=1}^{[nt]}(\mathbb{H}(\bar{X}^{nm,\sigma}_{\frac{i-1}{n}})-\mathbb{H}(X^{n}_{\frac{i-1}{n}}))\blackdiamond (\Delta W_{i}\Delta W_{i}^{\top}-{\Iq}\Delta_n)
	+\sum_{i=1}^{[nt]}\bar{M}^{nm}_{\frac{i-1}{n}}+\sum_{i=1}^{[nt]}\bar{N}^{nm}_{\frac{i-1}{n}},
\end{multline*}
where $\bar{N}^{nm}_{\frac{i-1}{n}}$ and $\bar{M}^{nm}_{\frac{i-1}{n}}$are respectively given by \eqref{eq:nbar} and \eqref{eq:mbar}. 
By the Taylor expansion, we have
\begin{multline*}	V^n_t=\sum_{i=1}^{[nt]}\bar{\dot{f}}^n_i\blackdiamond V^n_{\frac{i-1}{n}}\Delta_n+\sum_{i=1}^{[nt]}\left(\bar{\dot g}^n_i \blackdiamond V^{n}_{\frac{i-1}{n}}\right)\Delta W_{i}\\+\sum_{i=1}^{[nt]}(\mathbb{H}(\bar{X}^{nm,\sigma}_{\frac{i-1}{n}})-\mathbb{H}(X^{n}_{\frac{i-1}{n}}))\blackdiamond (\Delta W_{i}\Delta W_{i}^{\top}-{\Iq}\Delta_n)+\sum_{i=1}^{[nt]}\bar{M}_{\frac{i-1}{n}}+\sum_{i=1}^{[nt]}\bar{N}_{\frac{i-1}{n}},
\end{multline*}
where $\bar{\dot f}^n_i\in(\R^{d\times1})^{d\times1}$ and $\bar{\dot g}^n_i\in(\R^{d\times 1})^{d\times q}$  are block matrices such that for $\ell\in\{1,\dots,d\}$ the $\ell$-th block of $\bar{\dot f}^n_i$ is given by $(\bar{\dot f}^n_i)_{\ell}=\nabla f_{\ell}({\bar\xi}_{\frac{i-1}{n}}^{1,n})$,  for 
 $\ell\in\{1,\dots,d\}$  and $j\in\{1,\dots,q\}$ the  $\ell j$-th block of  $\bar{\dot g}^n_i$  is given by $(\bar{\dot g}^n_i)_{\ell j}=\nabla g_{\ell j}(\bar{\xi}^{2,n}_{\frac{i-1}{n}})$ with ${\bar\xi}_{\frac{i-1}{n}}^{1,n}$ and $\bar{\xi}^{2,n}_{\frac{i-1}{n}}$
 are some vector points lying between $X^{n}_{\frac{i-1}{n}}$ and $\bar{X}^{nm,\sigma}_{\frac{i-1}{n}}$.
Thanks to Lemma \ref{lem:N}, Lemma \ref{lem:M1}, Lemma \ref{lem:M2} and Lemma \ref{lem:M3}, the above equation  rewrites as follows
\begin{align}\label{eq:3.1}
  		V^{n}_{t}=&\sum_{i=1}^{[nt]}\bar{f}^n_i\blackdiamond V^n_{\frac{i-1}{n}}\Delta_n+\sum_{i=1}^{[nt]}\left(\bar{g}^n_i \blackdiamond V^{n}_{\frac{i-1}{n}}\right)\Delta W_{i}+\mathcal{M}^{n,2}_{t}+\mathcal{R}^{n,2}_{t},
  	\end{align}
where $\mathcal{M}^{n,2}$ stands for the main contributing term of the above error expansion and $\mathcal{R}^{n,2}$ is the rest term, for $t\in[0,1]$ they are given by 
	\begin{align}\label{eq:cal-M2}
	\mathcal{M}_t^{n,2}&=\frac{1}{2}\sum_{\tilde{\sigma}\in\{\Id,\sigma\}}\sum_{r=1}^4\Gamma^{n,\tilde{\sigma}}_t(r)+\tilde{N}^{nm}_t+\tilde{M}^{nm,1}_t-\frac{1}{2}\tilde{M}^{nm,3}_t,\\
	\label{eq:cal-R2}
		\mathcal{R}^{n,2}_t=&\frac{1}{2}\sum_{\tilde{\sigma}\in\{\Id,\sigma\}}\sum_{i=1}^{[nt]}\bigg({\tilde R}^{nm,\tilde{\sigma}}_{\frac{i-1}{n}}(0)+\sum_{r=0}^3R^{nm,\tilde{\sigma}}_{\frac{i-1}{n}}(r)\bigg)+\tilde{M}^{nm,2}_t\\
		&+\sum_{i=1}^{[nt]}(\mathbb{H}(\bar{X}^{nm,\sigma}_{\frac{i-1}{n}})-\mathbb{H}(X^{n}_{\frac{i-1}{n}}))\blackdiamond (\Delta W_{i}\Delta W_{i}^{\top}-{\Iq}\Delta_n),\nonumber
\end{align}
where for $r\in\{1,2,3\}$, $\tilde{M}^{nm,r}_t=\sum_{i=1}^{[nt]}\tilde{M}^{nm,r}_{\frac{i-1}{n}}$, $\tilde{N}^{nm}_t=\sum_{i=1}^{[nt]}\tilde{N}^{nm}_{\frac{i-1}{n}}$  with  $(\tilde{M}^{nm,r}_{\frac{i-1}{n}})_{1\le r\le 3}$, $\tilde{N}^{nm}_{\frac{i-1}{n}}$ are respectively given by \eqref{eq:Ntild},\eqref{eq:M1tild}, \eqref{eq:M2tild} and \eqref{eq:M3tild} and 
the rest terms  ${R}^{nm,\tilde{\sigma}}_{\frac{i-1}{n}}(0)$  and ${\tilde R}^{nm,\tilde{\sigma}}_{\frac{i-1}{n}}(0)$ are implicitly defined in \eqref{exp:N}  and  $(R^{nm,\tilde{\sigma}}_{\frac{i-1}{n}}(r))_{1\le r\le 3}$   are respectively  implicitly defined in \eqref{exp:M1}, \eqref{exp:M2} and \eqref{exp:M3}. Now, we introduce the $d$-dimensional processes $(\Gamma^{n,\tilde{\sigma}}_t(i),1\le i\le 4,t\in[0,1])$ whose $\ell^{th}$ components are given by 
\begin{align}\nonumber
	\Gamma^{n,\tilde{\sigma}}_{\ell,t}(1)=&\sum_{i=1}^{[nt]}\Bigg[\frac{(m-1)}{2m}\nabla f^{\top}_{\ell}(X_{\frac{i-1}{n}})f(X_{\frac{i-1}{n}}){\Delta^2_n}\\\label{eq:gam1}
	&+\frac{1}{2}\Big[g(X_{\frac{i-1}{n}})^\top\nabla^2f_{\ell}(X_{\frac{i-1}{n}})g(X_{\frac{i-1}{n}})\Big]\blackdiamond \sum_{k=1}^{m-1}(m-k)\delta W_{i\tilde{\sigma}(k)}\delta W_{i\tilde{\sigma}(k)}^{\top}\frac{\Delta_n}{m}\Bigg],\\\nonumber
	\Gamma^{n,\tilde{\sigma}}_{\ell,t}(2)=&\sum_{i=1}^{[nt]}\frac{\Delta_n}{m}\Bigg[\Big[{\nabla f_{\ell}}(X_{\frac{i-1}{n}})^\top g(X_{\frac{i-1}{n}})\Big]\sum_{k=1}^{m-1}(m-k)\delta W_{i\tilde{\sigma}(k)}\\\nonumber&+\sum_{j=1}^q\nabla g_{\ell j}^{\top}(X_{\frac{i-1}{n}})f(X_{\frac{i-1}{n}})\sum_{k=1}^{m-1}(m-k)\delta W_{i\tilde{\sigma}(k)}^j\\\label{eq:gam2}&+\sum_{j=1}^q\frac{1}{2}g(X_{\frac{i-1}{n}})^\top\nabla^2g_{\ell j}(X_{\frac{i-1}{n}})g(X_{\frac{i-1}{n}})\blackdiamond{\Iq}\sum_{k=2}^{m}(k-1)\delta W_{i\tilde{\sigma}(k)}^j\Bigg],\\\nonumber
	\Gamma^{n,\tilde{\sigma}}_{\ell,t}(3)=&\sum_{i=1}^{[nt]}\sum_{j=1}^q\Bigg[\nabla g_{\ell j}^{\top}(X_{\frac{i-1}{n}})\mathbb{H}(X_{\frac{i-1}{n}})\blackdiamond\hspace{-0.5cm}\sum_{\substack{1\le k'<k\le m}}\hspace{-0.5cm}(\delta W_{i\tilde{\sigma}(k')}\delta W_{i\tilde{\sigma}(k')}^{\top}-{\Iq}\frac{\Delta_n}{m})\delta W_{i\tilde{\sigma}(k)}^j\\\nonumber&+\frac{1}{2}\Big[g(X_{\frac{i-1}{n}})^\top\nabla^2g_{\ell j}(X_{\frac{i-1}{n}})g(X_{\frac{i-1}{n}})\Big]\blackdiamond\hspace{-0.5cm}\sum_{\substack{1\le k'<k\le m}}\hspace{-0.5cm}(\delta W_{i\tilde{\sigma}(k')}\delta W_{i\tilde{\sigma}(k')}^{\top}-{\Iq}\frac{\Delta_n}{m})\delta W_{i\tilde{\sigma}(k)}^j\\\label{eq:gam3}&+\Big[ {\dot  h}_{\ell \bullet\bullet}^{n,i}\blackdiamond g_{\bullet j}(X_{\frac{i-1}{n}})\Big]\blackdiamond\sum_{\substack{1\le k'<k\le m}}(\delta W_{i\tilde{\sigma}(k)}\delta W_{i\tilde{\sigma}(k)}^{\top}-{\Iq} \frac{\Delta_n}{m})\delta W^j_{i\tilde{\sigma}(k')}\Bigg],\\\nonumber
	\Gamma^{n,\tilde{\sigma}}_{\ell,t}(4)&=\sum_{i=1}^{[nt]}\Bigg[\sum_{j,j'=1}^q\nabla g_{\ell j}^{\top}(X_{\frac{i-1}{n}})
	\Big[ {\dot g}^n_{i}
	\blackdiamond g_{\bullet j'}(X_{\frac{i-1}{n}})\Big] \sum_{\substack{1\le k'' <k'<k\le m}}\hspace{-0.5cm}\delta W_{i\tilde{\sigma}(k')}\delta W_{i\tilde{\sigma}(k'')}^{j'}\delta W_{i\tilde{\sigma}(k)}^j
	\nonumber\\\label{eq:gam4}&
	+\sum_{j=1}^q\Big[g(X_{\frac{i-1}{n}})^\top\nabla^2g_{\ell j}(X_{\frac{i-1}{n}})g(X_{\frac{i-1}{n}})\Big]\blackdiamond\hspace{-0.5cm}\sum_{\substack{1\le k'' <k'<k\le m}}\hspace{-0.5cm}\delta W_{i\tilde{\sigma}(k')}\delta W_{i\tilde{\sigma}(k'')}^{\top}\delta W_{i\tilde{\sigma}(k)}^j\Bigg].
\end{align}	
The proof of the following lemma is also postponed to appendix \ref{app:B}.
\begin{lemma}\label{lem:R2}
We have $n\mathcal{R}^{n,2}\stackrel{L^p}{\rightarrow}0$ as $n\to\infty$.
\end{lemma}
\begin{remark} These processes
	$(\Gamma^{n,\tilde{\sigma}}_t(r),1\le r\le 4,t\in[0,1])$ are obtained by gathering together 
	the main terms in \eqref{exp:N}, \eqref{exp:M1}, \eqref{exp:M2} and \eqref{exp:M3} with taking into account their noise types  and with neglecting the rest terms $\big(R^{nm,\tilde{\sigma}}_{\ell,\frac{i-1}{n}}(r)\big)_{0\le r\le 3}$ and ${\tilde R}^{nm,\tilde{\sigma}}_{\ell,\frac{i-1}{n}}(0)$, for $\ell\in\{1,\dots,d\}$.
\end{remark}
%=========================================================================================================================================================================================================================================================================================%
\section{Asymptotic Behavior of the main terms}
According to expansion  \eqref{eq:3.2} and \eqref{eq:3.1} appearing in the decompositions of $U^n$ and $V^n$  we need to focus on the main terms $(\mathcal{M}^{n,1},\mathcal{M}^{n,2})$, where we recall that 
\begin{align}\label{eq:newmain1}
 \mathcal{M}^{n,1}_t=&-\sum_{i=1}^{[nt]}\sum_{\substack{1\le k, k'\le m\\k<k'}}^m(\mathbb{H}(X^{nm}_{\frac{i-1}{n}})+\mathbb{H}(X^{nm,\sigma}_{\frac{i-1}{n}}))\blackdiamond\left(\delta W_{ik}\delta W_{ik'}^{\top}-\delta W_{ik'}\delta W_{ik}^{\top}\right),\\
 \mathcal{M}_t^{n,2}=&\frac{1}{2}\sum_{\tilde{\sigma}\in\{\Id,\sigma\}}\sum_{r=1}^4\Gamma^{n,\tilde{\sigma}}_t(r)+\tilde{N}^{nm}_t+\tilde{M}^{nm,1}_t-\frac{1}{2}\tilde{M}^{nm,3}_t\nonumber
 \end{align}
with $\tilde{N}^{nm}_t$ respectively $\tilde{M}^{nm,1}_t$ and $\tilde{M}^{nm,3}_t$ are given by relation \eqref{eq:Ntild} respectively $\eqref{eq:M1tild}$ and $\eqref{eq:M3tild}$, $(\Gamma^{n,\tilde{\sigma}}_t(r),1\le r\le 4,t\in[0,1])$ are defined as above in \eqref{eq:gam1}, \eqref{eq:gam2}, \eqref{eq:gam3} and \eqref{eq:gam4}.\\
Unlike the first main term $\mathcal{M}^{n,1}$, that has explicit form of the noise, the second main term $\mathcal{M}^{n,2}$ needs further development  in order to identify its noise parts. To do so,  we need the following lemma that will be proven in appendix \ref{app:B}. 
\begin{lemma}\label{lem:gambar}
Let $\bar{\Gamma}^{n}_t(r)=\frac{\Gamma^{n,\Id}_t(r)+\Gamma^{n,\sigma}_t(r)}{2}\in\R^d$, for $r\in\{1,2,3,4\}$. Then we rewrite $\bar{\Gamma}^{n}(r)$ as follows, for $\ell\in\{1,\hdots,d\}$,
\begin{multline*}
	\bar{\Gamma}^{n}_{\ell,t}(1)=\sum_{i=1}^{[nt]}\frac{(m-1)\Delta_n}{2m}\Bigg[\nabla f_{\ell}(X_{\frac{i-1}{n}})^{\top}f(X_{\frac{i-1}{n}}){\Delta_n}\\
	+\frac{1}{2}\Big[g(X_{\frac{i-1}{n}})^\top\nabla^2f_{\ell}(X_{\frac{i-1}{n}})g(X_{\frac{i-1}{n}})\Big]\blackdiamond \sum_{k=1}^{m}\delta W_{ik}\delta W_{ik}^{\top}\Bigg],
\end{multline*}
\begin{multline*}
	\bar{\Gamma}^{n}_{\ell,t}(2)=\sum_{i=1}^{[nt]}\sum_{j=1}^q\frac{(m-1)\Delta_n}{2m}\Bigg[{\nabla f_{\ell}}(X_{\frac{i-1}{n}})^\top g_{\bullet j}(X_{\frac{i-1}{n}})\Delta W_i^j+\nabla g_{\ell j}(X_{\frac{i-1}{n}})^{\top}f(X_{\frac{i-1}{n}})\Delta W_i^j\\+\frac{1}{2}g(X_{\frac{i-1}{n}})^\top\nabla^2g_{\ell j}(X_{\frac{i-1}{n}})g(X_{\frac{i-1}{n}})\blackdiamond{\Iq}\Delta W_i^j\Bigg],
\end{multline*}
\begin{align*}
	\bar{\Gamma}^{n}_{\ell,t}(3)=&\frac{1}{2}\sum_{i=1}^{[nt]}\sum_{j=1}^q\Bigg[\nabla g_{\ell j}(X_{\frac{i-1}{n}})^{\top}\mathbb{H}(X_{\frac{i-1}{n}})\blackdiamond\hspace{-0.5cm}\sum_{\substack{1\le k, k'\le m\\k'\neq k}}\left(\delta W_{ik'}\delta W_{ik'}^{\top}-{\Iq}\frac{\Delta_n}{m}\right)\delta W_{ik}^{j}\\&+\frac{1}{2}\Big[g(X_{\frac{i-1}{n}})^\top\nabla^2g_{\ell j}(X_{\frac{i-1}{n}})g(X_{\frac{i-1}{n}})\Big]\blackdiamond\hspace{-0.5cm}\sum_{\substack{1\le k, k'\le m\\k'\neq k}}\left(\delta W_{ik'}\delta W_{ik'}^{\top}-{\Iq}\frac{\Delta_n}{m}\right)\delta W_{ik}^{j}\\&+\Big[ {\dot  h}_{\ell \bullet\bullet}^{n,i}\blackdiamond g_{\bullet j}(X_{\frac{i-1}{n}})\Big]\blackdiamond\hspace{-0.5cm}\sum_{\substack{1\le k, k'\le m\\k'\neq k}}(\delta W_{ik'}\delta W_{ik'}^{\top}-{\Iq} \frac{\Delta_n}{m})\delta W_{ik}^j\Bigg],
\end{align*}
\begin{align*}
	\bar{\Gamma}^{n}_{\ell,t}(4)&=\frac{1}{2}\sum_{i=1}^{[nt]}\Bigg[\sum_{j,j'=1}^q\nabla g_{\ell j}(X_{\frac{i-1}{n}})^{\top}
	\Big[ {\dot g}^n_{i}
	\blackdiamond g_{\bullet j'}(X_{\frac{i-1}{n}})\Big]\hspace{-0.7cm}\sum_{\substack{1\le k'' <k'<k\le m}}\hspace{-0.7cm}\delta W_{ik'}(\delta W_{ik''}^{j'}\delta W_{ik}^j+\delta W_{ik''}^{j}\delta W_{ik}^{j'})
	\nonumber\\&
	+\sum_{j=1}^q\Big[g(X_{\frac{i-1}{n}})^\top\nabla^2g_{\ell j}(X_{\frac{i-1}{n}})g(X_{\frac{i-1}{n}})\Big]\blackdiamond\hspace{-0.7cm}\sum_{\substack{1\le k'' <k'<k\le m}}\hspace{-0.7cm}\delta W_{ik'}(\delta W_{ik''}^{\top}\delta W_{ik}^j+\delta W_{ik''}^j\delta W_{ik}^{\top})\Bigg].
\end{align*}
\end{lemma}
\begin{proof}
Concerning $\bar{\Gamma}^{n}(1)$, according to the definition of $\sigma$, we only use 
$$
\sum_{k=1}^{m}(m-k)\delta W_{i\sigma(k)}\delta W_{i\sigma(k)}^\top=\sum_{k=1}^{m}(k-1)\delta W_{ik}\delta W_{ik}^\top.
$$
Similarly, we obtain  $\bar{\Gamma}^{n}(2)$ using
$$
\sum_{k=1}^{m}(m-k)\delta W_{i\sigma(k)}=\sum_{k=1}^{m}(k-1)\delta W_{ik}\quad\mbox{and}\quad\sum_{k=1}^{m}(k-1)\delta W_{i\sigma(k)}=\sum_{k=1}^{m-1}(m-k)\delta W_{ik}.
$$
To get $\bar{\Gamma}^{n}(3)$, we use
$$
\sum_{\substack{k, k'=1\\k'<k}}^m\left(\delta W_{i\sigma(k')}\delta W_{i\sigma(k')}^{\top}-\Omega\frac{\Delta_n}{m}\right)\delta W_{i\sigma(k)}^{j} =\sum_{\substack{k, k'=1\\k< k'}}^m\left(\delta W_{ik'}\delta W_{ik'}^{\top}-\Omega\frac{\Delta_n}{m}\right)\delta W_{ik}^{j}.
$$
Finally, we obtain $\bar{\Gamma}^{n}(4)$ using
\begin{align*} \sum_{k=3}^m\sum_{k'=2}^{k-1}\sum_{k''=1}^{k'-1}\delta W_{i\sigma(k'')}^{j}\delta W_{i\sigma(k')}^{j'}\delta W_{i\sigma(k)}^{j''}=\sum_{k=1}^{m-2}\sum_{k'=k+1}^{m-1}\sum_{k''=k'+1}^{m}\delta W_{ik''}^{j}\delta W_{ik'}^{j'}\delta W_{ik}^{j''}.
		\end{align*}
\end{proof}		

Now, thanks to the above lemma, \eqref{eq:3.1} can be rewritten in a better way as follows
\begin{align*}
	V^{n}_{t}=&\sum_{i=1}^{[nt]}\bar{f}^n_i\blackdiamond V^n_{\frac{i-1}{n}}\Delta_n+\sum_{i=1}^{[nt]}\left(\bar{g}^n_i \blackdiamond V^{n}_{\frac{i-1}{n}}\right)\Delta W_{i}+\mathcal{M}^{n,2}_{t}+\mathcal{R}^{n,2}_{t},
\end{align*}
where for $t\in[0,1]$,
\begin{align}\label{eq:newmain2}
	\mathcal{M}^{n,2}_t=&\sum_{r=1}^{4}\bar{\Gamma}^{n}_t(r)+\tilde{N}^{nm}_t+\tilde{M}^{nm,1}_t-\frac{1}{2}\tilde{M}^{nm,3}_t,
\end{align}
here we recall that for $r\in\{1,3\}$, $\tilde{M}^{nm,r}_t=\sum_{i=1}^{[nt]}\tilde{M}^{nm,r}_{\frac{i-1}{n}}$, $\tilde{N}^{nm}_t=\sum_{i=1}^{[nt]}\tilde{N}^{nm}_{\frac{i-1}{n}}$  with  $\tilde{N}^{nm}_{\frac{i-1}{n}}$ and  $(\tilde{M}^{nm,r}_{\frac{i-1}{n}})_{r\in\{1, 3\}}$ are respectively given by \eqref{eq:Ntild},\eqref{eq:M1tild} and \eqref{eq:M3tild}.\\
Now, in order to prove the convergence in law of the couple $(\mathcal{M}^{n,1},\mathcal{M}^{n,2})$, we first need to study the asymptotic behavior of the distribution of the noises vector
$(Z^n_0,Z^n_1,Z^n_2,Z^n_3),$   where $Z^n_0=(Z^{n,jj'}_0)_{j,j'\in\{1,\dots,q\}}$, $Z^n_2=(Z^{n,jj'}_2)_{j,j'\in\{1,\dots,q\}}$ are $q^2$-matrices
\begin{align*}
  Z^{n,jj'}_{0,t}=&\sum_{i=1}^{[nt]}\sum_{k=1}^{m}\delta W_{ik}^j\delta W_{ik}^{j'},\\  Z^{n,jj'}_{2,t}=&\sqrt{n}\sum_{i=1}^{[nt]}\sum_{\substack{1\le k<k'\le m}}\hspace{-0.5cm}\left(\delta W_{ik}^j\delta W_{ik'}^{j'}-\delta W_{ik'}^j\delta W_{ik}^{j'}\right),
  \end{align*}
 and  $Z^n_1=(Z^{n,jj'j''}_1)_{j,j',j''\in\{1,\dots,q\}}$ and $Z^n_3=(Z^{n,jj'j''}_3)_{j,j',j''\in\{1,\dots,q\}}$ are $q^3$-matrices 
\begin{align*}
 Z^{n,jj'j''}_{1,t}=&n\sum_{i=1}^{[nt]}\hspace{-0.3cm}\sum_{\substack{k'\neq k\\1\leq k, k'\leq m}}\hspace{-0.3cm}(\delta W_{ik'}^j\delta W_{ik'}^{j'}-\delta_{jj'}\frac{\Delta_n}{m})\delta W_{ik}^{j''},\\ 
	Z^{n,jj'j''}_{3,t}=&n\sum_{i=1}^{[nt]}\sum_{\substack{1\leq k'' <k'<k\leq m}}\hspace{-0.5cm}\delta W_{ik'}^{j'}(\delta W_{ik}^{j}\delta W_{ik''}^{j''}+\delta W_{ik''}^{j}\delta W_{ik}^{j''}).
\end{align*}
\begin{lemma}\label{prop:1}
	Let us consider the triangular arrays given by
	\begin{align*}
			Z^{n}_{0,t}=\sum_{i=1}^{[nt]}\sum_{k=1}^{m}\delta W_{ik}\delta W_{ik}^{\top}, \quad \textrm{with}\quad t\in[0,1].
	\end{align*}
	$$ \textrm{As } n\rightarrow \infty, 	\textrm{we have}\quad Z^{n}_0-Z_0\stackrel{L^p}{\rightarrow}0,\quad\textrm{where}\quad Z_{0,t}=t{\Iq},\quad t\in[0,1].$$
	
\end{lemma}
\begin{proof}  For any fixed $t\in[0,1]$, we rewrite $Z^n_0$ as follows
\begin{align*}
		Z^{n}_{0,t}=\sum_{i=1}^{[nt]}\sum_{k=1}^{m}(\delta W_{ik}\delta W_{ik}^{\top}-\frac{\Delta_n}{m}{\Iq})+[nt]\Delta_n{\Iq}.
	\end{align*}
	 As  $\mathbb{E}(\sum_{k=1}^{m}(\delta W_{ik}\delta W_{ik}^{\top}-\frac{\Delta_n}{m}{\Iq})|\mathcal{F}_{\frac{i-1}{n}})=0$ the by the discrete BDG inequality  and Jensen's inequality, there is a generic positive constant $C$ such that
\begin{align*}
\mathbb{E}\left(\sup_{0\leq t\leq 1}|\sum_{i=1}^{[nt]}\sum_{k=1}^{m}(\delta W_{ik}\delta W_{ik}^{\top}-\frac{\Delta_n}{m}\Iq)|^p\right)&=
	\mathbb{E}\left(\max_{0\leq \ell\leq n}|\sum_{i=1}^{\ell}\sum_{k=1}^{m}(\delta W_{ik}\delta W_{ik}^{\top}-\frac{\Delta_n}{m}\Iq)|^p\right)\\
	&\leq C\mathbb{E} \left(\sum_{i=1}^{n}\sum_{k=1}^{m}|\delta W_{ik}\delta W_{ik}^{\top}-\frac{\Delta_n}{m}\Iq|^2\right)^{p/2}\\
	&\leq  C n^{p/2-1}\sum_{i=1}^{n}\sum_{k=1}^{m}\mathbb{E}|\delta W_{ik}\delta W_{ik}^{\top}-\frac{\Delta_n}{m}\Iq|^p \leq C \Delta_n^{p/2}.
\end{align*}
 Then it follows that  $\max_{0\leq \ell\leq [nt]}\sum_{i=1}^{\ell}\sum_{k=1}^{m}(\delta W_{ik}\delta W_{ik}^{\top}-\frac{\Delta_n}{m}{\Iq})\stackrel{L^p}{\rightarrow} 0.$
Thus, we get the convergence of $Z^n$ using that $[nt]\Delta_n\Iq\rightarrow t {\Iq}$ as $n\rightarrow\infty$.
\end{proof}	
\begin{theorem}\label{thm:1}
	Let us consider the scalar components of the triangular array triplet $(Z^{n}_{1},Z^{n}_{2},Z^{n}_{3})$ given by 
	\begin{align*}
		\forall j,j',j''&\in\{1,\dots,q\},\\&Z^{n,jj'j''}_{1,t}=\sum_{i=1}^{[nt]}\zeta^{n,jj'j''}_{i,1}\textrm{ where }\zeta^{n,jj'j''}_{i,1}=n\hspace{-0.4cm}\sum_{\substack{k\neq k'\\1\leq k, k'\leq m}}\hspace{-0.4cm}\left(\delta W_{ik'}^{j}\delta W_{ik'}^{j'}-\delta_{jj'}\frac{\Delta_n}{m}\right)\delta W_{ik}^{j''},\\
		\forall j,j'\in&\{1,\dots,q\},\\&Z^{n,jj'}_{2,t}=\sum_{i=1}^{[nt]}\zeta^{n,jj'}_{i,2}\textrm{ where }\zeta^{n,jj'}_{i,2}=\sqrt{n}\hspace{-0.4cm}\sum_{\substack{1\le k<k'\le m}}\hspace{-0.4cm}\left(\delta W_{ik}^j\delta W_{ik'}^{j'}-\delta W_{ik'}^j\delta W_{ik}^{j'}\right),\\
		\forall j,j',j''& \in\{1,\dots,q\},\\& Z_{3,t}^{n,jj'j''}=\sum_{i=1}^{[nt]}\zeta_{i,3}^{n,jj'j''}\textrm{ where }\zeta_{i,3}^{n,jj'j''}=n\hspace{-0.6cm}\sum_{\substack{1\le k'' <k'<k\le m}}\hspace{-0.5cm}\delta W_{ik'}^{j'}(\delta W_{ik}^{j}\delta W_{ik''}^{j''}+\delta W_{ik''}^{j}\delta W_{ik}^{j''}),
	\end{align*}
	with $t\in[0,1]$ and $Z^n_{1,t}\in\mathbb{R}^{q^3}$, $Z^n_{2,t}\in\mathbb{R}^{q^2}$ and $Z^n_{3,t}\in\mathbb{R}^{q^3}$.
	 Then as $n\rightarrow \infty$, we have
	\begin{align*}
		(W,Z^{n}_{1},Z^{n}_{2},Z^{n}_{3})\stackrel{\rm stably}{\Rightarrow}(W,Z_1,Z_2,Z_3),
	\end{align*}
	where $Z_{1,t}^{jj'j''}=\left\{\begin{array}{lll}
	\frac{\sqrt{m-1}}{m}B^{jj'j''}_{1,t}&,&j> j'\\\frac{\sqrt{2(m-1)}}{m}B^{jjj''}_{1,t}&,&j=j'\\
	\frac{\sqrt{m-1}}{m}B^{j'jj''}_{1,t}&,&j< j'
	\end{array}\right.$ , $Z_{2,t}^{jj'}=\left\{\begin{array}{lll}
	\sqrt{\frac{m-1}{m}}B^{jj'}_{2,t}&,&j>j'\\0&,&j=j'\\-\sqrt{\frac{m-1}{m}}B^{j'j}_{2,t}&,&j< j'\end{array}\right.$ 
	$$\mbox{and }\quad Z_{3,t}^{jj'j''}=\left\{\begin{array}{lll}
	\sqrt{\frac{(m-1)(m-2)}{3m^2}}B^{jj'j''}_{3,t}&,&j> j''\\ \sqrt{\frac{2(m-1)(m-2)}{3m^2}}B^{jj'j''}_{3,t}&,&j= j''\\ \sqrt{\frac{(m-1)(m-2)}{3m^2}}B^{j''j'j}_{3,t}&,&j< j''\end{array}\right.$$ with $(B^{jj'j''}_1)_{\substack{1\leq j,j',j''\leq q\\j\geq j'}}$ and $(B^{jj'j''}_3)_{\substack{1\leq j,j',j''\leq q\\j\geq j''}}$   are two standard $q^2(q+1)/2$-dimensional Brownian motions and $(B^{jj'}_2)_{1\leq j'<j\leq q}$ is a standard $q(q-1)/2$-dimensional Brownian motion. Moreover, we have $B_1$,  $B_2$ and $B_3$ are independent of $W$ and also independent of each other. Furthermore, we have the  \eqref{ut} of $Z^n_0$, $Z^{n}_{1}$, $Z^{n}_{2}$ and $Z^{n}_{3}$.
\end{theorem}
\begin{remark}\label{rk:UT}
It is worth noticing that when $m=2$ the noise term $Z_{3}^{n}$ vanishes at the limit. 
\end{remark}

%%%%%%%%%%%%%%%%%%%%%%%%%%%%% PROOF OF THEOREM 5.4
\begin{proof}[Proof of Theorem \ref{thm:1}]
We aim to use Theorem 3.2 of Jacod in \cite{c} (see in appendix Theorem \ref{thm:d3}) combined with some useful technical tools in the proof of  Theorem 5.1 of Jacod and Protter (1998) \cite{d}. We split our proof into four main steps to check the four conditions of  theorem \ref{thm:d3}. 
	\paragraph{\textbf{Step 1}} For all $ j,j',j''\in\{1,\dots,q\}$, we have $\mathbb{E}(\zeta^{n,jj'j''}_{i,1}|\mathcal{F}_{\frac{i-1}{n}})=\mathbb{E}(\zeta^{n,jj'}_{i,2}|\mathcal{F}_{\frac{i-1}{n}})=\mathbb{E}(\zeta^{n,jj'j''}_{i,3}|\mathcal{F}_{\frac{i-1}{n}})=0.$
	 Then the first condition \eqref{eq:a} of theorem \ref{thm:d3} is satisfied.%%%%%%%%%%%%%%%%%%%%%%%%%%%%%%%%%%%%%%%%%%%%%%%%%%%%%%%%% END STEP 1
	
	\paragraph{\textbf{Step 2}} For this step, we need to check the validity of condition \eqref{eq:b} of theorem \ref{thm:d3} for our three triangular arrays.
\paragraph{{\bf First triangular array.}}
		Using the symmetric structure of  $\zeta^{n,jj'j''}_{i,1}$ it is sufficient to consider only the case $j\geq j'$. Now thanks to the independence between the increments, we have for all $i\in\{1,\dots, n\}$,
		\begin{align*}
			&\mathbb{E}((\zeta^{n,jj'j''}_{i,1})^2|\mathcal{F}_{\frac{i-1}{n}})=\sum_{\substack{k\neq k'\\1\leq k, k'\leq m}}n^2\mathbb{E}\left(\left(\delta W_{ik'}^{j}\delta W_{ik'}^{j'}-\delta_{jj'}\frac{\Delta_n}{m}\right)^2(\delta W_{ik}^{j''})^2\right)+\\
			&\sum_{\substack{{k}_1\neq{k'}_1, {k}_2\neq{k'}_2 \\1\leq {k'}_1,{k'}_2,{k}_1,{k}_2\leq m\\({k'}_1,{k}_1)\neq({k'}_2,{k}_2)}}n^2\mathbb{E}\left(\left(\delta W_{i{k'}_1}^{j}\delta W_{i,{k'}_1}^{j'}-\delta_{jj'}\frac{\Delta_n}{m}\right)\delta W_{i{k}_1}^{j''}\left(\delta W_{i{k'}_2}^{j}\delta W_{i{k'}_2}^{j'}-\delta_{jj'}\frac{\Delta_n}{m}\right)\delta W_{i{k}_2}^{j''}\right).
		\end{align*}	
Actually, the generic term of the above second sum is equal to zero. To check that we consider the three following subcases.
\begin{itemize}
\item If $k_1\neq k_2$ and ${k'}_1={k'}_2$, then we can deduce that $k_1\notin\{{k'}_1,{k'}_2,k_2\}$ and therefore the generic term  is equal to 
	$$\mathbb{E}(\delta W_{i{k}_1}^{j''})\mathbb{E}\left(\left(\delta W_{i{k'}_1}^{j}\delta W_{i,{k'}_1}^{j'}-\delta_{jj'}\frac{\Delta_n}{m}\right)\left(\delta W_{i{k'}_2}^{j}\delta W_{i{k'}_2}^{j'}-\delta_{jj'}\frac{\Delta_n}{m}\right)\delta W_{i{k}_2}^{j''}\right)=0.$$
\item If $k_1= k_2$ and ${k'}_1\neq{k'}_2$, then we  can deduce that ${k'}_1\notin\{{k}_1,{k'}_2,k_2\}$. Therefore, the generic term is equal to
$$
\mathbb{E}\left(\delta W_{i{k'}_1}^{j}\delta W_{i,{k'}_1}^{j'}-\delta_{jj'}\frac{\Delta_n}{m}\right)\mathbb{E}\left(\delta W_{i{k}_1}^{j''}\left(\delta W_{i{k'}_2}^{j}\delta W_{i{k'}_2}^{j'}-\delta_{jj'}\frac{\Delta_n}{m}\right)\delta W_{i{k}_2}^{j''}\right)=0.
$$
\item If $k_1\neq k_2$ and ${k'}_1\neq{k'}_2$, then we have two  subsubcases:
\begin{itemize}  
	\item If $k_1={k'}_2$, we have  ${k}_1\notin\{{k'}_1,k_2\}$. Then, the  generic term is equal to 
	$$\mathbb{E}\left(\left(\delta W_{i{k'}_1}^{j}\delta W_{i,{k'}_1}^{j'}-\delta_{jj'}\frac{\Delta_n}{m}\right)\delta W_{i{k}_2}^{j''}\right)\mathbb{E}\left(\delta W_{i{k}_1}^{j''}\left(\delta W_{ik_1}^{j}\delta W_{ik_1}^{j'}-\delta_{jj'}\frac{\Delta_n}{m}\right)\right)=0.$$
	\item If $k_1\neq{k'}_2$, we have $k_1\notin\{{k'}_1,{k'}_2,k_2\}$. Then, the generic term is equal to 
	$$\mathbb{E}(\delta W_{i{k}_1}^{j''})\mathbb{E}\left(\left(\delta W_{i{k'}_1}^{j}\delta W_{i,{k'}_1}^{j'}-\delta_{jj'}\frac{\Delta_n}{m}\right)\left(\delta W_{i{k'}_2}^{j}\delta W_{i{k'}_2}^{j'}-\delta_{jj'}\frac{\Delta_n}{m}\right)\delta W_{i{k}_2}^{j''}\right)=0.$$
\end{itemize}
\end{itemize}
It is worth noticing that the above arguments rely only on the independence between the increments without using the independence between the components of the Brownian vector.     
Concerning the generic term of the first sum, if $j=j'$ then  it is equal to
$$
\mathbb{E}\left((\delta W_{ik'}^{j})^2-\frac{\Delta_n}{m}\right)^2\mathbb{E}(\delta W_{ik}^{j''})^2=\frac{2}{n^3m^3}
$$
and for $j>j'$ it is equal to 
$
\mathbb{E}(\delta W_{ik'}^{j})^2\mathbb{E}(\delta W_{ik'}^{j'})^2\mathbb{E}(\delta W_{ik}^{j''})^2=\frac{1}{n^3m^3}.
$
Thus, we get as ${n\rightarrow\infty}$
	$$\sum_{i=1}^{[nt]}\mathbb{E}((\zeta^{n,jj'j''}_{i,1})^2|\mathcal{F}_{\frac{i-1}{n}})	\displaystyle {\longrightarrow}\left\{\begin{array}{ll}
		\displaystyle \dfrac{m-1}{m^2}t,&j>j'\\\\\displaystyle \dfrac{2(m-1)}{m^2}t&j=j'.
	\end{array}\right.
$$
Now, it remains to check that for any $q\ge j\geq j'\ge 1$, $q\ge \bar{j}\geq \bar{j'}\ge 1$ and $j'',\bar{j''}\in\{1,\dots,q\}$ s.t. $(j,j',j'')\neq(\bar{j},\bar{j'},\bar{j''})$ , we have
		$\sum_{i=1}^{[nt]}\mathbb{E}(\zeta_{i,1}^{n,jj'j''}\zeta_{i,1}^{n,\bar j\bar{j'}\bar{j''}}|\mathcal{F}_{\frac{i-1}{n}})=0.$ To do so, we write
		\begin{align*}
			&\mathbb{E}(\zeta_{i,1}^{n,jj'j''}\zeta_{i,1}^{n,\bar j\bar{j'}\bar{j''}}|\mathcal{F}_{\frac{i-1}{n}})\\
			=&\hspace{-0.6cm}\sum_{\substack{1\leq k_1,{k'}_1,k_2,{k'}_2\leq m\\k_1\neq {k'}_1,k_2\neq {k'}_2}}\hspace{-0.6cm}n^2\mathbb{E}\left(\left(\delta W_{i{k'}_1}^{j}\delta W_{i{k'}_1}^{j'}-\delta_{jj'}\frac{\Delta_n}{m}\right)\delta W_{i{k}_1}^{j''}\left(\delta W_{i{k'}_2}^{\bar{j}}\delta W_{i{k'}_2}^{\bar{j'}}-\delta_{\bar{j}\bar{j'}}\frac{\Delta_n}{m}\right)\delta W_{i{k}_2}^{\bar{j''}}\right).
		\end{align*}
		It is easy to check that the arguments given above to prove that this term vanishes remain valid for the particular case $(k_1,{k'}_1)\neq(k_2,{k'}_2)$ and this, as noticed above is independent of the choice of $(j,j',j'')$ and $(\bar{j},\bar{j'},\bar{j''})$.
		Thus, we only need to consider the case $(k_1,{k'}_1)=(k_2,{k'}_2)$. Therefore, by the independence between the increments we rewrite the generic term as follows
		\begin{equation}\label{cov_first_TA} \mathbb{E}\left((\delta W_{i{k'}_1}^{j}\delta W_{i{k'}_1}^{j'}-\delta_{jj'}\frac{\Delta_n}{m})(\delta W_{i{k'}_1}^{\bar{j}}\delta W_{i{k'}_1}^{\bar{j'}}-\delta_{\bar{j}\bar{j'}}\frac{\Delta_n}{m})\right)
		       \mathbb{E}\left( \delta W_{i{k}_1}^{j''}\delta W_{i{k}_1}^{\bar{j''}}\right).
		 \end{equation}
		Then, it is obvious that when $j''\neq \bar{ j''}$ this generic term vanishes. Now
		when $j''= \bar{ j''}$, thanks to its symmetric structure it is sufficient to consider only the case $j\neq \bar{j}$. For this we have three subcases.
		   \begin{itemize}
		        \item If $j>j'$ and $\bar j> \bar{j'}$ then $\delta_{jj'}=\delta_{\bar j\bar{j'}}=0$ and we have two possibilities: either $j\neq \bar{j'}$ then the generic term rewrites 
		         $
		        \mathbb{E}\left(\delta W_{i{k'}_1}^{j}\right) \mathbb{E}\left(\delta W_{i{k'}_1}^{j'}\delta W_{i{k'}_1}^{\bar{j}}\delta W_{i{k'}_1}^{\bar{j'}}\right)=0
		        $
		        or $j=\bar{j'}$ and then as $\bar j>\bar j'=j>j'$  the generic term rewrites 
		        $\mathbb{E}\left((\delta W_{i{k'}_1}^{j})^2\right) \mathbb{E}\left(\delta W_{i{k'}_1}^{j'}\delta W_{i{k'}_1}^{\bar{j}}\right)=0$.
		        \item If $j=j'$ and $\bar j> \bar{j'}$ or $j>j'$ and $\bar j= \bar{j'}$, by the symmetry we can consider only the first case and as $\bar j \notin \{{j},j', \bar{j'}\}$ for which the generic term is equal to zero.
		        \item If $j=j'$ and $\bar j=\bar{j'}$ then the generic term rewrites 
		        $
		         \mathbb{E}(((\delta W_{i{k'}_1}^{j})^2-\frac{\Delta_n}{m})((\delta W_{i{k'}_1}^{\bar{j}})^2-\frac{\Delta_n}{m}))=0.
		        $
		       \end{itemize}
%%%%%%%%%%%%%%%%%%%%%%%%%%%%%%%%%%%%%%%%%%%%%%%%%%%%%%%%%%%%%%%%%%%%%%%%%%
\paragraph{{\bf Second triangular array.}}	
Using the anti-symmetric structure of $\zeta^{n,jj'}_{i,2}$ it is also sufficient to consider only the case $j> j'$ as $\zeta^{n,jj}_{i,2}=0$. Then, for $i\in\{1,\dots, n\}$, we have 
		\begin{align*}
			&	\mathbb{E}((\zeta^{n,jj'}_{i,2})^2|\mathcal{F}_{\frac{i-1}{n}})=\sum_{1\leq k<k'\leq m}n\mathbb{E}\left(\left(\delta W_{ik}^j\delta W_{ik'}^{j'}-\delta W_{ik'}^j\delta W_{ik}^{j'}\right)^2\right)\\&+\sum_{\substack{1\leq k_1<{k'}_1\leq m\\1\leq k_2<{k'}_2\leq m\\(k_1,{k'}_1)\neq (k_2,{k'}_2)}}n\mathbb{E}\left(\left(\delta W_{ik_1}^j\delta W_{i{k'}_1}^{j'}-\delta W_{i{k'}_1}^j\delta W_{ik_1}^{j'}\right)\left(\delta W_{ik_2}^j\delta W_{i{k'}_2}^{j'}-\delta W_{i{k'}_2}^j\delta W_{ik_2}^{j'}\right)\right).
		\end{align*}
In the same way as the first triangular array, the generic term of the above second sum is also equal to zero. This follows easily by expanding this generic term and using  the independence structure between the increments under conditions  $(k_1,{k'}_1)\neq (k_2,{k'}_2)$, $k_1<{k'}_1$ and $k_2<{k'}_2$.
Now, concerning the generic term of the first sum, as $j>j'$ it is easy to check that
$
\mathbb{E}\left(\left(\delta W_{ik}^j\delta W_{ik'}^{j'}-\delta W_{ik'}^j\delta W_{ik}^{j'}\right)^2\right)=\frac{2}{n^2m^2}.
$
Thus, as ${n\rightarrow\infty}$,	
$\sum_{i=1}^{[nt]}\mathbb{E}((\zeta^{n,jj'}_{i,2})^2|\mathcal{F}_{\frac{i-1}{n}})
			{\rightarrow}\frac{m-1}{m}t.$
Now, it remains to check that for any  $q\ge j>j'\ge 1$, $q\ge\bar{j}>\bar{j'}\ge 1$ s.t. $(j,j')\neq(\bar{j},\bar{j'})$, $\sum_{i=1}^{[nt]}\mathbb{E}(\zeta_{i,2}^{n,jj'}\zeta_{i,2}^{n,\bar{j}\bar{j'}}|\mathcal{F}_{\frac{i-1}{n}})=0.$
 So, we have
		\begin{align*}
			&\mathbb{E}(\zeta_{i,2}^{n,jj'}\zeta_{i,2}^{n,\bar{j}\bar{j'}}|\mathcal{F}_{\frac{i-1}{n}})\\=&\sum_{\substack{1\leq k_1<{k'}_1\leq m\\1\leq k_2<{k'}_2\leq m}}\hspace{-0.5cm}n\mathbb{E}((\delta W_{i,k_1}^{j}\delta W_{i,{k'}_1}^{j'}-\delta W_{i,{k'}_1}^{j}\delta W_{i,k_1}^{j'})(\delta W_{ik_2}^{\bar{j}}\delta W_{i{k'}_2}^{\bar{j'}}-\delta W_{i{k'}_2}^{\bar{j}}\delta W_{ik_2}^{\bar{j'}})).
		\end{align*}
		When $(k_1,{k'}_1)\neq(k_2,{k'}_2)$ by similar arguments as for the first triangular array the generic term of the above sum is equal to zero thanks to the independence between the increments. When $(k_1,{k'}_1)=(k_2,{k'}_2)$ $(j,j')\neq (\bar j,\bar{j')}$, we need to treat  two cases.
		\begin{itemize}
		 \item If $j=\bar{j}$ and $j'\neq \bar{j'}$ as $j>j'$ we have $j'\notin \{j,\bar j, \bar{j'}\}$ and consequently the generic term  is equal to zero. If $j\neq\bar{j}$ and $j'= \bar{j'}$ we use similar arguments to prove that the generic term is zero.  
		 \item If $j\neq\bar j$, $j'\neq \bar{j'}$ we have two possibilities: either $j\neq\bar{j'}$ then $j\notin\{j',\bar j, \bar{j'}\}$ or $\bar j >\bar{j'}=j>j'$ and in both cases it is obvious that the generic term also vanishes. 
		\end{itemize}
		
%%%%%%%%%%%%%%%%%%%%%%%%%%%%%%%%%%%%%%%%%%%%%%%%%%%%%%%%%%%%%%%%%%%%%%%%
\paragraph{{\bf Third triangular array.}}	
Using the symmetric structure of $\zeta^{n,jj'j''}_{i,3}$ it is also sufficient to consider only the case  $j\geq j''$. 
 Then, for $i\in\{1,\dots, n\}$, we have 
		\begin{align*}
			&\mathbb{E}((\zeta_{i,3}^{n,jj'j''})^2|\mathcal{F}_{\frac{i-1}{n}})
			=\sum_{1\leq k''<k'<k\leq m}n^2\mathbb{E}\left(\delta W^{j'}_{ik'}(\delta W^{j}_{ik}\delta W^{j''}_{ik''}+\delta W^{j}_{ik''}\delta W^{j''}_{ik})\right)^2\\
			&+\hspace{-0.5cm}\sum_{\substack{1\leq {k''}_1<{k'}_1<k_1\leq m\\1\leq {k''}_2<{k'}_2<k_2\leq m\\(k_1,{k'}_1,{k''}_1)\neq(k_2,{k'}_2,{k''}_2)}} \hspace{-1cm}n^2\mathbb{E}\left(\delta W^{j'}_{i{k'}_1}(\delta W^{j}_{ik_1}\delta W^{j''}_{i{k''}_1}+\delta W^{j}_{i{k''}_1}\delta W^{j''}_{ik_1})\delta W^{j'}_{i{k'}_2}(\delta W^{j}_{ik_2}\delta W^{j''}_{i{k''}_2}+\delta W^{j}_{i{k''}_2}\delta W^{j''}_{ik_2})\right).
		\end{align*}
Similarly as for the first triangular array, we use the independence structure between the increments under conditions
$(k_1,{k'}_1,{k''}_1)\neq(k_2,{k'}_2,{k''}_2)$, $1\leq {k''}_1<{k'}_1<k_1\leq m$ and $1\leq {k''}_2<{k'}_2<k_2\leq m$  
to check that the generic term of the  second sum is also equal to zero. 	Now, concerning the generic term of the first sum, if $j=j''$ then  it is equal to
$$
4\mathbb{E}(\delta W_{ik'}^{j'})^2\mathbb{E}(\delta W_{ik}^{j})^2\mathbb{E}(\delta W_{ik''}^{j})^2=\frac{4}{n^3m^3}
$$
and for $j>j'$ it is equal to 
$$
\mathbb{E}(\delta W_{ik'}^{j'})^2(\mathbb{E}(\delta W_{ik}^{j})^2\mathbb{E}(\delta W_{ik''}^{j''})^2+
\mathbb{E}(\delta W_{ik''}^{j})^2\mathbb{E}(\delta W_{ik}^{j''})^2)=\frac{2}{n^3m^3}.
$$
Thus, we get as ${n\rightarrow\infty}$
	$$\sum_{i=1}^{[nt]}\mathbb{E}((\zeta^{n,jj'j''}_{i,3})^2|\mathcal{F}_{\frac{i-1}{n}})	\displaystyle {\longrightarrow}\left\{\begin{array}{ll}
		\displaystyle  \frac{(m-1)(m-2)}{3m^2}t&j>j''\\\\\displaystyle \frac{2(m-1)(m-2)}{3m^2}t&j=j''.
	\end{array}\right.
$$		
Now, it remains to check that for any $q\ge j\geq j''\ge 1$, $q\ge \bar{j}\geq \bar{j''}\ge 1$ and $j',\bar{j'}\in\{1,\dots,q\}$ s.t. $(j,j',j'')\neq(\bar{j},\bar{j'},\bar{j''})$ , we have
		$\sum_{i=1}^{[nt]}\mathbb{E}(\zeta_{i,3}^{n,jj'j''}\zeta_{i,3}^{n,j\bar{j'}\bar{j''}}|\mathcal{F}_{\frac{i-1}{n}})=0.$ To do so, we write
		\begin{align*}
			&\mathbb{E}(\zeta_{i,3}^{n,jj'j''}\zeta_{i,3}^{n,j\bar{j'}\bar{j''}}|\mathcal{F}_{\frac{i-1}{n}})\\
			=&n^2\hspace{-1cm}\sum_{\substack{1\leq {k''}_1<{k'}_1<k_1\leq m\\1\leq {k''}_2<{k'}_2<k_2\leq m}} \hspace{-1cm}\mathbb{E}\left(\delta W^{j'}_{i{k'}_1}(\delta W^{j}_{ik_1}\delta W^{j''}_{i{k''}_1}+\delta W^{j}_{i{k''}_1}\delta W^{j''}_{ik_1})\delta W^{j'}_{i{k'}_2}(\delta W^{j}_{ik_2}\delta W^{j''}_{i{k''}_2}+\delta W^{j}_{i{k''}_2}\delta W^{j''}_{ik_2})\right).	
		\end{align*}
		When $({k''}_1,{k'}_1,{k}_1)\neq({k''}_2,{k'}_2,{k}_2)$, by similar arguments as for the first triangular array it is easy to check that the generic term of the above sum is equal to zero. When $({k''}_1,{k'}_1,{k}_1)=({k''}_2,{k'}_2,{k}_2)$, with the condition $(j,j',j'')\neq(\bar{j},\bar{j'},\bar{j''})$, the generic term equals to
	$$\mathbb{E}(\delta W_{i{k'}_1}^{{j'}}\delta W_{i{k'}_1}^{\bar{j'}})\mathbb{E}((\delta W_{i{k_1}}^{{j''}}\delta W_{i{k''}_1}^{j}+\delta W_{i{k}_1}^{j}\delta W_{i{k''}_1}^{{j''}})(\delta W_{i,{k}_1}^{\bar{j''}}\delta W_{i{k''}_1}^{\bar{j}}+\delta W_{i{k}_1}^{\bar{j}}\delta W_{i{k''}_1}^{\bar{j''}})).$$ 
	By the same arguments used to treat \eqref{cov_first_TA} we easily deduce that the above generic term vanishes.
%%%%%%%%%%%%%%%%%%%%%%%%%%%%%%%%%%%%%%%%%%%%%%%%%%%%%%%%%%%%%%%%%%%%
\paragraph{{\bf Covariance between of the different triangular arrays.} }
	For any $j,j',j''$ and $\bar{j},\bar{j'}$ in $\{1,\dots,q\}$ $j\geq j'$, $\bar{j}> \bar{j'}$, we have
		\begin{align*}
			&\mathbb{E}(\zeta_{i,1}^{n,jj'j''}\zeta_{i,2}^{n,\bar{j}\bar{j'}}|\mathcal{F}_{\frac{i-1}{n}})\\=&n\sqrt{n}\hspace{-1cm}\sum_{\substack{1\leq k_1,{k'}_1,k_2,{k'}_2\leq m\\k_1\neq {k'}_1\\k_2<{k'}_2}}\hspace{-1cm}\mathbb{E}((\delta W_{ik_1}^{j}\delta W_{ik_1}^{j'}-\delta_{jj'}\frac{\Delta_n}{m})\delta W_{i{k'}_1}^{j''}(\delta W_{ik_2}^{\bar{j}}\delta W_{i{k'}_2}^{\bar{j'}}-\delta W_{i{k'}_2}^{\bar j}\delta W_{ik_2}^{\bar{j'}})).
		\end{align*}
	 For any $j,{j'},{j''}$ and $\bar{j},\bar{j'},\bar{j''}$ in $\{1,\dots,q\}$ $j\geq {j'}$, $\bar{j}\geq \bar{j''}$,  we have
		\begin{align*}
			&\mathbb{E}(\zeta_{i,1}^{n,j{j'}{j''}}\zeta_{i,3}^{n,\bar{j}\bar{j'}\bar{j''}}|\mathcal{F}_{\frac{i-1}{n}})\\=&n^2\hspace{-1cm}\sum_{\substack{1\leq k_1,{k'}_1\leq m\\k_1\neq {k'}_1\\1\le {k''}_2<{k'}_2<k_2\le m}}\hspace{-0.5cm}\mathbb{E}((\delta W_{ik_1}^{j}\delta W_{ik_1}^{{j'}}-\delta_{j{j'}}\frac{\Delta_n}{m})\delta W_{i{k'}_1}^{{j''}}\delta W_{i{k'}_2}^{\bar{j'}}(\delta W_{i{k_2}}^{\bar{j}}\delta W_{i{k''}_2}^{\bar{j''}}+\delta W_{i{k_2}}^{\bar{j''}}\delta W_{i{k''}_2}^{\bar{j}})).
		\end{align*}
	 For any $j,{j'},{j''}$ and $\bar{j},\bar{j'}\in \{1,\dots,q\}$ $j\geq {j''}$, $\bar{j}> \bar{j'}$, we have
		\begin{align*}
			&\mathbb{E}(\zeta_{i,3}^{n,j{j'}{j''}}\zeta_{i,2}^{n,\bar{j}\bar{j'}}|\mathcal{F}_{\frac{i-1}{n}})\\=&n\sqrt{n}\hspace{-0.5cm}\sum_{\substack{1\le {k''}_1<{k'}_1<k_1\le m\\1\leq k_2<{k'}_2\leq m}}\hspace{-0.5cm}\mathbb{E}(\delta W_{i{k'}_1}^{{j'}}(\delta W_{i{k_1}}^{{j}}\delta W_{i{k''}_1}^{{j''}}+\delta W_{i{k_1}}^{{j''}}\delta W_{i{k''}_1}^{{j}})(\delta W_{ik_2}^{\bar{j}}\delta W_{i{k'}_2}^{\bar{j'}}-\delta W_{i{k'}_2}^{\bar{j'}}\delta W_{ik_2}^{\bar{j'}})).
		\end{align*}
		When developing the above three generic terms, we notice that we always have a product of an odd number of increments of the Brownian motion. Then, combining this together with the independence structure between the increments, we easily get  
		$\sum_{i=1}^{[nt]}\mathbb{E}(\zeta_{i,\alpha}^{n,jj'j''}\zeta_{i,\beta}^{n,\bar j{\bar j'}}|\mathcal{F}_{\frac{i-1}{n}})=0, \mbox{for all }  \alpha,\beta\in\{1,2,3\}\; \mbox{ with } \alpha\neq\beta.$
		
	\paragraph{\textbf{Step 3}} {{\bf Independence with respect to the original Brownian motion}} We check the condition \eqref{eq:c} of theorem \ref{thm:d3}.
	\paragraph{{\bf The first triangular array}}
		For any $j,j',j''$ and $j_1$ in $\{1,\dots,q\}$, $j\geq j'$, using the independence between the increments, we have
		\begin{align*} 
			\mathbb{E}(\zeta_{i,1}^{n,jj'j''}\Delta W_{i}^{j_1}|\mathcal{F}_{\frac{i-1}{n}})		&=\sum_{\substack{1\leq k_1,k_2\leq m\\k_1\neq k_2}}n\mathbb{E}((\delta W_{ik_1}^j\delta W_{ik_1}^{j'}-\delta_{jj'}\frac{\Delta_n}{m})\delta W_{ik_2}^{j''}(\delta W_{ik_1}^{j_1}+\delta W_{ik_2}^{j_1}))\\&
			=\sum_{\substack{1\leq k_1,k_2\leq m\\k_1\neq k_2}}n\mathbb{E}((\delta W_{ik_1}^j\delta W_{ik_1}^{j'}-\delta_{jj'}\frac{\Delta_n}{m})\delta W_{ik_1}^{j_1})\mathbb E(\delta W_{ik_2}^{j''}) 
			\\&+\sum_{\substack{1\leq k_1,k_2\leq m\\k_1\neq k_2}}n\mathbb{E}((\delta W_{ik_1}^j\delta W_{ik_1}^{j'}-\delta_{jj'}\frac{\Delta_n}{m}))\mathbb E(\delta W_{ik_2}^{j''}\delta W_{ik_2}^{j_1})
			=0.
			\end{align*}
		\paragraph{{\bf The second triangular array}}
		For any $j,j'$ and $j_1$ in $\{1,\dots,q\}$, it is straight forward that 
		$$
			\mathbb{E}(\zeta_{i,2}^{n,jj'}\Delta W_{i}^{j_1}|\mathcal{F}_{\frac{i-1}{n}})
			=\sum_{\substack{1\leq k_1,k_2\leq m\\k_1<k_2}}\sqrt{n}\mathbb{E}((\delta W_{ik_1}^{j}\delta W_{ik_2}^{j'}-\delta W_{ik_2}^{j}\delta W_{ik_1}^{j'})(\delta W_{ik_1}^{j_1}+\delta W_{ik_2}^{j_1}))=0
		$$
	since when developping the generic term of the above sum we always have the expectation of a product of an odd number of the Brownian increments. 	
		 \paragraph{{\bf The third triangular array}}
		 
		 For any $j,j',j''$, $j\geq j''$ and $j_1$ in $\{1,\dots,q\}$,   using the independence between the different increments we have
	$$
			\mathbb{E}(\zeta_{i,3}^{n,jj'j''}\Delta W_{i}^{j_1}|\mathcal{F}_{\frac{i-1}{n}})
			=n\hspace{-0.7cm}\sum_{\substack{1\le k'' <k'<k\le m}}\hspace{-0.7cm}\mathbb{E}(\delta W^{j'}_{ik'}(\delta W^{j}_{ik}\delta W^{j''}_{ik''}+\delta W^{j}_{ik''}\delta W^{j''}_{ik})(\delta W_{ik'}^{j_1}+\delta W_{ik}^{j_1}+\delta W_{ik''}^{j_1}))=0.
		$$	
	\paragraph{\textbf{Step 4}}({\bf Lyapunov's condition})  Now we check condition \eqref{eq:d} of theorem \ref{thm:d3}.  
	\paragraph{{\bf First triangular array.}} For any $j,j',j''\in\{1,\dots,q\}$, $j\geq j'$, we prove $\sum_{i=1}^{[nt]}\mathbb{E}(|\zeta_{i,1}^{n,jj'j''}|^4|\mathcal{F}_{\frac{i-1}{n}})$ tends to $0$ when $n\rightarrow\infty.$
		In fact, using the convexity property of the function $x \mapsto x^4$ we note first that there is a constant $C_q>0$ depending only on $q$ such that 
		$$
		\mathbb{E}(|\zeta_{i,1}^{n,jj'j''}|^4|\mathcal{F}_{\frac{i-1}{n}})
			\leq C_q\sum_{\substack{1\leq k_1,k_2\leq q\\k_1\neq k_2}}n^4\mathbb{E}\left((\delta W^j_{ik_1}\delta W^{j'}_{ik_1}-\delta_{jj'}\Delta_n/m)^4(\delta W^{j''}_{ik_2})^4\right).$$
Then by the scaling property of the Brownian motion it is easy to check that there is a constant $C_m>0$ depending only on $m$ such that for all $j,j',j''\in\{1,\dots,q\}$, $j\geq j'$ and $1\leq k_1,k_2\leq q$ with $k_1\neq k_2$, we have $\mathbb{E}\left((\delta W^j_{ik_1}\delta W^{j'}_{ik_1}-\delta_{jj'}\Delta_n/m)^4(\delta W^{j''}_{ik_2})^4\right) \leq \frac{C_m}{n^6}$.

\paragraph{{\bf Second triangular array.}}  Similarly, for any $j,j'\in\{1,\dots,q\}$, $j>j'$, there is a constant $C_q>0$ depending only on $q$ such that 
$$
\mathbb{E}(|\zeta_{i,2}^{n,jj'}|^4|\mathcal{F}_{\frac{i-1}{n}})\leq	C_q\sum_{\substack{1\leq k_1,k_2\leq q\\k_1<k_2}}n^2\mathbb{E}\left((\delta W^j_{ik_1}\delta W^{j'}_{ik_2}-\delta W^{j'}_{ik_1}\delta W^j_{ik_2})^4\right)
$$
and we deduce  the result using the estimate $\mathbb{E}\left((\delta W^j_{ik_1}\delta W^{j'}_{ik_2}-\delta W^{j'}_{ik_1}\delta W^j_{ik_2})^4\right)\leq \frac{C_m}{n^4}$ where $C_m$ is a positive constant depending only on $m$.
\paragraph{{\bf Third triangular array.}} In the same way we have  $\mathbb{E}\left(\delta W^{j'}_{ik'}(\delta W^{j}_{ik}\delta W^{j''}_{ik''}+\delta W^{j}_{ik''}\delta W^{j''}_{ik})\right)^4\leq \frac{C_m}{n^6}$
for $C_m>0$.
\end{proof}
%%%%%%%%%%%%%%%%%%%%%%%%%%%%%%%%%%%
 Now we are ready to prove the convergence in law of the couple of main terms $(\mathcal{M}^{n,1},\mathcal{M}^{n,2})$ given by \eqref{eq:newmain1} and \eqref{eq:newmain2}. The following proposition is the core of our main result Theorem \ref{thm:main}. 
\begin{proposition}\label{prop:M}
		As $n\to\infty$ , we have
		\begin{align}\label{eq:M}(\sqrt{n}\mathcal{M}^{n,1},n\mathcal{M}^{n,2})\stackrel{\rm stably}{\Rightarrow}(\mathcal{M}_{1},\mathcal{M}_{2}),\end{align}
		where for $\ell\in\{1,\dots,d\}$, the $\ell^{th}$ components  of $\mathcal{M}_{1}$ and $\mathcal{M}_{2}$ are given by  $$\mathcal{M}_{1,t}^{\ell}=-2\int_0^th_{\ell\bullet\bullet}(X_s)\blackdiamond dZ_{2,s} \mbox{ and } \mathcal{M}_{2,t}^{\ell}=\sum_{r=1}^{4}\bar{\Gamma}_{\ell,t}(r)+\tilde{N}_{\ell,t}+\tilde{M}^{1}_{\ell,t}-\frac{1}{2}\tilde{M}^{3}_{\ell,t},  \;\;t\in[0,1],$$ with 
		\begin{align*}
		\bar{\Gamma}_{\ell,t}(1)=&\frac{m-1}{2m}\int_0^t\bigg(\nabla f_{\ell}(X_s)^{\top} f(X_s)+\frac{1}{2}\sum_{j,j'=1}^qg_{\bullet j}(X_s)^{\top}\nabla^2f_{\ell}(X_s)g_{\bullet j'}(X_s)\bigg)ds\\
		\bar{\Gamma}_{\ell,t}(2)=&\frac{m-1}{2m}\sum_{j=1}^q\int_0^t\bigg(\nabla f_{\ell}(X_{s})^{\top}g_{\bullet j}(X_{s})+\nabla g_{\ell j}(X_s)^{\top}f(X_s)+
	\frac{1}{2}g(X_s)^{\top}\nabla^2 g_{\ell j}(X_s)g(X_s)\blackdiamond {\Iq} \bigg)dW^{j}_s\\
		\bar{\Gamma}_{\ell,t}(3)=&\frac{1}{2}\sum_{j=1}^q\int_0^t\bigg[\nabla g_{\ell j}(X_{s})^{\top}\mathbb{H}(X_{s})+\frac{1}{2}g(X_s)^{\top}\nabla^2g_{\ell j}(X_s)g(X_s)+{\dot h}_{\ell\bullet\bullet}^s\blackdiamond g_{\bullet j}(X_s)\bigg]\blackdiamond dZ_{1,s}^{\bullet\bullet j}\\
		\bar{\Gamma}_{\ell,t}(4)=&\frac{1}{2}\sum_{j,j'=1}^q\int_0^t\nabla g_{\ell j}(X_s)^{\top}[{\dot{g}^s}\blackdiamond g_{\bullet j''}(X_s)]dZ^{j\bullet j'}_{3,s}+\frac{1}{2}\sum_{j=1}^q\int_0^t[g(X_s)^{\top}\nabla^2 g_{\ell  j}(X_{s})g(X_s)]\blackdiamond dZ_{3,s}^{j \bullet\bullet}\\
		\tilde{N}_{\ell,t}=&\int_0^t\frac{1}{8}U^{\top}_s\nabla^2f_{\ell}(X_s)U_sds\\
		\tilde{M}^{1}_{\ell,t}=&\sum_{j=1}^q\int_0^t\frac{1}{8}U^{\top}_s\nabla^2 g_{\ell j}(X_s)U_sdW^{j}_s\\
		\tilde{M}^{3}_{\ell,t}=&\int_0^t({\dot h}_{\ell\bullet\bullet}^s\blackdiamond U_s)\blackdiamond dZ_{2,s},
		\end{align*}
		 where $\dot{g}^s\in(\R^{d\times1})^{d\times q}$ is a block matrix such that for $\ell\in\{1,\dots,d\}$, $j\in\{1,\dots,q\}$, the $\ell j$-th block is given by $(\dot{g}^s)_{\ell j}=\nabla g_{\ell j}(X_s)$, $s\in[0,t]$ and   ${\dot  h}_{\ell \bullet\bullet}^s\in(\R^{d\times 1})^{q\times q}$ is a random block matrix such that for $j$ and $j'\in\{1,\dots,q\}$, the $jj'$-th block is given by  $({\dot  h}_{\ell \bullet\bullet}^s)_{j j'}=\nabla h_{\ell jj'}( X_s)\in \R^{d\times 1}$, $s\in[0,t]$. Here, $Z_{1}$, $Z_{2}$ and $Z_{3}$ are defined above in Theorem \ref{thm:1} and for any $j,j''\in\{1,\dots,q\}$, for $r\in\{1,3\}$, we denote $$Z_{r,s}^{\bullet\bullet j}=\left(\begin{array}{ccc}Z_{r,s}^{11 j}&\hdots&Z_{r,s}^{1q j}\\\vdots&\ddots&\vdots\\Z_{r,s}^{q1 j}&\hdots&Z_{r,s}^{qq j}\end{array}\right)\quad \textrm{and} \quad Z^{j\bullet j''}_{r,t}=(Z^{j1j''}_{r,t},\dots,Z^{jq j''}_{r,t})^{\top}.$$
\end{proposition}
\begin{proof}
At first, let us denote $\rho^n=(W,Z^n_0,Z^n_1,Z^n_2,Z^n_3)$. From Lemma \ref{prop:1} and Theorem \ref{thm:1} combined with Lemma \ref{lm:d5}, we deduce that $\rho^n\stackrel{\rm stably}{\Rightarrow}\rho,$
as $n\to\infty$  with $\rho=(W,Z_0,Z_1,Z_2,Z_3)$.
Besides, as the coefficients $\dot{f}^n_i$ and  $\dot{g}^n_i$ are functions of vector points lying between $X^{nm}_{\frac{i-1}{n}}$ and $X^{nm,\sigma}_{\frac{i-1}{n}}$, the equation \eqref{eq:3.2} can be rewritten into the following continuous form 
\begin{multline*}
	\sqrt{n}U^n_{t}=\sum_{j=0}^q\int_{0}^{\eta_n(t)}\dot{F}^{n,j}_{\eta_n(s)}\blackdiamond \sqrt{n}U^n_{\eta_n(s)}dY^j_s-\int_{0}^{\eta_n(t)}(\mathbb{H}(X^{nm}_{\eta_n(s)})+\mathbb{H}(X^{nm,\sigma}_{\eta_n(s)}))\blackdiamond dZ_{2,s}^n+\sqrt{n}\mathcal{R}^{n,1}_{t},
\end{multline*}
where \begin{align*}
\dot{F}^{n,j}_{\frac{i-1}{n}}=	\left\{\begin{array}{ll}\dot{f}^{n}_i,&j=0\\(\dot{g}^{n}_i)_{\bullet j},&j\in\{1,\dots,q\}\end{array}\right., \hskip 0.1cm \textrm{ where }(\dot{g}^{n}_i)_{\bullet j}=((\dot{g}^{n}_i)_{1 j},\hdots,(\dot{g}^{n}_i)_{d j})^{\top}.
	\end{align*}
Here we used that $\int_{\frac{i-1}{n}}^{\frac i n}dZ^n_{2,s}=Z^n_{2,\frac i n}-Z^n_{2,\frac{i-1}{n}}$ and $Y_t=(t,W_t^1,\hdots,W_t^q)^{\top}$.  Thanks to lemmas \ref{lem:1} and  \ref{Lp}, under assumption  (\nameref{Assume}) the process $(\mathbb{H}(X^{nm})+\mathbb{H}(X^{nm,\sigma}))-(2\mathbb{H}(X))\stackrel{L^p}{\rightarrow}0$. 
 Then, since $\rho^n$  is  \eqref{ut} (see Theorem \ref{thm:1}) we deduce thanks to  Theorem \ref{thm:B5} that as $n\to\infty$
 $$(\rho^n,\sqrt{n}\mathcal{M}^{n,1})=(\rho^n,\int(\mathbb{H}(X^{nm}_{\eta_n(s)})+\mathbb{H}(X^{nm,\sigma}_{\eta_n(s)}))\blackdiamond dZ_{2,s}^n)\stackrel{\rm stably}{\Rightarrow}(\rho,\int2\mathbb{H}(X_s)\blackdiamond dZ_{2,s}).$$
 Moreover, under assumption (\nameref{Assume}) by  lemmas \ref{lem:1} and  \ref{Lp}, it is straightforward that for any $j\in\{0,\dots,q\}$,  $\int\dot{F}^{n,j}_{\eta_n(s)}\blackdiamond \mathds{1}_ddY^j_s-\int\dot{F}^{j}_{s}\blackdiamond \mathds{1}_ddY^j_s\stackrel{L^p}{\rightarrow}0$, with $\mathds{1}_d=(1,\hdots,1)^{\top}$. Thus, by Lemma \ref{lm:d5} we deduce that as $n\to\infty$
   \begin{align*}(\rho^n,\sqrt{n}\mathcal{M}^{n,1},\int\dot{F}^{n,j}_{\eta_n(s)}\blackdiamond \mathds{1}_ddY^j_s)\stackrel{\rm stably}{\Rightarrow}(\rho,\int2\mathbb{H}(X_s)\blackdiamond dZ_{2,s},\int\dot{F}^{j}_{s}\blackdiamond \mathds{1}_ddY^j_s),
   \end{align*}
   with $\dot{F}^0_s=\nabla f(X_s)$ and for any $j\in\{1,\dots,q\}$, $\dot{F}^j_s=\nabla g_{\bullet j}(X_s)$.
Therefore, by Lemma \ref{lem:R1} and Theorem \ref{thm:2.5 1998} we get that 
\begin{align}\label{eq:vect}(\rho^n,\sqrt{n}\mathcal{M}^{n,1}, \sqrt{n}U^n)\stackrel{\rm stably}{\Rightarrow}(\rho,J,U), \quad \mbox{as } n\to\infty.
\end{align}
Now let us recall that \eqref{eq:Ntild}, \eqref{eq:M1tild} and \eqref{eq:M3tild}, can be rewritten into a continuous  form
\begin{align*}
	n{\tilde N}^{nm}_{\ell,t}=&\int_0^{\eta_n(t)}\frac{1}{16}\sqrt{n}{U^n_{\eta_n(t)}}^{\top}\left(\nabla^2 f_{\ell}(\zeta^{n,1}_{\eta_n(t)})+\nabla^2 f_{\ell}(\zeta^{n,2}_{\eta_n(t)})\right)\sqrt{n}U^n_{\eta_n(t)}ds,\\
	n{\tilde M}^{nm,1}_{\ell,t}=&\int_0^{\eta_n(t)}\frac{1}{16}\sum_{j'=1}^q\sqrt{n}{U^n_{\eta_n(t)}}^{\top} \left(\nabla^2g_{\ell j'}(\zeta^{n,3}_{\eta_n(t)})+\nabla^2g_{\ell j'}(\zeta^{n,4}_{\eta_n(t)})\right)\sqrt{n}{U^n_{\eta_n(t)}}dW^{j'}_s,\\
	n{\tilde M}^{nm,3}_{\ell,t}=&\int_0^{\eta_n(t)}\Big[{{\dot h}}^{n,n\eta_n(t)+1,2}_{\ell\bullet\bullet}\blackdiamond\sqrt{n}{U^n_{\eta_n(t)}}\Big] \blackdiamond dZ^n_{2,s}.
\end{align*}
Under the assumption (\nameref{Assume}) and thanks to lemmas \ref{lem:1} and  \ref{Lp} combined with \eqref{eq:vect}, we deduce by Theorem \ref{thm:B5} that 
\begin{align*}(\rho^n,\sqrt{n}\mathcal{M}^{n,1}, n{\tilde N}^{nm},n{\tilde M}^{nm,1},n{\tilde M}^{nm,3})\stackrel{\rm stably}{\Rightarrow}(\rho,J, {\tilde N},{\tilde M}^{1},{\tilde M}^{3})\quad \mbox{as } n\to\infty.
\end{align*}
Similarly, by rewriting $\bar{\Gamma}^n(r)$, $r\in\{1,\dots,4\}$ in continuous forms  we deduce by Theorem \ref{thm:B5}
\begin{multline*}(\rho^n,\sqrt{n}\mathcal{M}^{n,1}, n{\tilde N}^{nm},n{\tilde M}^{nm,1},n{\tilde M}^{nm,3},\bar{\Gamma}^n(1),\bar{\Gamma}^n(2),\bar{\Gamma}^n(3),\bar{\Gamma}^n(4))\\\stackrel{\rm stably}{\Rightarrow}(\rho,J, {\tilde N},{\tilde M}^{1},{\tilde M}^{3},\bar{\Gamma}(1),\bar{\Gamma}(2),\bar{\Gamma}(3),\bar{\Gamma}(4))\quad \mbox{as } n\to\infty,
\end{multline*}
where for $i=1,\dots,4$ and $1\le \ell\le d$  the $\ell$-th component of  the process $\bar{\Gamma}(i)$  is given by the process  $\bar{\Gamma}_{\ell}(i)$. This completes the proof. 
\end{proof}
\appendix

\section{Proofs concerning analysis of $U^n$ and $V^n$}\label{app:B}
\begin{proof}[Proof of Lemma \ref{lem:R1}]
By \eqref{eq:3.2}, we recall that
\begin{multline*}
	\mathcal{R}^{n,1}_{t}=\sum_{i=1}^{[nt]}(\mathbb{H}(X^{nm}_{\frac{i-1}{n}})-\mathbb{H}(X^{nm,\sigma}_{\frac{i-1}{n}}))\blackdiamond (\Delta W_{i}\Delta W_{i}^{\top}-{\Iq}\Delta_n)\\+\sum_{i=1}^{[nt]}(M^{nm,\Id}_{\frac{i-1}{n}}-M^{nm,\sigma}_{\frac{i-1}{n}})+\sum_{i=1}^{[nt]}(N^{nm,\Id}_{\frac{i-1}{n}}-N^{nm,\sigma}_{\frac{i-1}{n}}).
\end{multline*}	
At first,  it is obvious that $\mathbb{E}(\sqrt{n}\sum_{i=1}^{[nt]}(\mathbb{H}(X^{nm}_{\frac{i-1}{n}})-\mathbb{H}(X^{nm,\sigma}_{\frac{i-1}{n}}))\blackdiamond (\Delta W_{i}\Delta W_{i}^{\top}-{\Iq}\Delta_n)|\mathcal{F}_{\frac{i-1}{n}})=0$. Then, by the discrete BDG inequality  combined with Lemma \ref{Lp} and assumption (\nameref{Assume}), there is a generic positive constant $C$ such that
\begin{align*}
	&n^{p/2}\mathbb{E}\left(\sup_{0\leq t\leq 1}|\sum_{i=1}^{[nt]}(\mathbb{H}(X^{nm}_{\frac{i-1}{n}})-\mathbb{H}(X^{nm,\sigma}_{\frac{i-1}{n}}))\blackdiamond (\Delta W_{i}\Delta W_{i}^{\top}-{\Iq}\Delta_n)|^p\right)\\
	&\leq Cn^{p/2}\mathbb{E} \left(\sum_{i=1}^{n}|\mathbb{H}(X^{nm}_{\frac{i-1}{n}})-\mathbb{H}(X^{nm,\sigma}_{\frac{i-1}{n}})|^2||\Delta W_{i}\Delta W_{i}^{\top}-{\Iq}\Delta_n|^2\right)^{p/2}\\
	&{\leq} C n^{p-1}\sum_{i=1}^{n}\mathbb{E}|\mathbb{H}(X^{nm}_{\frac{i-1}{n}})-\mathbb{H}(X^{nm,\sigma}_{\frac{i-1}{n}})|^p\mathbb{E}|\Delta W_{i}\Delta W_{i}^{\top}-{\Iq}\Delta_n|^p \leq C \Delta_n^{p/2}.
\end{align*}
Then the process  $\sqrt{n}\sum_{i=1}^{[n.]}(\mathbb{H}(X^{nm}_{\frac{i-1}{n}})-\mathbb{H}(X^{nm,\sigma}_{\frac{i-1}{n}}))\blackdiamond (\Delta W_{i}\Delta W_{i}^{\top}-{\Iq}\Delta_n)\stackrel{L^p}{\rightarrow} 0$. 
In the same way as above,  we use the discrete  BDG inequality and \eqref{eq:2.6.3a}, there exists a positive constant $C$ such that 
\begin{align*}
        \mathbb{E}(\sup_{0\leq t\leq 1} |\sum_{i=1}^{[nt]}M^{nm,\Id}_{\frac{i-1}{n}}-M^{nm,\sigma}_{\frac{i-1}{n}}|^p)\leq& Cn^{p/2-1}\sum_{i=1}^n\mathbb{E}(|M^{nm,\Id}_{\frac{i-1}{n}}-M^{nm,\sigma}_{\frac{i-1}{n}}|^p)\leq C\Delta_n^p. 
\end{align*}
Therefore, we obtain also the convergence of the process $\sqrt{n} \sum_{i=1}^{[n.]}(M^{nm,\Id}_{\frac{i-1}{n}}-M^{nm,\sigma}_{\frac{i-1}{n}})\stackrel{L^p}{\rightarrow}0$.
Now, by $\eqref{eq:2.6.3b}$, we have
$\mathbb{E}(\sup_{0\leq t\leq 1} |\sum_{i=1}^{[nt]}(N^{nm,\Id}_{\frac{i-1}{n}}-N^{nm,\sigma}_{\frac{i-1}{n}})|^p)$ is bounded by
$$\ n^{p-1}\sum_{i=1}^n\mathbb{E}(|N^{nm,\Id}_{\frac{i-1}{n}}-N^{nm,\sigma}_{\frac{i-1}{n}}|^p)\leq n^{p-1}\sum_{i=1}^n2K_p\Delta_n^{2p}=2K_p\Delta_n^p.$$
Thus, we get  $\sqrt{n} \sum_{i=1}^{[n.]}(N^{nm,\Id}_{\frac{i-1}{n}}-N^{nm,\sigma}_{\frac{i-1}{n}})\stackrel{L^p}{\rightarrow}0$. 
Finally, we get the  \eqref{ut} of $Z^n_0$, $Z^{n}_{1}$, $Z^{n}_{2}$ and $Z^{n}_{3}$ thanks to Lemma \ref{lm:d1}.
\end{proof}
\begin{proof}[Proof of Lemma \ref{lem:R2}]
By \eqref{eq:cal-R2}, we recall that
\begin{multline*}
		\mathcal{R}^{n,2}_t=\frac{1}{2}\sum_{\tilde{\sigma}\in\{\Id,\sigma\}}\sum_{i=1}^{[nt]}\bigg({\tilde R}^{nm,\tilde{\sigma}}_{\frac{i-1}{n}}(0)+\sum_{r=0}^3R^{nm,\tilde{\sigma}}_{\frac{i-1}{n}}(r)\bigg)
		+\tilde{M}^{nm,2}_t\\+\sum_{i=1}^{[nt]}(\mathbb{H}(\bar{X}^{nm,\sigma}_{\frac{i-1}{n}})-\mathbb{H}(X^{n}_{\frac{i-1}{n}}))\blackdiamond (\Delta W_{i}\Delta W_{i}^{\top}-{\Iq}\Delta_n).
	\end{multline*}
At first, thanks to Jensen's inequality and \eqref{eq:tilR0}, there is a positive constant $C$ such that  
\begin{align*}n^p\mathbb{E}(\sup_{0\leq t\leq 1}|\sum_{i=1}^{[nt]}{\tilde R}^{nm,\tilde{\sigma}}_{\frac{i-1}{n}}(0)|^p)\le C n^{2p-1}\sum_{i=1}^{n}\mathbb{E}|{\tilde R}^{nm,\tilde{\sigma}}_{\frac{i-1}{n}}(0)|^p=o(1).\end{align*}
Then  we get $n\sum_{i=1}^{[n.]}{\tilde R}^{nm,\tilde{\sigma}}_{\frac{i-1}{n}}(0)\stackrel{L^p}{\rightarrow}0$ as $n\to\infty$.
Now, we consider $n\sum_{i=1}^{[nt]}R^{nm,\tilde{\sigma}}_{\frac{i-1}{n}}(r)$ for any $r\in\{0,\dots,3\}$.  
By the discrete BDG inequality, \eqref{eq:R0}, \eqref{eq:R1}, \eqref{eq:R2}, \eqref{eq:R3} and Jensen's inequality, there is a generic constant $C>0$ such that  for all $r\in\{0,\dots,3\}$ we have 
$$
n^p\mathbb{E}(\sup_{0\leq t\leq 1}|\sum_{i=1}^{[nt]}{ R}^{nm,\tilde{\sigma}}_{\frac{i-1}{n}}(r)|^p)\le C n^{p}\mathbb{E}\big(\sum_{i=1}^{n}|{ R}^{nm,\tilde{\sigma}}_{\frac{i-1}{n}}(r)|^2\big)^{p/2}
\le C n^{3p/2-1}\sum_{i=1}^{n}\mathbb{E}|{ R}^{nm,\tilde{\sigma}}_{\frac{i-1}{n}}(r)|^p=o(1).
$$
Then we get $n\sum_{i=1}^{[n.]}{R}^{nm,\tilde{\sigma}}_{\frac{i-1}{n}}(r)\stackrel{L^p}{\rightarrow}0$ as n$\to\infty$, for any $r\in\{0,\dots,3\}$.
Next, we recall  from \eqref{eq:M2tild} that for $\ell\in\{1,\dots,d\}$, the $\ell^{th}$ component of the generic term of the martingale triangular array $\tilde{M}^{nm,2}_t$ is  given by
\begin{align*}
	{\tilde M}^{nm,2}_{\ell,\frac{i-1}{n}}=&\frac{1}{4}\Big[{{\dot h}}^{n,i,1}_{\ell\bullet\bullet}\blackdiamond(X^{nm}_{\frac{i-1}{n}}-X^{nm,\sigma}_{\frac{i-1}{n}})\Big] \blackdiamond  (\Delta W_{i}\Delta W_{i}^{\top}-{\Iq}\Delta_n).
\end{align*}
Similarly, we use the discrete BDG and Jensen inequalities to get     
$
	\mathbb{E}(n^p\sup\limits_{0\leq t\leq 1}|\sum_{i=1}^{[nt]}\tilde M^{nm,2}_{\frac{i-1}{n}}|^p )
$
is  bounded by  $C n^{3p/2-1}\sum_{i=1}^{n}	\mathbb{E}|\tilde M^{nm,2}_{\frac{i-1}{n}}|^p$.
Besides, according to \eqref{eq:M2tild}, for $j$ and $j'\in\{1,\dots,q\}$, the $jj'$-th block is given by  $({\dot  h}_{\ell \bullet\bullet}^{n,i,1})_{j j'}=\nabla h_{\ell jj'}( \zeta^{n,5}_{\frac{i-1}{n}})-\nabla h_{\ell jj'}( \zeta^{n,6}_{\frac{i-1}{n}})\in \R^{d\times 1}$ where $ \zeta^{n,5}_{\frac{i-1}{n}}\in(X^{nm}_{\frac{i-1}{n}},\bar{X}^{nm,\sigma}_{\frac{i-1}{n}})$ and $ \zeta^{n,6}_{\frac{i-1}{n}}\in(X^{nm,\sigma}_{\frac{i-1}{n}},\bar{X}^{nm,\sigma}_{\frac{i-1}{n}})$. By using the independence between $\Delta W_i$ and $\mathcal{F}_{\frac{i-1}{n}}$, Cauchy-Schwarz inequality, Lemma \ref{Lp}, Corollary \ref{bar_LP}  and (\nameref{Assume}), we have 
\begin{align*}
	\max_{1\le i\le n}\mathbb{E}|\tilde{M}^{nm,2}_{\frac{i-1}{n}}|^p&\le C \Big[\max_{1\le i\le n}\mathbb{E}|{{\dot h}}^{n,i,1}_{\ell\bullet\bullet}|^{2p}\max_{1\le i\le n}\mathbb{E}|X^{nm}_{\frac{i-1}{n}}-X^{nm,\sigma}_{\frac{i-1}{n}}|^{2p}\Big]^{1/2}\max_{1\le i\le n}\mathbb{E} |\Delta W_{i}\Delta W_{i}^{\top}-{\Iq}\Delta_n|^p\\&\le C (\Delta_n^{p}\Delta_n^{p})^{1/2}\Delta_n^p= C\Delta_n^{2p}.	\end{align*} 
Therefore, we get
$$\mathbb{E}\left(n^p\sup_{0\leq t\leq 1}|\sum_{i=1}^{[nt]}\tilde{M}^{nm,2}_{\frac{i-1}{n}}|^p\right)= O(\Delta_n^{p/2}).$$
Finally, similarly as above, since $\mathbb{E}((\mathbb{H}(\bar{X}^{nm,\sigma}_{\frac{i-1}{n}})-\mathbb{H}(X^{n}_{\frac{i-1}{n}}))\blackdiamond (\Delta W_{i}\Delta W_{i}^{\top}-{\Iq}\Delta_n)|\mathcal{F}_{\frac{i-1}{n}})=0$,  by the discrete BDG and Jensen inequalities and as $\mathbb{E}|\Delta W_{i}\Delta W_{i}^{\top}-{\Iq}\Delta_n|^p=O(\Delta_n^p)$ , we get
$
n^{p}\mathbb{E}\big(\sup\limits_{0\leq t\leq 1}|\sum_{i=1}^{[nt]}(\mathbb{H}(\bar{X}^{nm,\sigma}_{\frac{i-1}{n}})-\mathbb{H}(X^{n}_{\frac{i-1}{n}}))\blackdiamond (\Delta W_{i}\Delta W_{i}^{\top}-{\Iq}\Delta_n)|^p\big)$
is bounded up to a positive multiplicative constant by 
$n^{p/2-1}\sum_{i=1}^{n}\mathbb{E}|\mathbb{H}(\bar{X}^{nm,\sigma}_{\frac{i-1}{n}})-\mathbb{H}(X^{n}_{\frac{i-1}{n}})|^p$. Next, thanks to Corollary \ref{bar_LP}  and assumption (\nameref{Assume}), we deduce that 
this upper bound is $O(\Delta_n^{p/2})$.
Then we get   $n\sum_{i=1}^{[n.]}(\mathbb{H}(\bar{X}^{nm,\sigma}_{\frac{i-1}{n}})-\mathbb{H}(X^{n}_{\frac{i-1}{n}}))\blackdiamond (\Delta W_{i}\Delta W_{i}^{\top}-{\Iq}\Delta_n)\stackrel{L^p}{\rightarrow} 0$ as $n\to\infty$.
\end{proof}

\section{Proof of essential lemmas}\label{app:A}
\begin{proof}[Proof of Lemma \ref{finer}]
By the tower property we have 
 \begin{align*}
\mathbb{E}( M^{nm,\tilde{\sigma},1}_{\ell,\frac{i-1}{n}}|\mathcal{F}_{\frac{i-1}{n}})&=\sum_{k=2}^m{\nabla f_{\ell}}(X^{nm,\tilde{\sigma}}_{\frac{i-1}{n}})^\top\sum_{k'=1}^{k-1}\mathbb{E}\bigg(g(X^{nm,\tilde{\sigma}}_{\frac{m(i-1)+k'-1}{nm}})\mathbb{E}\bigg(\delta W_{i\tilde{\sigma}(k')}\frac{\Delta_n}{m}|\mathcal{F}_{\frac{i-1}{n}} ^{k'-1,\tilde \sigma}\bigg)
|\mathcal{F}_{\frac{i-1}{n}}\bigg)
\\
\mathbb{E}( M^{nm,\tilde{\sigma},2}_{\ell,\frac{i-1}{n}}|\mathcal{F}_{\frac{i-1}{n}})&=\sum_{k=2}^m\sum_{j=1}^q\mathbb{E}\Bigg(\Bigg[\nabla g_{\ell j}^{\top}(X^{nm,\tilde{\sigma}}_{\frac{i-1}{n}})\sum_{k'=1}^{k-1}\bigg(f(X^{nm,\tilde{\sigma}}_{\frac{m(i-1)+k'-1}{nm}})\frac{\Delta_n}{m}+\\
&\hspace{-2cm}\Big[ {\dot g}^n_{ik'}\blackdiamond(X^{nm,\tilde{\sigma}}_{\frac{m(i-1)+k'-1}{nm}}-X^{nm,\tilde{\sigma}}_{\frac{i-1}{n}})\Big]\delta W_{i\tilde{\sigma}(k')}+\mathbb{H}(X^{nm,\tilde{\sigma}}_{\frac{m(i-1)+k'-1}{nm}})\blackdiamond(\delta W_{i\tilde{\sigma}(k')}\delta W_{i\tilde{\sigma}(k')}^{\top}-{\Iq}\frac{\Delta_n}{m})\bigg)\\
&\hspace{-2cm}+\frac{1}{2}(X^{nm,\tilde{\sigma}}_{\frac{m(i-1)+k-1}{nm}}-X^{nm,\tilde{\sigma}}_{\frac{i-1}{n}})^{\top}\nabla^2g_{\ell j}(\xi^{2,n}_{ik})(X^{nm,\tilde{\sigma}}_{\frac{m(i-1)+k-1}{nm}}-X^{nm,\tilde{\sigma}}_{\frac{i-1}{n}})\Bigg]\mathbb{E}(\delta W_{i\tilde{\sigma}(k)}^j |\mathcal{F}_{\frac{i-1}{n}} ^{k-1,\tilde \sigma})
|\mathcal{F}_{\frac{i-1}{n}}\Bigg)\\
\mathbb{E}( M^{nm,\tilde{\sigma},3}_{\ell,\frac{i-1}{n}}|\mathcal{F}_{\frac{i-1}{n}})&=\sum_{k=2}^m\mathbb{E}\bigg( \Big[ {\dot  h}_{\ell \bullet\bullet}^{n,ik}\blackdiamond(X^{nm,\tilde{\sigma}}_{\frac{m(i-1)+k-1}{nm}}-X^{nm,\tilde{\sigma}}_{\frac{i-1}{n}})\Big]\\&\hspace{5cm}\blackdiamond
\mathbb{E}\bigg( \delta W_{i\tilde{\sigma}(k)}\delta W_{i\tilde{\sigma}(k)}^{\top}-{\Iq} \frac{\Delta_n}{m}|\mathcal{F}_{\frac{i-1}{n}} ^{k-1,\tilde\sigma}\bigg)
|\mathcal{F}_{\frac{i-1}{n}}\bigg).
 \end{align*}
Since $\delta W_{i,\tilde\sigma(k)}$ is independent of  $\mathcal{F}_{\frac{i-1}{n}} ^{k-1,\tilde \sigma}$, $k\in\{1,\dots,m\}$, 
$\mathbb{E}( \delta W_{i,\tilde\sigma(k)})=0$ and
$\mathbb{E}( \delta W_{i\tilde{\sigma}(k)}\delta W_{i\tilde{\sigma}(k)}^{\top}-{\Iq} \frac{\Delta_n}{m})=0$, we get $\mathbb{E}(M^{nm,\sigma}_{\frac{i-1}{n}}|\mathcal{F}_{\frac{i-1}{n}})=0$. Now, it remains to have upper bounds for $\mathbb{E}(|M^{nm,\sigma}_{\frac{i-1}{n}}|^p)$ and $\mathbb{E}(|N^{nm,\sigma}_{\frac{i-1}{n}}|^p)$. Thanks to  our assumption (\nameref{Assume}) it is easy to see the existence of $C>0$ s.t. 
 \begin{align*}
|M^{nm,\tilde{\sigma},1}_{\ell,\frac{i-1}{n}}|&\leq C \sum_{k=2}^m \sum_{k'=1}^{k-1}\Big(1+\Big|X^{nm,\tilde{\sigma}}_{\frac{m(i-1)+k'-1}{nm}}\Big|\Big)\Big|\delta W_{i\tilde{\sigma}(k')}\frac{\Delta_n}{m}\Big|
\\
|M^{nm,\tilde{\sigma},2}_{\ell,\frac{i-1}{n}}|&\leq C\sum_{k=2}^m\sum_{j=1}^q\Bigg(\sum_{k'=1}^{k-1}\Bigg[(1+|X^{nm,\tilde{\sigma}}_{\frac{m(i-1)+k'-1}{nm}}|)(\frac{\Delta_n}{m}+|\delta W_{i\tilde{\sigma}(k')}\delta W_{i\tilde{\sigma}(k')}^{\top}-{\Iq}\frac{\Delta_n}{m}|)\\
&+|X^{nm,\tilde{\sigma}}_{\frac{m(i-1)+k'-1}{nm}}-X^{nm,\tilde{\sigma}}_{\frac{i-1}{n}}||\delta W_{i\tilde{\sigma}(k')}|\Bigg]
+\frac{1}{2}|X^{nm,\tilde{\sigma}}_{\frac{m(i-1)+k-1}{nm}}-X^{nm,\tilde{\sigma}}_{\frac{i-1}{n}}|^2\Bigg)|\delta W_{i\tilde{\sigma}(k)}^j |\\
|M^{nm,\tilde{\sigma},3}_{\ell,\frac{i-1}{n}}|&\leq C\sum_{k=2}^m |X^{nm,\tilde{\sigma}}_{\frac{m(i-1)+k-1}{nm}}-X^{nm,\tilde{\sigma}}_{\frac{i-1}{n}}|
|\delta W_{i\tilde{\sigma}(k)}\delta W_{i\tilde{\sigma}(k)}^{\top}-{\Iq} \frac{\Delta_n}{m}|,
 \end{align*}
 Here, the constant $C$ is a generic positive constant whose values may vary from line to line. We obtain \eqref{eq:2.6.3a}  using  the independence between the above increments combined with Lemma \ref{Lp} and the fact that $\mathbb{E}|\delta W_{i,\tilde\sigma(k)}|^p=O(\Delta_n^{p/2})$  and
$\mathbb{E}|\delta W_{i\tilde{\sigma}(k)}\delta W_{i\tilde{\sigma}(k)}^{\top}-{\Iq} \frac{\Delta_n}{m}|^p=O(\Delta_n^{p})$. 
Similar arguments give us inequality \eqref{eq:2.6.3b}. 
\end{proof}

%%%%%%%%%%%%%%%%

%%%%%%%%%%%%%%%%%

\begin{proof}[Proof of Lemma \ref{lem:N}]
Thanks to equations \eqref{eq:iter1} and \eqref{eq:N} combined with \eqref{prop:bd}, we deduce relation \eqref{exp:N} with  :
 \begin{multline*}
	R^{nm,\tilde{\sigma}}_{\ell,\frac{i-1}{n}}(0)=\frac{\Delta_n}{m}\sum_{\substack{k,k'=1\\k'<k}}^m\nabla f^{\top}_{\ell}(X^{nm,\tilde{\sigma}}_{\frac{i-1}{n}})\mathbb{H}(X^{nm,\tilde{\sigma}}_{\frac{m(i-1)+k'-1}{nm}})\blackdiamond(\delta W_{i\tilde{\sigma}(k')}\delta W_{i\tilde{\sigma}(k')}^{\top}-{\Iq} \frac{\Delta_n}{m})\\+\Big[g(X_{\frac{i-1}{n}})^\top\nabla^2f_{\ell}(X_{\frac{i-1}{n}})g(X_{\frac{i-1}{n}})\Big]\blackdiamond \sum_{\substack{k,k',k''=1\\k''<k'<k}}^{m}\delta W_{i\tilde{\sigma}(k')}\delta W_{i\tilde{\sigma}(k'')}^{\top}\frac{\Delta_n}{m}
\end{multline*}
and
 \begin{align*}
 	&{\tilde R}^{nm,\tilde{\sigma}}_{\ell,\frac{i-1}{n}}(0)=\frac{\Delta_n}{m}\sum_{\substack{k,k'=1\\k'<k}}^m\nabla f^{\top}_{\ell}(X^{nm,\tilde{\sigma}}_{\frac{i-1}{n}})\big(f(X^{nm,\tilde{\sigma}}_{\frac{m(i-1)+k'-1}{nm}})-f(X_{\frac{i-1}{n}})\big)\frac{\Delta_n}{m}\\&+	\frac{m-1}{2m}\big(\nabla f^{\top}_{\ell}(X^{nm,\tilde{\sigma}}_{\frac{i-1}{n}})-\nabla f^{\top}_{\ell}(X_{\frac{i-1}{n}})\big)f(X_{\frac{i-1}{n}})\Delta_n^2\\&+\frac{1}{2}\sum_{k=2}^m (X^{nm,\tilde{\sigma}}_{\frac{m(i-1)+k-1}{nm}}-X^{nm,\tilde{\sigma}}_{\frac{i-1}{n}})^{\top}\big(\nabla^2f_{\ell}(\xi^{1,n}_{ik})-\nabla^2f_{\ell}(X_{\frac{i-1}{n}})\big)(X^{nm,\tilde{\sigma}}_{\frac{m(i-1)+k-1}{nm}}-X^{nm,\tilde{\sigma}}_{\frac{i-1}{n}})\frac{\Delta_n}{m}\\&+\frac{1}{2}\sum_{\substack{k,k'=1\\k'<k}}^m ( f(X^{nm,\tilde{\sigma}}_{\frac{m(i-1)+k'-1}{nm}})\frac{\Delta_n}{m}	+\mathbb{H}(X^{nm,\tilde{\sigma}}_{\frac{m(i-1)+k'-1}{nm}})\blackdiamond(\delta W_{i\tilde{\sigma}(k')}\delta W_{i\tilde{\sigma}(k')}^{\top}-{\Iq}\frac{\Delta_n}{m}))^{\top}\\&\hspace{7.5cm}\times\nabla^2f_{\ell}(X_{\frac{i-1}{n}})(X^{nm,\tilde{\sigma}}_{\frac{m(i-1)+k-1}{nm}}-X^{nm,\tilde{\sigma}}_{\frac{i-1}{n}})\frac{\Delta_n}{m}\\&+\frac{1}{2}\sum_{\substack{k,k'=1\\k'<k}}^m \bigg(\big(g(X^{nm,\tilde{\sigma}}_{\frac{m(i-1)+k'-1}{nm}})-g(X_{\frac{i-1}{n}})\big)\delta W_{i\tilde{\sigma}(k')}\bigg)^{\top}\nabla^2f_{\ell}(X_{\frac{i-1}{n}})(X^{nm,\tilde{\sigma}}_{\frac{m(i-1)+k-1}{nm}}-X^{nm,\tilde{\sigma}}_{\frac{i-1}{n}})\frac{\Delta_n}{m}\\&+\frac{1}{2}\sum_{k=2}^m \sum_{k'=1}^{k-1}\big(g(X_{\frac{i-1}{n}})\delta W_{i\tilde{\sigma}(k')}\big)^{\top}\nabla^2f_{\ell}(X_{\frac{i-1}{n}})\bigg(\sum_{k'=1}^{k-1} f(X^{nm,\tilde{\sigma}}_{\frac{m(i-1)+k'-1}{nm}})\frac{\Delta_n}{m}
 	+\sum_{k'=1}^{k-1}\mathbb{H}(X^{nm,\tilde{\sigma}}_{\frac{m(i-1)+k'-1}{nm}})\\&\hspace{2.5cm}\blackdiamond(\delta W_{i\tilde{\sigma}(k')}\delta W_{i\tilde{\sigma}(k')}^{\top}-{\Iq}\frac{\Delta_n}{m})+\sum_{k'=1}^{k-1}\big(g(X^{nm,\tilde{\sigma}}_{\frac{m(i-1)+k'-1}{nm}})-g(X_{\frac{i-1}{n}})\big)\delta W_{i\tilde{\sigma}(k')}\bigg)\frac{\Delta_n}{m}.
 \end{align*}
By the tower property we have 
\begin{align*}
	&\mathbb{E}(R^{nm,\tilde{\sigma}}_{\ell,\frac{i-1}{n}}(0)|\mathcal{F}_{\frac{i-1}{n}})=\mathbb{E}\bigg(\frac{\Delta_n}{m}\sum_{\substack{k,k'=1\\k'<k}}^m\nabla f^{\top}_{\ell}(X^{nm,\tilde{\sigma}}_{\frac{i-1}{n}})\mathbb{H}(X^{nm,\tilde{\sigma}}_{\frac{m(i-1)+k'-1}{nm}})\blackdiamond\mathbb{E}(\delta W_{i\tilde{\sigma}(k')}\delta W_{i\tilde{\sigma}(k')}^{\top}-{\Iq} \frac{\Delta_n}{m}|\mathcal{F}^{k'-1,\tilde{\sigma}}_{\frac{i-1}{n}})\\&+\Big[g(X_{\frac{i-1}{n}})^\top\nabla^2f_{\ell}(X_{\frac{i-1}{n}})g(X_{\frac{i-1}{n}})\Big]\blackdiamond \sum_{\substack{k,k',k''=1\\k''<k'<k}}^{m}\mathbb{E}(\delta W_{i\tilde{\sigma}(k')}|\mathcal{F}^{k'-1,\tilde{\sigma}}_{\frac{i-1}{n}})\delta W_{i\tilde{\sigma}(k'')}^{\top}\frac{\Delta_n}{m}\bigg|\mathcal{F}_{\frac{i-1}{n}}\bigg).
	\end{align*}
Since $\delta W_{i\tilde\sigma(k)}$ is independent of  $\mathcal{F}_{\frac{i-1}{n}} ^{k-1,\tilde \sigma}$, $k\in\{1,\dots,m\}$, 
$\mathbb{E}( \delta W_{i,\tilde\sigma(k)})=0$ and
$\mathbb{E}( \delta W_{i\tilde{\sigma}(k)}\delta W_{i\tilde{\sigma}(k)}^{\top}-{\Iq} \frac{\Delta_n}{m})=0$, we get $\mathbb{E}(R^{nm,\tilde{\sigma}}_{\ell,\frac{i-1}{n}}(0)|\mathcal{F}_{\frac{i-1}{n}})=0$. 
Thanks to  our assumption (\nameref{Assume}) it is easy to see the existence of $C>0$ s.t. 
 \begin{multline*}
	|R^{nm,\tilde{\sigma}}_{\ell,\frac{i-1}{n}}(0)|\le C\Delta_n\sum_{\substack{k,k'=1\\k'<k}}^m(1+|X^{nm,\tilde{\sigma}}_{\frac{m(i-1)+k'-1}{nm}}|)|\delta W_{i\tilde{\sigma}(k')}\delta W_{i\tilde{\sigma}(k')}^{\top}-{\Iq} \frac{\Delta_n}{m})|\\+C{\Delta_n}(1+|X_{\frac{i-1}{n}}|^2) \sum_{\substack{k,k',k''=1\\k''<k'<k}}^{m}|\delta W_{i\tilde{\sigma}(k')}||\delta W_{i\tilde{\sigma}(k'')}|,
\end{multline*}
and
 \begin{align*}
	&|{\tilde R}^{nm,\tilde{\sigma}}_{\ell,\frac{i-1}{n}}(0)|\le C \Delta_n^2\sum_{\substack{k,k'=1\\k'<k}}^m|X^{nm,\tilde{\sigma}}_{\frac{m(i-1)+k'-1}{nm}}-X_{\frac{i-1}{n}}|+C	\Delta_n^2|X^{nm,\tilde{\sigma}}_{\frac{i-1}{n}}-X_{\frac{i-1}{n}}|(1+|X_{\frac{i-1}{n}}|)\\&+C{\Delta_n}\sum_{k=2}^m |X^{nm,\tilde{\sigma}}_{\frac{m(i-1)+k-1}{nm}}-X^{nm,\tilde{\sigma}}_{\frac{i-1}{n}}|^{2}|\nabla^2f_{\ell}^{\top}(\xi_{ik}^{1,n})-\nabla^2f_{\ell}^{\top}(X_{\frac{i-1}{n}})|\\&+C\Delta_n\sum_{\substack{k,k'=1\\k'<k}}^m (1+|X^{nm,\tilde{\sigma}}_{\frac{m(i-1)+k'-1}{nm}}|)(\frac{\Delta_n}{m}	+|\delta W_{i\tilde{\sigma}(k')}\delta W_{i\tilde{\sigma}(k')}^{\top}-{\Iq}\frac{\Delta_n}{m}|)|X^{nm,\tilde{\sigma}}_{\frac{m(i-1)+k-1}{nm}}-X^{nm,\tilde{\sigma}}_{\frac{i-1}{n}}|\\&+C\Delta_n\sum_{\substack{k,k'=1\\k'<k}}^m |X^{nm,\tilde{\sigma}}_{\frac{m(i-1)+k'-1}{nm}}-X_{\frac{i-1}{n}}||\delta W_{i\tilde{\sigma}(k')}||X^{nm,\tilde{\sigma}}_{\frac{m(i-1)+k-1}{nm}}-X^{nm,\tilde{\sigma}}_{\frac{i-1}{n}}|\\&+C\Delta_n\sum_{k=2}^m\bigg( \sum_{k'=1}^{k-1}(1+|X_{\frac{i-1}{n}}|)|\delta W_{i\tilde{\sigma}(k')}|\bigg)\bigg(\sum_{k'=1}^{k-1}|X^{nm,\tilde{\sigma}}_{\frac{m(i-1)+k'-1}{nm}}-X_{\frac{i-1}{n}}||\delta W_{i\tilde{\sigma}(k')}|\\&+\sum_{k'=1}^{k-1} (1+|X^{nm,\tilde{\sigma}}_{\frac{m(i-1)+k'-1}{nm}}|)(\frac{\Delta_n}{m}+|\delta W_{i\tilde{\sigma}(k')}\delta W_{i\tilde{\sigma}(k')}^{\top}-{\Iq}\frac{\Delta_n}{m}|)\bigg).
\end{align*}
Here, the constant $C$ is a generic positive constant whose values may vary from line to line.
Next, using  the independence between the  increments, we get
 \begin{multline*}
	\mathbb{E}|R^{nm,\tilde{\sigma}}_{\ell,\frac{i-1}{n}}(0)|^p\le C{\Delta^p_n}\sum_{\substack{k,k'=1\\k'<k}}^m(1+\mathbb{E}|X^{nm,\tilde{\sigma}}_{\frac{m(i-1)+k'-1}{nm}}|^p)\mathbb E|\delta W_{i\tilde{\sigma}(k')}\delta W_{i\tilde{\sigma}(k')}^{\top}-{\Iq} \frac{\Delta_n}{m}|^p	\\+C{\Delta_n^p}(1+\mathbb{E}|X_{\frac{i-1}{n}}|^{2p}) \sum_{\substack{k,k',k''=1\\k''<k'<k}}^{m}\mathbb{E}|\delta W_{i\tilde{\sigma}(k')}|^p\mathbb{E}|\delta W_{i\tilde{\sigma}(k'')}|^p,
\end{multline*}
and  by Cauchy-Schwarz inequality combined with the independence between the  increments, we also get
\begin{align*}
	&\mathbb{E}|{\tilde R}^{nm,\tilde{\sigma}}_{\ell,\frac{i-1}{n}}(0)|^p\le C{\Delta^{2p}_n}\sum_{\substack{k,k'=1\\k'<k}}^m\mathbb{E}|X^{nm,\tilde{\sigma}}_{\frac{m(i-1)+k'-1}{nm}}-X_{\frac{i-1}{n}}|^p\\&+	C\Delta_n^{2p}\big(\mathbb{E}|X^{nm,\tilde{\sigma}}_{\frac{i-1}{n}}-X_{\frac{i-1}{n}}|^{2p}\big)^{1/2}(1+\mathbb{E}|X_{\frac{i-1}{n}}|^{2p})^{1/2}\\&+C\Delta_n^p\sum_{k=2}^m \big(\mathbb{E}|X^{nm,\tilde{\sigma}}_{\frac{m(i-1)+k-1}{nm}}-X^{nm,\tilde{\sigma}}_{\frac{i-1}{n}}|^{4p}\big)^{1/2}\big(\mathbb{E}|\nabla^2f_{\ell}^{\top}(\xi_{ik}^{1,n})-\nabla^2f_{\ell}^{\top}(X_{\frac{i-1}{n}})|^{2p}\big)^{1/2}\\&+C\Delta_n^p\sum_{\substack{k,k'=1\\k'<k}}^m\bigg( (1+\mathbb{E}|X^{nm,\tilde{\sigma}}_{\frac{m(i-1)+k'-1}{nm}}|^{2p})({\Delta_n^{2p}}+\mathbb E|\delta W_{i\tilde{\sigma}(k')}\delta W_{i\tilde{\sigma}(k')}^{\top}-{\Iq}\frac{\Delta_n}{m}|^{2p})\bigg)^{1/2}\\&\times\big(\mathbb{E}|X^{nm,\tilde{\sigma}}_{\frac{m(i-1)+k-1}{nm}}-X^{nm,\tilde{\sigma}}_{\frac{i-1}{n}}|^{2p}\big)^{1/2}\\&+C\Delta_n^p\sum_{\substack{k,k'=1\\k'<k}}^m \big(\mathbb{E}|X^{nm,\tilde{\sigma}}_{\frac{m(i-1)+k'-1}{nm}}-X_{\frac{i-1}{n}}|^{2p}\mathbb{E}|\delta W_{i\tilde{\sigma}(k')}|^{2p}\big)^{1/2}\big(\mathbb{E}|X^{nm,\tilde{\sigma}}_{\frac{m(i-1)+k-1}{nm}}-X^{nm,\tilde{\sigma}}_{\frac{i-1}{n}}|^{2p}\big)^{1/2}\\&+C\Delta_n^p\sum_{k=2}^m \sum_{k'=1}^{k-1}\big[(1+\mathbb{E}|X_{\frac{i-1}{n}}|^{2p})\mathbb{E}|\delta W_{i\tilde{\sigma}(k')}|^{2p}\big]^{1/2}\bigg(\sum_{k'=1}^{k-1} (1+\mathbb{E}|X^{nm,\tilde{\sigma}}_{\frac{m(i-1)+k'-1}{nm}}|^{2p})\Delta_n^{2p}
	+\\&\sum_{k'=1}^{k-1}(1+\mathbb{E}|X^{nm,\tilde{\sigma}}_{\frac{m(i-1)+k'-1}{nm}}|^{2p})\mathbb{E}|\delta W_{i\tilde{\sigma}(k')}\delta W_{i\tilde{\sigma}(k')}^{\top}-{\Iq}\frac{\Delta_n}{m}|^{2p}\\&+\sum_{k'=1}^{k-1}\mathbb{E}|X^{nm,\tilde{\sigma}}_{\frac{m(i-1)+k'-1}{nm}}-X_{\frac{i-1}{n}}|^{2p}\mathbb{E}|\delta W_{i\tilde{\sigma}(k')}|^{2p}\bigg)^{1/2}.
\end{align*}
Now, using Lemma \ref{lem:1} combined with Lemma \ref{Lp} and the fact that $\mathbb{E}|\delta W_{i,\tilde\sigma(k)}|^p=O(\Delta_n^{p/2}),$ 
$\mathbb{E}|\delta W_{i\tilde{\sigma}(k)}\delta W_{i\tilde{\sigma}(k)}^{\top}-{\Iq} \frac{\Delta_n}{m}|^p=O(\Delta_n^{p})$, we obtain \eqref{eq:R0} and we have
\begin{align*}
	&\mathbb{E}|{\tilde R}^{nm,\tilde{\sigma}}_{\ell,\frac{i-1}{n}}(0)|^p= O(\Delta_n^{5p/2})+O(\Delta_n^{2p})\sum_{k=2}^m \big(\mathbb{E}|\nabla^2f_{\ell}^{\top}(\xi_{ik}^{1,n})-\nabla^2f_{\ell}^{\top}(X_{\frac{i-1}{n}})|^{2p}\big)^{1/2}.
\end{align*}
We recall from relation \eqref{eq:xi1n} section 4 that $\xi^{1,n}_{ik}\in(X^{nm,\tilde{\sigma}}_{\frac{i-1}{n}},X^{ nm,\tilde{\sigma}}_{\frac{m(i-1)+k-1}{nm}})$. Then, by using  Lemma \ref{lem:1}, Lemma \ref{Lp} and the assumption (\nameref{Assume}), we have $\mathbb{E}|\nabla^2f_{\ell}^{\top}(\xi_{ik}^{1,n})-\nabla^2f_{\ell}^{\top}(X_{\frac{i-1}{n}})|^{2p}=O(\Delta_n^p)$, which yields us \eqref{eq:tilR0}.
\end{proof}
\begin{proof}[Proof of Lemma \ref{lem:M1}]
Thanks to equation \eqref{eq:M1}, we deduce from relation \eqref{exp:M1}  the exact form of $R^{nm,\tilde{\sigma}}_{\ell,\frac{i-1}{n}}(1)$ that is given by
\begin{multline*}
	R^{nm,\tilde{\sigma}}_{\ell,\frac{i-1}{n}}(1)=\frac{\Delta_n}{m}{\nabla f_{\ell}}(X^{nm,\tilde{\sigma}}_{\frac{i-1}{n}})^\top\sum_{\substack{k,k'=1\\k'<k}}^{m}\big(g(X^{nm,\tilde{\sigma}}_{\frac{m(i-1)+k'-1}{nm}})-g(X_{\frac{i-1}{n}})\big)\delta W_{i\tilde{\sigma}(k')}\\+\frac{\Delta_n}{m}\big({\nabla f_{\ell}}(X^{nm,\tilde{\sigma}}_{\frac{i-1}{n}})^\top-{\nabla f_{\ell}}(X_{\frac{i-1}{n}})^\top\big) g(X_{\frac{i-1}{n}})\sum_{\substack{k,k'=1\\k'<k}}^{m}\delta W_{i\tilde{\sigma}(k')}.
\end{multline*}
By the tower property we have 
\begin{align*}
	\mathbb{E}(R^{nm,\tilde{\sigma}}_{\ell,\frac{i-1}{n}}(1)|\mathcal{F}_{\frac{i-1}{n}})&=\mathbb{E}\bigg(\frac{\Delta_n}{m}{\nabla f_{\ell}}(X^{nm,\tilde{\sigma}}_{\frac{i-1}{n}})^\top\sum_{\substack{k,k'=1\\k'<k}}^{m}\big(g(X^{nm,\tilde{\sigma}}_{\frac{m(i-1)+k'-1}{nm}})-g(X_{\frac{i-1}{n}})\big)\mathbb{E}(\delta W_{i\tilde{\sigma}(k')}|\mathcal{F}^{k'-1,\tilde{\sigma}}_{\frac{i-1}{n}})\\&+\frac{\Delta_n}{m}\big({\nabla f_{\ell}}(X^{nm,\tilde{\sigma}}_{\frac{i-1}{n}})^\top-{\nabla f_{\ell}}(X_{\frac{i-1}{n}})^\top\big) g(X_{\frac{i-1}{n}})\sum_{\substack{k,k'=1\\k'<k}}^{m}\mathbb{E}(\delta W_{i\tilde{\sigma}(k')}|\mathcal{F}^{k'-1,\tilde{\sigma}}_{\frac{i-1}{n}})|\mathcal{F}_{\frac{i-1}{n}}\bigg).
\end{align*}
Since $\delta W_{i,\tilde\sigma(k)}$ is independent of  $\mathcal{F}_{\frac{i-1}{n}} ^{k-1,\tilde \sigma}$, $k\in\{1,\dots,m\}$, 
$\mathbb{E}( \delta W_{i,\tilde\sigma(k)})=0$ and then $\mathbb{E}(R^{nm,\tilde{\sigma}}_{\ell,\frac{i-1}{n}}(1)|\mathcal{F}_{\frac{i-1}{n}})=0$. Now,
thanks to  our assumption (\nameref{Assume}) it is easy to see the existence of $C>0$ s.t. 
\begin{multline*}
	|R^{nm,\tilde{\sigma}}_{\ell,\frac{i-1}{n}}(1)|\le C\Delta_n\sum_{\substack{k,k'=1\\k'<k}}^{m}|X^{nm,\tilde{\sigma}}_{\frac{m(i-1)+k'-1}{nm}}-X_{\frac{i-1}{n}}||\delta W_{i\tilde{\sigma}(k')}|\\+C\Delta_n|X^{nm,\tilde{\sigma}}_{\frac{i-1}{n}}-X_{\frac{i-1}{n}}| (1+|X_{\frac{i-1}{n}}|)\sum_{\substack{k,k'=1\\k'<k}}^{m}|\delta W_{i\tilde{\sigma}(k')}|.
\end{multline*}
Here, the constant $C$ is a generic positive constant whose values may vary from line to line.
Next, applying Cauchy-Schwarz  inequality and using  the independence between the increments, we get
\begin{multline*}
	\mathbb{E}|R^{nm,\tilde{\sigma}}_{\ell,\frac{i-1}{n}}(1)|^p\le C\Delta_n^p\sum_{\substack{k,k'=1\\k'<k}}^{m}\mathbb{E}|X^{nm,\tilde{\sigma}}_{\frac{m(i-1)+k'-1}{nm}}-X_{\frac{i-1}{n}}|^p\mathbb{E}|\delta W_{i\tilde{\sigma}(k')}|^p\\+C\Delta_n^p\big(\mathbb{E}|X^{nm,\tilde{\sigma}}_{\frac{i-1}{n}}-X_{\frac{i-1}{n}}|^{2p}\big)^{1/2} (1+\big(\mathbb{E}|X_{\frac{i-1}{n}}|^{2p}\big)^{1/2})\sum_{\substack{k,k'=1\\k'<k}}^{m}\mathbb{E}|\delta W_{i\tilde{\sigma}(k')}|^p.
\end{multline*}
By the fact that $\mathbb{E}|\delta W_{i,\tilde\sigma(k)}|^p=O(\Delta_n^{p/2})$  and
$\mathbb{E}|\delta W_{i\tilde{\sigma}(k)}\delta W_{i\tilde{\sigma}(k)}^{\top}-{\Iq} \frac{\Delta_n}{m}|^p=O(\Delta_n^{p})$, we have
\begin{multline*}
	\mathbb{E}|R^{nm,\tilde{\sigma}}_{\ell,\frac{i-1}{n}}(1)|^p\le C\Delta_n^{3p/2}\sum_{\substack{k,k'=1\\k'<k}}^{m}\mathbb{E}|X^{nm,\tilde{\sigma}}_{\frac{m(i-1)+k'-1}{nm}}-X_{\frac{i-1}{n}}|^p\\+C\Delta_n^{3p/2}\big(\mathbb{E}|X^{nm,\tilde{\sigma}}_{\frac{i-1}{n}}-X_{\frac{i-1}{n}}|^{2p}\big)^{1/2} (1+\big(\mathbb{E}|X_{\frac{i-1}{n}}|^{2p}).
\end{multline*}
We obtain \eqref{eq:R1} using Lemma \ref{lem:1} and Lemma \ref{Lp}.
\end{proof}
\begin{proof}[Proof of Lemma \ref{lem:M2}]
Thanks to equation \eqref{eq:M2} and  \eqref{prop:bd}, we deduce from relation \eqref{exp:M2}  the exact form of $R^{nm,\tilde{\sigma}}_{\ell,\frac{i-1}{n}}(2)$. We have 
\begin{align*}
	&R^{nm,\tilde{\sigma}}_{\ell,\frac{i-1}{n}}(2)=\sum_{j=1}^q\Bigg[\nabla g_{\ell j}^{\top}(X^{nm,\tilde{\sigma}}_{\frac{i-1}{n}})\sum_{\substack{k,k'=1\\k'<k}}^{m}\bigg(\big(f(X^{nm,\tilde{\sigma}}_{\frac{m(i-1)+k'-1}{nm}})-f(X_{\frac{i-1}{n}})\big)\frac{\Delta_n}{m}+\nonumber\\&\big(\mathbb{H}(X^{nm,\tilde{\sigma}}_{\frac{m(i-1)+k'-1}{nm}})-\mathbb{H}(X_{\frac{i-1}{n}})\big)\blackdiamond(\delta W_{i\tilde{\sigma}(k')}\delta W_{i\tilde{\sigma}(k')}^{\top}-{\Iq}\frac{\Delta_n}{m})\bigg)\Bigg]\delta W_{i\tilde{\sigma}(k)}^j\nonumber\\&+\sum_{j=1}^q\Bigg[\big(\nabla g_{\ell j}^{\top}(X^{nm,\tilde{\sigma}}_{\frac{i-1}{n}})-\nabla g_{\ell j}^{\top}(X_{\frac{i-1}{n}})\big)\sum_{\substack{k,k'=1\\k'<k}}^{m}\bigg(f(X_{\frac{i-1}{n}})\frac{\Delta_n}{m}+\nonumber\\&\mathbb{H}(X_{\frac{i-1}{n}})\blackdiamond(\delta W_{i\tilde{\sigma}(k')}\delta W_{i\tilde{\sigma}(k')}^{\top}-{\Iq}\frac{\Delta_n}{m})\bigg)\Bigg]\delta W_{i\tilde{\sigma}(k)}^j\\&+\sum_{k=3}^m\sum_{j=1}^q\big(\nabla g_{\ell j}^{\top}(X^{nm,\tilde{\sigma}}_{\frac{i-1}{n}})-\nabla g_{\ell j}^{\top}(X_{\frac{i-1}{n}})\big)\sum_{k'=2}^{k-1}\Big[ {\dot g}^n_{ik'}\blackdiamond(X^{nm,\tilde{\sigma}}_{\frac{m(i-1)+k'-1}{nm}}-X^{nm,\tilde{\sigma}}_{\frac{i-1}{n}})\Big]\delta W_{i\tilde{\sigma}(k')}\delta W_{i\tilde{\sigma}(k)}^j\\&+\sum_{k=3}^m\sum_{j=1}^q\nabla g_{\ell j}^{\top}(X_{\frac{i-1}{n}})\sum_{k'=2}^{k-1}\Big[ \big({\dot g}^n_{ik'}-{\dot g}^n_{i}\big)\blackdiamond(X^{nm,\tilde{\sigma}}_{\frac{m(i-1)+k'-1}{nm}}-X^{nm,\tilde{\sigma}}_{\frac{i-1}{n}})\Big]\delta W_{i\tilde{\sigma}(k')}\delta W_{i\tilde{\sigma}(k)}^j\\&+\sum_{k=3}^m\sum_{j=1}^q\nabla g_{\ell j}^{\top}(X_{\frac{i-1}{n}})\sum_{k'=2}^{k-1}\Big[ {\dot g}^n_{i}\blackdiamond\bigg(\sum_{k''=1}^{k'-1} f(X^{nm,\tilde{\sigma}}_{\frac{m(i-1)+k''-1}{nm}})\frac{\Delta_n}{m}
 	\\&+\sum_{k''=1}^{k'-1}\mathbb{H}(X^{nm,\tilde{\sigma}}_{\frac{m(i-1)+k''-1}{nm}})\blackdiamond(\delta W_{i\tilde{\sigma}(k'')}\delta W_{i\tilde{\sigma}(k'')}^{\top}-{\Iq}\frac{\Delta_n}{m})\\&+\sum_{k''=1}^{k'-1}\big(g(X^{nm,\tilde{\sigma}}_{\frac{m(i-1)+k''-1}{nm}})-g(X_{\frac{i-1}{n}})\big)\delta W_{i\tilde{\sigma}(k'')}\bigg)\Big]\delta W_{i\tilde{\sigma}(k')}\delta W_{i\tilde{\sigma}(k)}^j\\&+\sum_{k=2}^m\sum_{j=1}^q\frac{1}{2}(X^{nm,\tilde{\sigma}}_{\frac{m(i-1)+k-1}{nm}}-X^{nm,\tilde{\sigma}}_{\frac{i-1}{n}})^{\top}\big(\nabla^2g_{\ell j}(\xi^{2,n}_{ik})-\nabla^2g_{\ell j}(X_{\frac{i-1}{n}})\big)(X^{nm,\tilde{\sigma}}_{\frac{m(i-1)+k-1}{nm}}-X^{nm,\tilde{\sigma}}_{\frac{i-1}{n}})\delta W_{i\tilde{\sigma}(k)}^j\\&+\frac{1}{2}\sum_{j=1}^q\sum_{\substack{k,k'=1\\k'<k}}^m ( f(X^{nm,\tilde{\sigma}}_{\frac{m(i-1)+k'-1}{nm}})\frac{\Delta_n}{m}	+\mathbb{H}(X^{nm,\tilde{\sigma}}_{\frac{m(i-1)+k'-1}{nm}})\blackdiamond(\delta W_{i\tilde{\sigma}(k')}\delta W_{i\tilde{\sigma}(k')}^{\top}-{\Iq}\frac{\Delta_n}{m}))^{\top}\\&\times\nabla^2g_{\ell j}(X_{\frac{i-1}{n}})(X^{nm,\tilde{\sigma}}_{\frac{m(i-1)+k-1}{nm}}-X^{nm,\tilde{\sigma}}_{\frac{i-1}{n}})\delta W_{i\tilde{\sigma}(k)}^j\\&+\frac{1}{2}\sum_{j=1}^q\sum_{\substack{k,k'=1\\k'<k}}^m \bigg(\big(g(X^{nm,\tilde{\sigma}}_{\frac{m(i-1)+k'-1}{nm}})-g(X_{\frac{i-1}{n}})\big)\delta W_{i\tilde{\sigma}(k')}\bigg)^{\top}\nabla^2g_{\ell j}(X_{\frac{i-1}{n}})(X^{nm,\tilde{\sigma}}_{\frac{m(i-1)+k-1}{nm}}-X^{nm,\tilde{\sigma}}_{\frac{i-1}{n}})\delta W_{i\tilde{\sigma}(k)}^j\\&+\sum_{j=1}^q\frac{1}{2}\sum_{k=2}^m \big(\sum_{k'=1}^{k-1}g(X_{\frac{i-1}{n}})\delta W_{i\tilde{\sigma}(k')}\big)^{\top}\nabla^2g_{\ell j}(X_{\frac{i-1}{n}})\bigg(\sum_{k'=1}^{k-1} f(X^{nm,\tilde{\sigma}}_{\frac{m(i-1)+k'-1}{nm}})\frac{\Delta_n}{m}
 	+\\&\sum_{k'=1}^{k-1}\mathbb{H}(X^{nm,\tilde{\sigma}}_{\frac{m(i-1)+k'-1}{nm}})\blackdiamond(\delta W_{i\tilde{\sigma}(k')}\delta W_{i\tilde{\sigma}(k')}^{\top}-{\Iq}\frac{\Delta_n}{m})+\sum_{k'=1}^{k-1}\big(g(X^{nm,\tilde{\sigma}}_{\frac{m(i-1)+k'-1}{nm}})-g(X_{\frac{i-1}{n}})\big)\delta W_{i\tilde{\sigma}(k')}\bigg)\delta W_{i\tilde{\sigma}(k)}^j.
\end{align*}
By the tower property, the independence of  $\delta W_{i,\tilde\sigma(k)}$ of  $\mathcal{F}_{\frac{i-1}{n}} ^{k-1,\tilde \sigma}$, $k\in\{1,\dots,m\}$, 
and $\mathbb{E}( \delta W_{i,\tilde\sigma(k)})=0$, we deduce $\mathbb{E}(R^{nm,\tilde{\sigma}}_{\ell,\frac{i-1}{n}}(2)|\mathcal{F}_{\frac{i-1}{n}})=0$. Moreover, 
thanks to  our assumption (\nameref{Assume}) it is easy to see the existence of a generic positive constant $C$ which  does not depend on $i$ s.t. 
\begin{align*}
&|R^{nm,\tilde{\sigma}}_{\ell,\frac{i-1}{n}}(2)|\leq 
C\sum_{j=1}^q\sum_{\substack{k,k'=1\\k'<k}}^{m}|X^{nm,\tilde{\sigma}}_{\frac{m(i-1)+k'-1}{nm}}-X_{\frac{i-1}{n}}|\bigg(\frac{\Delta_n}{m}+|\delta W_{i\tilde{\sigma}(k')}\delta W_{i\tilde{\sigma}(k')}^{\top}-{\Iq}\frac{\Delta_n}{m}|\bigg)|\delta W_{i\tilde{\sigma}(k)}^j|\\
&+C\sum_{j=1}^q |X^{nm,\tilde{\sigma}}_{\frac{i-1}{n}}-X_{\frac{i-1}{n}}|\sum_{\substack{k,k'=1\\k'<k}}^{m}(1+|X_{\frac{i-1}{n}}|)\bigg(\frac{\Delta_n}{m}+|\delta W_{i\tilde{\sigma}(k')}\delta W_{i\tilde{\sigma}(k')}^{\top}-{\Iq}\frac{\Delta_n}{m}|\bigg)|\delta W_{i\tilde{\sigma}(k)}^j|\\
&+C\sum_{k=3}^m\sum_{j=1}^q|X^{nm,\tilde{\sigma}}_{\frac{i-1}{n}}-X_{\frac{i-1}{n}}|\sum_{k'=2}^{k-1}|X^{nm,\tilde{\sigma}}_{\frac{m(i-1)+k'-1}{nm}}-X^{nm,\tilde{\sigma}}_{\frac{i-1}{n}}||\delta W_{i\tilde{\sigma}(k')}||\delta W_{i\tilde{\sigma}(k)}^j|\\
&+C\sum_{k=3}^m\sum_{j=1}^q\sum_{k'=2}^{k-1} |X^{nm,\tilde{\sigma}}_{\frac{m(i-1)+k'-1}{nm}}-X^{nm,\tilde{\sigma}}_{\frac{i-1}{n}}||{\dot g}^n_{ik'}-{\dot g}^n_i||\delta W_{i\tilde{\sigma}(k')}||\delta W_{i\tilde{\sigma}(k)}^j|\\&+C\sum_{k=3}^m\sum_{j=1}^q\sum_{k'=2}^{k-1}\Big[ \sum_{k''=1}^{k'-1} (1+|X^{nm,\tilde{\sigma}}_{\frac{m(i-1)+k''-1}{nm}}|)\big(\frac{\Delta_n}{m}+|\delta W_{i\tilde{\sigma}(k'')}\delta W_{i\tilde{\sigma}(k'')}^{\top}-{\Iq}\frac{\Delta_n}{m}|\big)\\
&+\sum_{k''=1}^{k'-1}|X^{nm,\tilde{\sigma}}_{\frac{m(i-1)+k''-1}{nm}}-X_{\frac{i-1}{n}}||\delta W_{i\tilde{\sigma}(k'')}|\Big]|\delta W_{i\tilde{\sigma}(k')}||\delta W_{i\tilde{\sigma}(k)}^j|\\
&+C\sum_{k=2}^m\sum_{j=1}^q|X^{nm,\tilde{\sigma}}_{\frac{m(i-1)+k-1}{nm}}-X^{nm,\tilde{\sigma}}_{\frac{i-1}{n}}|^2|\nabla^2g_{\ell j}(\xi_{ik}^{2,n})-\nabla^2g_{\ell j}(X_{\frac{i-1}{n}})||\delta W_{i\tilde{\sigma}(k)}^j|\\
&+C\sum_{j=1}^q\sum_{\substack{k,k'=1\\k'<k}}^m  (1+|X^{nm,\tilde{\sigma}}_{\frac{m(i-1)+k'-1}{nm}}|)\big(\frac{\Delta_n}{m}	+|\delta W_{i\tilde{\sigma}(k')}\delta W_{i\tilde{\sigma}(k')}^{\top}-{\Iq}\frac{\Delta_n}{m}|\big)
|X^{nm,\tilde{\sigma}}_{\frac{m(i-1)+k-1}{nm}}-X^{nm,\tilde{\sigma}}_{\frac{i-1}{n}}||\delta W_{i\tilde{\sigma}(k)}^j|\\
&+C\sum_{j=1}^q\sum_{\substack{k,k'=1\\k'<k}}^m |X^{nm,\tilde{\sigma}}_{\frac{m(i-1)+k'-1}{nm}}-X_{\frac{i-1}{n}}||\delta W_{i\tilde{\sigma}(k')}||X^{nm,\tilde{\sigma}}_{\frac{m(i-1)+k-1}{nm}}-X^{nm,\tilde{\sigma}}_{\frac{i-1}{n}}||\delta W_{i\tilde{\sigma}(k)}^j|\\
&+C\sum_{j=1}^q\sum_{k=2}^m \bigg(\sum_{k'=1}^{k-1}(1+|X_{\frac{i-1}{n}}|)|\delta W_{i\tilde{\sigma}(k')}|\bigg)\bigg(\sum_{k'=1}^{k-1} (1+|X^{nm,\tilde{\sigma}}_{\frac{m(i-1)+k'-1}{nm}}|)\big(\frac{\Delta_n}{m}
 	+|\delta W_{i\tilde{\sigma}(k')}\delta W_{i\tilde{\sigma}(k')}^{\top}-{\Iq}\frac{\Delta_n}{m}|\big)\\
&+\sum_{k'=1}^{k-1}|X^{nm,\tilde{\sigma}}_{\frac{m(i-1)+k'-1}{nm}}-X_{\frac{i-1}{n}}||\delta W_{i\tilde{\sigma}(k')}|\bigg)|\delta W_{i\tilde{\sigma}(k)}^j|.
\end{align*}
Here, the values of the constant $C$  may vary from line to line.
Next, we apply Cauchy-Schwarz  inequality and use the independence between the increments, we get
\begin{align*}
	&\mathbb{E}|R^{nm,\tilde{\sigma}}_{\ell,\frac{i-1}{n}}(2)|^p\leq C\sum_{j=1}^q\sum_{\substack{k,k'=1\\k'<k}}^{m}\mathbb{E}|X^{nm,\tilde{\sigma}}_{\frac{m(i-1)+k'-1}{nm}}-X_{\frac{i-1}{n}}|^p\bigg(\frac{\Delta_n^p}{m^p}+\mathbb{E}|\delta W_{i\tilde{\sigma}(k')}\delta W_{i\tilde{\sigma}(k')}^{\top}-{\Iq}\frac{\Delta_n}{m}|^p\bigg)\mathbb{E}|\delta W_{i\tilde{\sigma}(k)}^j|^p\\
	&+C\sum_{j=1}^q\sum_{\substack{k,k'=1\\k'<k}}^{m}\big(\mathbb{E}|X^{nm,\tilde{\sigma}}_{\frac{i-1}{n}}-X_{\frac{i-1}{n}}|^{2p}(1+\mathbb{E}|X_{\frac{i-1}{n}}|^{2p})\big)^{1/2}\bigg(\frac{\Delta_n^p}{m^p}+
	\mathbb{E}|\delta W_{i\tilde{\sigma}(k')}\delta W_{i\tilde{\sigma}(k')}^{\top}-{\Iq}\frac{\Delta_n}{m}|^p\bigg)
	\mathbb{E}|\delta W_{i\tilde{\sigma}(k)}^j|^p\\
	&+C\sum_{k=3}^m\sum_{j=1}^q\big(\mathbb{E}|X^{nm,\tilde{\sigma}}_{\frac{i-1}{n}}-X_{\frac{i-1}{n}}|^{2p}\big)^{1/2}\sum_{k'=2}^{k-1}\big(\mathbb{E}|X^{nm,\tilde{\sigma}}_{\frac{m(i-1)+k'-1}{nm}}-X^{nm,\tilde{\sigma}}_{\frac{i-1}{n}}|^{2p}\big)^{1/2}\mathbb{E}|\delta W_{i\tilde{\sigma}(k')}|^p\mathbb{E}|\delta W_{i\tilde{\sigma}(k)}^j|^p\\
	&+C\sum_{k=3}^m\sum_{j=1}^q\sum_{k'=2}^{k-1}\Big[ \mathbb{E}|X^{nm,\tilde{\sigma}}_{\frac{m(i-1)+k'-1}{nm}}-X_{\frac{i-1}{n}}|^{2p}\mathbb{E}|{\dot g}^n_{ik'}-{\dot g}^n_i|^{2p}\Big]^{1/2}\mathbb{E}|\delta W_{i\tilde{\sigma}(k')}|^p\mathbb{E}|\delta W_{i\tilde{\sigma}(k)}^j|^p\\
	&+C\sum_{k=3}^m\sum_{j=1}^q\sum_{k'=2}^{k-1}\Big[ \sum_{k''=1}^{k'-1} (1+\mathbb{E}|X^{nm,\tilde{\sigma}}_{\frac{m(i-1)+k''-1}{nm}}|^p)\big(\frac{\Delta_n^p}{m^p}+\mathbb{E}|\delta W_{i\tilde{\sigma}(k'')}\delta W_{i\tilde{\sigma}(k'')}^{\top}-{\Iq}\frac{\Delta_n}{m}|^p\big)\\
	&+\sum_{k''=1}^{k'-1}\mathbb{E}|X^{nm,\tilde{\sigma}}_{\frac{m(i-1)+k''-1}{nm}}-X_{\frac{i-1}{n}}|^p\mathbb{E}|\delta W_{i\tilde{\sigma}(k'')}|^p\Big]\mathbb{E}|\delta W_{i\tilde{\sigma}(k')}|^p\mathbb{E}|\delta W_{i\tilde{\sigma}(k)}^j|^p\\
	&+C\sum_{k=2}^m\sum_{j=1}^q\bigg(\mathbb{E}|X^{nm,\tilde{\sigma}}_{\frac{m(i-1)+k-1}{nm}}-X^{nm,\tilde{\sigma}}_{\frac{i-1}{n}}|^{4p}\mathbb{E}|\nabla^2 g_{\ell j}(\xi^{2,n}_{ik})-\nabla^2 g_{\ell j}(X_{\frac{i-1}{n}})|^{2p}\bigg)^{1/2}\mathbb{E}|\delta W_{i\tilde{\sigma}(k)}^j|^p\\
	&+C\sum_{j=1}^q\sum_{\substack{k,k'=1\\k'<k}}^m \bigg( (1+\mathbb{E}|X^{nm,\tilde{\sigma}}_{\frac{m(i-1)+k'-1}{nm}}|^{2p})
	(\frac{\Delta_n^{2p}}{m^{2p}}+\mathbb{E}|\delta W_{i\tilde{\sigma}(k')}\delta W_{i\tilde{\sigma}(k')}^{\top}-{\Iq}\frac{\Delta_n}{m}|^{2p})\bigg)^{1/2}\\
	&\times\big(\mathbb{E}|X^{nm,\tilde{\sigma}}_{\frac{m(i-1)+k-1}{nm}}-X^{nm,\tilde{\sigma}}_{\frac{i-1}{n}}|^{2p}\big)^{1/2}\mathbb{E}|\delta W_{i\tilde{\sigma}(k)}^j|^p\\
	&+C\sum_{j=1}^q\sum_{\substack{k,k'=1\\k'<k}}^m \bigg(\mathbb{E}|X^{nm,\tilde{\sigma}}_{\frac{m(i-1)+k'-1}{nm}}-X_{\frac{i-1}{n}}|^{2p}\mathbb{E}|\delta W_{i\tilde{\sigma}(k')}|^{2p}\bigg)^{1/2}\big(\mathbb{E}|X^{nm,\tilde{\sigma}}_{\frac{m(i-1)+k-1}{nm}}-X^{nm,\tilde{\sigma}}_{\frac{i-1}{n}}|^{2p}\big)^{1/2}\mathbb{E}|\delta W_{i\tilde{\sigma}(k)}^j|^p\\
	&+C\sum_{j=1}^q\sum_{k=2}^m \bigg(\sum_{k'=1}^{k-1}(1+\mathbb{E}|X_{\frac{i-1}{n}}|^{2p})\mathbb{E}|\delta W_{i\tilde{\sigma}(k')}|^{2p}\bigg)^{1/2}\bigg(\sum_{k'=1}^{k-1} (1+\mathbb{E}|X^{nm,\tilde{\sigma}}_{\frac{m(i-1)+k'-1}{nm}}|^{2p})(\frac{\Delta_n^{2p}}{m^{2p}}\\
	&+\mathbb{E}|\delta W_{i\tilde{\sigma}(k')}\delta W_{i\tilde{\sigma}(k')}^{\top}-{\Iq}\frac{\Delta_n}{m}|^{2p})
	+\sum_{k'=1}^{k-1}\mathbb{E}|X^{nm,\tilde{\sigma}}_{\frac{m(i-1)+k'-1}{nm}}-X_{\frac{i-1}{n}}|^{2p}\mathbb{E}|\delta W_{i\tilde{\sigma}(k')}|^{2p}\bigg)^{1/2}\mathbb{E}|\delta W_{i\tilde{\sigma}(k)}^j|^p.
\end{align*}
By  using Lemma \ref{lem:1} combined with Lemma \ref{Lp} and the fact that $\mathbb{E}|\delta W_{i,\tilde\sigma(k)}|^p=O(\Delta_n^{p/2})$,
$\mathbb{E}|\delta W_{i\tilde{\sigma}(k)}\delta W_{i\tilde{\sigma}(k)}^{\top}-{\Iq} \frac{\Delta_n}{m}|^p=O(\Delta_n^{p})$, we get
\begin{multline*}
	\mathbb{E}|R^{nm,\tilde{\sigma}}_{\ell,\frac{i-1}{n}}(2)|^p= O(\Delta_n^{2p})+O(\Delta_n^{3p/2})\sum_{k=3}^m\sum_{k'=2}^{k-1}\Big[\mathbb{E}|{\dot g}^n_{ik'}-{\dot g}^n_i|^{2p}\Big]^{1/2}\\+O(\Delta_n^{3p/2})\sum_{k=2}^m\sum_{j=1}^q\bigg(\mathbb{E}|\nabla^2 g_{\ell j}(\xi^{2,n}_{ik})-\nabla^2 g_{\ell j}(X_{\frac{i-1}{n}})|^{2p}\bigg)^{1/2}.
\end{multline*}
Now,  let us recall that from relation \eqref{eq:M2} we have ${\dot g}^n_{ik'}\in (\R^{d\times 1})^{d\times q}$ and $\dot{g}^n_i\in(\R^{d\times1})^{d\times q}$,  for $\ell\in\{1,\dots,d\},j\in\{1,\dots,q\}$, $({\dot g}^n_{ik'})_{\ell j}=\nabla g_{\ell j}( \xi'^{2,n}_{ik'})\in\R^{d\times 1}$ and $(\dot{g}^n_i)_{\ell j}=\nabla g_{\ell j}(X_{\frac{i-1}{n}})$ where  $ \xi'^{2,n}_{ik'}\in(X^{nm,\tilde{\sigma}}_{\frac{i-1}{n}},X^{nm,\tilde{\sigma}}_{\frac{m(i-1)+k'-1}{nm}})$. We also recall that from \eqref{eq:M2}  $\xi^{2,n}_{ik}\in(X^{nm,\tilde{\sigma}}_{\frac{i-1}{n}},X^{nm,\tilde{\sigma}}_{\frac{m(i-1)+k-1}{nm}})$. Then, by   Lemma \ref{lem:1}, Lemma \ref{Lp} and assumption (\nameref{Assume}), we get $\mathbb{E}|\nabla^2 g_{\ell j}(\xi^{2,n}_{ik})-\nabla^2 g_{\ell j}(X_{\frac{i-1}{n}})|^{2p}=\mathbb{E}|{\dot g}^n_{ik'}-{\dot g}^n_i|^{2p}=O(\Delta_n^{p})$. Hence, we deduce \eqref{eq:R2}. 
\end{proof}
\begin{proof}[Proof of Lemma \ref{lem:M3}]
Thanks to equation \eqref{eq:M3}, we deduce from relation \eqref{exp:M3}  the exact form of $R^{nm,\tilde{\sigma}}_{\ell,\frac{i-1}{n}}(3)$. We have 
\begin{align*}	&R^{nm,\tilde{\sigma}}_{\ell,\frac{i-1}{n}}(3)=\sum_{k=2}^m\Big[ \big({\dot  h}_{\ell \bullet\bullet}^{n,ik}-{\dot  h}_{\ell \bullet\bullet}^{n,i}\big)\blackdiamond(X^{nm,\tilde{\sigma}}_{{\frac{m(i-1)+k-1}{nm}}}-X^{nm,\tilde{\sigma}}_{\frac{i-1}{n}})\Big]\blackdiamond(\delta W_{i\tilde{\sigma}(k)}\delta W_{i\tilde{\sigma}(k)}^{\top}-{\Iq} \frac{\Delta_n}{m})\\
 	&+\sum_{k=2}^m\Big[ {\dot  h}_{\ell \bullet\bullet}^{n,i}\blackdiamond\bigg(\sum_{k'=1}^{k-1} f(X^{nm,\tilde{\sigma}}_{\frac{m(i-1)+k'-1}{nm}})\frac{\Delta_n}{m}
 	+\sum_{k'=1}^{k-1}\mathbb{H}(X^{nm,\tilde{\sigma}}_{\frac{m(i-1)+k'-1}{nm}})\blackdiamond(\delta W_{i\tilde{\sigma}(k')}\delta W_{i\tilde{\sigma}(k')}^{\top}-{\Iq}\frac{\Delta_n}{m})\\&+\sum_{k'=1}^{k-1}\big(g(X^{nm,\tilde{\sigma}}_{\frac{m(i-1)+k'-1}{nm}})-g(X_{\frac{i-1}{n}})\big)\delta W_{i\tilde{\sigma}(k')}\bigg)\Big]\blackdiamond(\delta W_{i\tilde{\sigma}(k)}\delta W_{i\tilde{\sigma}(k)}^{\top}-{\Iq} \frac{\Delta_n}{m}).
\end{align*}
By the tower property we have 
\begin{align*}
	&\mathbb{E}(R^{nm,\tilde{\sigma}}_{\ell,\frac{i-1}{n}}(3)|\mathcal{F}_{\frac{i-1}{n}})=\mathbb{E}\bigg(\sum_{k=2}^m\Big[ \big({\dot  h}_{\ell \bullet\bullet}^{n,ik}-{\dot  h}_{\ell \bullet\bullet}^{n,i}\big)\blackdiamond(X^{nm,\tilde{\sigma}}_{{\frac{m(i-1)+k-1}{nm}}}-X^{nm,\tilde{\sigma}}_{\frac{i-1}{n}})\Big]\blackdiamond\mathbb{E}(\delta W_{i\tilde{\sigma}(k)}\delta W_{i\tilde{\sigma}(k)}^{\top}-{\Iq} \frac{\Delta_n}{m}|\mathcal{F}^{k-1,\tilde{\sigma}}_{\frac{i-1}{n}})\\
	&+\sum_{k=2}^m\Big[ {\dot  h}_{\ell \bullet\bullet}^{n,i}\blackdiamond\bigg(\sum_{k'=1}^{k-1} f(X^{nm,\tilde{\sigma}}_{\frac{m(i-1)+k'-1}{nm}})\frac{\Delta_n}{m}
	+\sum_{k'=1}^{k-1}\mathbb{H}(X^{nm,\tilde{\sigma}}_{\frac{m(i-1)+k'-1}{nm}})\blackdiamond(\delta W_{i\tilde{\sigma}(k')}\delta W_{i\tilde{\sigma}(k')}^{\top}-{\Iq}\frac{\Delta_n}{m})\\&+\sum_{k'=1}^{k-1}\big(g(X^{nm,\tilde{\sigma}}_{\frac{m(i-1)+k'-1}{nm}})-g(X_{\frac{i-1}{n}})\big)\delta W_{i\tilde{\sigma}(k')}\bigg)\Big]\blackdiamond\mathbb{E}(\delta W_{i\tilde{\sigma}(k)}\delta W_{i\tilde{\sigma}(k)}^{\top}-{\Iq} \frac{\Delta_n}{m}|\mathcal{F}^{k-1,\tilde{\sigma}}_{\frac{i-1}{n}})|\mathcal{F}_{\frac{i-1}{n}}\bigg).
\end{align*}
Since $\delta W_{i,\tilde\sigma(k)}$ is independent of  $\mathcal{F}_{\frac{i-1}{n}} ^{k-1,\tilde \sigma}$, $k\in\{1,\dots,m\}$, 
$\mathbb{E}( \delta W_{i\tilde{\sigma}(k)}\delta W_{i\tilde{\sigma}(k)}^{\top}-{\Iq} \frac{\Delta_n}{m})=0$, we get  $\mathbb{E}(R^{nm,\tilde{\sigma}}_{\ell,\frac{i-1}{n}}(3)|\mathcal{F}_{\frac{i-1}{n}})=0$.
Now, thanks to  our assumption (\nameref{Assume}) it is easy to see the existence of constant $C>0$ s.t. 
\begin{align*}	&|R^{nm,\tilde{\sigma}}_{\ell,\frac{i-1}{n}}(3)|\le C\sum_{k=2}^m|{\dot  h}_{\ell \bullet\bullet}^{n,ik}-{\dot  h}_{\ell \bullet\bullet}^{n,i}||X^{nm,\tilde{\sigma}}_{{\frac{m(i-1)+k-1}{nm}}}-X^{nm,\tilde{\sigma}}_{\frac{i-1}{n}}||\delta W_{i\tilde{\sigma}(k)}\delta W_{i\tilde{\sigma}(k)}^{\top}-{\Iq} \frac{\Delta_n}{m}|\\
	&+C\sum_{k=2}^m\bigg(\sum_{k'=1}^{k-1} (1+|X^{nm,\tilde{\sigma}}_{\frac{m(i-1)+k'-1}{nm}}|)(\frac{\Delta_n}{m}
	+|\delta W_{i\tilde{\sigma}(k')}\delta W_{i\tilde{\sigma}(k')}^{\top}-{\Iq}\frac{\Delta_n}{m}|)\\
	&+\sum_{k'=1}^{k-1}|X^{nm,\tilde{\sigma}}_{\frac{m(i-1)+k'-1}{nm}}-X_{\frac{i-1}{n}}||\delta W_{i\tilde{\sigma}(k')}|\bigg)|\delta W_{i\tilde{\sigma}(k)}\delta W_{i\tilde{\sigma}(k)}^{\top}-{\Iq} \frac{\Delta_n}{m}|.
\end{align*}
Here, the constant $C$ is a generic positive constant whose values may vary from line to line.
Next, we apply Cauchy-Schwarz inequality and use  the independence between the increments, to get
\begin{align*}	&\mathbb{E}|R^{nm,\tilde{\sigma}}_{\ell,\frac{i-1}{n}}(3)|^p\le C\sum_{k=2}^m\big(\mathbb{E}|{\dot  h}_{\ell \bullet\bullet}^{n,ik}-{\dot  h}_{\ell \bullet\bullet}^{n,i}|^{2p}\mathbb{E}|X^{nm,\tilde{\sigma}}_{{\frac{m(i-1)+k-1}{nm}}}-X^{nm,\tilde{\sigma}}_{\frac{i-1}{n}}|^{2p}\big)^{1/2}\mathbb{E}|\delta W_{i\tilde{\sigma}(k)}\delta W_{i\tilde{\sigma}(k)}^{\top}-{\Iq} \frac{\Delta_n}{m}|^p\\
	&+C\sum_{k=2}^m\bigg(\sum_{k'=1}^{k-1} (1+\mathbb{E}|X^{nm,\tilde{\sigma}}_{\frac{m(i-1)+k'-1}{nm}}|^{p})(\frac{\Delta_n^p}{m^p}
	+\mathbb{E}|\delta W_{i\tilde{\sigma}(k')}\delta W_{i\tilde{\sigma}(k')}^{\top}-{\Iq}\frac{\Delta_n}{m}|^p)
	\\&+\sum_{k'=1}^{k-1}\mathbb{E}|X^{nm,\tilde{\sigma}}_{\frac{m(i-1)+k'-1}{nm}}-X_{\frac{i-1}{n}}|^p\mathbb{E}|\delta W_{i\tilde{\sigma}(k')}|^p\bigg)\mathbb{E}|\delta W_{i\tilde{\sigma}(k)}\delta W_{i\tilde{\sigma}(k)}^{\top}-{\Iq} \frac{\Delta_n}{m}|^p.
\end{align*}
 By Lemma \ref{lem:1}, Lemma \ref{Lp} and the fact that $\mathbb{E}|\delta W_{i,\tilde\sigma(k)}|^p=O(\Delta_n^{p/2})$,
$\mathbb{E}|\delta W_{i\tilde{\sigma}(k)}\delta W_{i\tilde{\sigma}(k)}^{\top}-{\Iq} \frac{\Delta_n}{m}|^p=O(\Delta_n^{p})$, we have 
\begin{align*}	&\mathbb{E}|R^{nm,\tilde{\sigma}}_{\ell,\frac{i-1}{n}}(3)|^p= O(\Delta_n^{2p})+O(\Delta_n^{3p/2})\sum_{k=2}^m\big(\mathbb{E}|{\dot  h}_{\ell \bullet\bullet}^{n,ik}-{\dot  h}_{\ell \bullet\bullet}^{n,i}|^{2p}\big)^{1/2}.
\end{align*}
We recall from relation \eqref{eq:M3} in Section 4 that ${\dot  h}_{\ell \bullet\bullet}^{n,ik}\in(\R^{d\times 1})^{q\times q}$ and $\dot{h}^{n,i}_{\ell\bullet\bullet}\in(\R^{d\times1})^{q\times q}$, for $j$ and $j'\in\{1,\dots,q\}$, $({\dot  h}_{\ell \bullet\bullet}^{n,ik})_{j j'}=\nabla h_{\ell jj'}( \xi^{3,n}_{ik})\in \R^{d\times 1}$ and $(\dot{h}^{n,i}_{\ell\bullet\bullet})_{j j'}=\nabla h_{\ell j j'}(X_{\frac{i-1}{n}})$ where $ \xi^{3,n}_{ik}\in(X^{nm,\tilde{\sigma}}_{\frac{i-1}{n}},X^{nm,\tilde{\sigma}}_{\frac{m(i-1)+k-1}{nm}})$. Then, by using  Lemma \ref{Lp} and the assumption (\nameref{Assume}), we have $\mathbb{E}|{\dot  h}_{\ell \bullet\bullet}^{n,ik}-{\dot  h}_{\ell \bullet\bullet}^{n,i}|^{2p}=O(\Delta_n^{p})$. Hence, we obtain \eqref{eq:R3}.
\end{proof}
\begin{proof}[Proof of Lemma \ref{average}]
	From \eqref{eq:M1tild}, \eqref{eq:M2tild} and \eqref{eq:M3tild}, we use similar arguments as in the proof of Lemma \ref{finer} to get
	\begin{align*}
	\mathbb{E}({\tilde M}^{nm,1}_{\ell,\frac{i-1}{n}}|\mathcal{F}_{\frac{i-1}{n}})=&\frac{1}{16}\sum_{j'=1}^q(X^{nm}_{\frac{i-1}{n}}-X^{nm,\sigma}_{\frac{i-1}{n}})^\top \left(\nabla^2g_{\ell j'}(\zeta^{n,3}_i)+\nabla^2g_{\ell j'}(\zeta^{n,4}_i)\right)(X^{nm}_{\frac{i-1}{n}}-X^{nm,\sigma}_{\frac{i-1}{n}})^{\top}\mathbb{E}(\Delta W^{j'}_{i}|\mathcal{F}_{\frac{i-1}{n}}),\\
	\mathbb{E}({\tilde M}^{nm,2}_{\ell,\frac{i-1}{n}}|\mathcal{F}_{\frac{i-1}{n}})=&\frac{1}{4}\Big[{{\dot h}}^{n,i,1}_{\ell\bullet\bullet}\blackdiamond(X^{nm}_{\frac{i-1}{n}}-X^{nm,\sigma}_{\frac{i-1}{n}})\Big] \blackdiamond  \mathbb{E}(\Delta W_{i}\Delta W_{i}-{\Iq}\Delta_n|\mathcal{F}_{\frac{i-1}{n}}),\\
	\mathbb{E}({\tilde M}^{nm,3}_{\ell,\frac{i-1}{n}}|\mathcal{F}_{\frac{i-1}{n}})=&\sum_{\substack{k,k'=1\\k<k'}}^m\Big[{{\dot h}}^{n,i,2}_{\ell\bullet\bullet}\blackdiamond(X^{nm}_{\frac{i-1}{n}}-X^{nm,\sigma}_{\frac{i-1}{n}})\Big] \blackdiamond \mathbb{E}(\delta W_{ik}\delta W_{ik'}^{\top}-\delta W_{ik'}\delta W_{ik}^{\top}|\mathcal{F}_{\frac{i-1}{n}}).
\end{align*}
Now if we also consider \eqref{eq:Ntild}, we get thanks to assumption (\nameref{Assume}), the existence of  a generic positive constant $C$ s.t.
	\begin{align*}
	\mathbb{E}|{\tilde N}^{nm}_{\ell,\frac{i-1}{n}}|^p\le&C\mathbb{E}|X^{nm}_{\frac{i-1}{n}}-X^{nm,\sigma}_{\frac{i-1}{n}}|^{2p}\Delta_n,\\
	\mathbb{E}|{\tilde M}^{nm,1}_{\ell,\frac{i-1}{n}}|^p\le&C\sum_{j'=1}^q\mathbb{E}|X^{nm}_{\frac{i-1}{n}}-X^{nm,\sigma}_{\frac{i-1}{n}}|^{2p}\mathbb{E}|\Delta W^{j'}_{i}|^p,\\
	\mathbb{E}|{\tilde M}^{nm,2}_{\ell,\frac{i-1}{n}}|^p\le&C\mathbb{E}|X^{nm}_{\frac{i-1}{n}}-X^{nm,\sigma}_{\frac{i-1}{n}}|^{p} \mathbb{E}|\Delta W_{i}\Delta W_{i}-{\Iq}\Delta_n|^{p},\\
	\mathbb{E}|{\tilde M}^{nm,3}_{\ell,\frac{i-1}{n}}|^p\le&C\sum_{\substack{k,k'=1\\k<k'}}^m\mathbb{E}|X^{nm}_{\frac{i-1}{n}}-X^{nm,\sigma}_{\frac{i-1}{n}}|^{p} \mathbb{E}|\delta W_{ik}\delta W_{ik'}^{\top}-\delta W_{ik'}\delta W_{ik}^{\top}|^{p}.
\end{align*}
Thus,  we  easily deduce that $\mathbb{E}({\tilde M}^{nm,1}_{\ell,\frac{i-1}{n}}|\mathcal{F}_{\frac{i-1}{n}})=\mathbb{E}({\tilde M}^{nm,2}_{\ell,\frac{i-1}{n}}|\mathcal{F}_{\frac{i-1}{n}})=\mathbb{E}({\tilde M}^{nm,3}_{\ell,\frac{i-1}{n}}|\mathcal{F}_{\frac{i-1}{n}})=0$ and using Lemma \ref{Lp} we get  $\mathbb{E}|{\tilde N}^{nm}_{\ell,\frac{i-1}{n}}|^p=O(\Delta_n^{2p})$ and $\mathbb{E}|{\tilde M}^{nm,1}_{\ell,\frac{i-1}{n}}|^p=\mathbb{E}|{\tilde M}^{nm,2}_{\ell,\frac{i-1}{n}}|^p=\mathbb{E}|{\tilde M}^{nm,3}_{\ell,\frac{i-1}{n}}|^p=O(\Delta_n^{3p/2})$.
Finally, combining the above estimates with the obtained bounds on $\mathbb{E}(|M^{nm,\tilde \sigma}_{\frac{i-1}{n}}|)$ and $\mathbb{E}(|N^{nm,\tilde \sigma}_{\frac{i-1}{n}}|)$ for $\tilde \sigma\in\{\Id,\sigma\}$ (see \eqref{eq:2.6.3a} and \eqref{eq:2.6.3b}), we easily get the required bounds for the moments of $M_{\frac{i-1}{n}}$ and $N_{\frac{i-1}{n}}$.
\end{proof}
%%%%%%%%%%%%%%%%%%%%%
%%%%%%%%%%%%%%%%%%%%%
\section{Theoretical tools}\label{app:C}
	\subsection{Uniform tightness}
We first recall the uniform tightness property (UT) defined in Jakubowski, M\'emin and Pag\`es \cite{JakMemPag}.  Let $X^n=(X^{n,i})_{1\leq i\leq d}$ be a sequence of $\mathbb R^d$-valued continuous semimartingales with the decomposition $$X^{n,i}_t =X^{n,i}_0 +A^{n,i}_t +M^{n,i}_t ,\hskip 1cm 0\leq t \leq T,$$ where, for each $n \in\mathbb{N}$ and $1\leq i \leq d$, $A^{n,i}$ is a predictable process with finite variation, null at $0$ and $M^{n,i}$ is a martingale null at $0$.  We say that $X^n$ has (UT) if for each $i$
\begin{equation}\label{ut}
\langle M^{n,i}\rangle_T +\int_0^T|dA^{n,i}_s| \;\;  \mbox{ is tight}\tag{UT}.
\end{equation}
\subsection{Stable convergence}
Let $(X_n)$ be a sequence of random variables with values in a Polish space $E$ defined on a probability space $(\Omega,\mathcal{F},\mathbb{P})$. Let $(\tilde{\Omega},
\tilde{\mathcal{F}},\tilde{\mathbb{P}})$ be an extension of $(\Omega,\mathcal{F},\mathbb{P})$, and let $X$ be an $E$-valued random variable on the extension.
We say that $(X_n)$ converges in law to $X$ stably and write $X_n\stackrel{\rm stably}{\Rightarrow}X$, as $n\to\infty$ if $$\mathbb{E}(Uh(X_n))\rightarrow\tilde{\mathbb{E}}(Uh(X)), \quad \mbox{as }n\to\infty$$
for all $h:E \rightarrow R$ bounded continuous and all bounded random variable $U$ on $(\Omega, \mathcal{F})$. This convergence is obviously stronger than convergence in law that we will denote here by “$\stackrel{\rm stably}{\Rightarrow}$”.

Now, we recall the Lemma 2.1 in \cite{Jacod2004} about the uniform tightness property. For this aim, we consider sums of triangular arrays of the form
\begin{displaymath}
\Gamma_t^n=\sum_{i=1}^{[nt]}\zeta_i^n,
\end{displaymath}
where for each $n$ we have $\R^d$-valued random variables $(\zeta_i^n)_{i\geq1}$ such that each $\zeta_i^n$ is $\mathcal{F}_{i/n}$-measurable.	
\begin{lemma}\label{lm:d1}
	If $\zeta_i^n$ are i.i.d. random variables and $\Gamma^n_1$ converges in law to a limit $U$, then there is a L\'evy process $\Gamma$ such that $\Gamma_1=U$. This process $\Gamma$ is unique in law and $\Gamma^n$ converges in law to $\Gamma$ (for the Skorokhod topology). Further, the sequence $(\Gamma^n)$ has \eqref{ut}. 	 	
\end{lemma}

Next, we recall the convergence theorem 3.2. of Jacod in \cite{c} for an $\mathbb{R}^d$-semimartingale process without jumps of form
$$Z^n_t=\sum_{i=1}^{[nt]}\chi^n_i,$$
where $\chi^n_i$ is $\mathcal{F}_{\frac{i}{n}}$-measurable.
\begin{theorem}\label{thm:d3}
	Assume that $M$ is a square-integrable continuous martingale, and that
	each $\chi$ is square-integrable. Assume also that there are two continuous processes $F$
	and $G$ and a continuous process $b$ of bounded variation on $(\Omega, \mathcal{F}, (\mathcal{F}_t)_{0\leq t\leq 1}, \mathbb{P})$ such that
	\begin{align}
		&\sup_t|\sum_{i=1}^{[nt]}\mathbb{E}(\chi_i^n|\mathcal{F}_{\frac{i-1}{n}})-b_t|\stackrel{\mathbb{P}}{\rightarrow}0,\label{eq:a}\tag{a}\\
		&\sum_{i=1}^{[nt]}\left(\mathbb{E}(\chi_i^n{\chi_i^n}^{\top}|\mathcal{F}_{\frac{i-1}{n}})-\mathbb{E}(\chi_i^n|\mathcal{F}_{\frac{i-1}{n}})\mathbb{E}({\chi_i^n}^{\top}|\mathcal{F}_{\frac{i-1}{n}})\right)\stackrel{\mathbb{P}}{\rightarrow}F_t,\hskip 1cm\forall t\in[0,1],\label{eq:b}\tag{b}\\
		&\sum_{i=1}^{[nt]}\mathbb{E}(\chi_i^n\Delta M_{\frac{i-1}{n}}^{\top}|\mathcal{F}_{\frac{i-1}{n}})\stackrel{\mathbb{P}}{\rightarrow}G_t\hskip 1cm\forall t\in[0,1],\label{eq:c}\tag{c}\\
		&\sum_{i=1}^{n}\mathbb{E}(|\chi_i^n|^21_{|\chi^n_i|>\epsilon}|\mathcal{F}_{\frac{i-1}{n}})\stackrel{\mathbb{P}}{\rightarrow}0\hskip 1cm\forall \epsilon>0\hskip 1cm \textrm{(Lindeberg's condition)}.\label{eq:d}\tag{d}
	\end{align}
	Then assume further that $d\langle M^i,M^i\rangle_t\ll dt$ and $dF^{ii}_t\ll dt$, there are predictable processes $u,v, w$ with values in $\mathbb{R}^{D\times D}$, $\mathbb{R}^{d\times D}$ and $\mathbb{R}^{d\times d}$ respectively, such that
	$$\langle M,M^{\top}\rangle_t=\int_0^tu_su^{\top}_sds,\hskip 1cm G_t=\int_0^t v_su_su_s^{\top}ds,$$
	$$F_t=\int_0^t (v_su_su_s^{\top}v_s^{\top}+w_sw^{\top}_s)ds,$$
	we have  $$Z^n\stackrel{\rm stably}{\Rightarrow}Z,$$
	with the limit $Z$ can be realized on the canonical $d$-dimensional Wiener extension of $(\Omega,\mathcal{F},(\mathcal{F}_t)_{0\leq t\leq 1},\mathbb{P})$, with the canonical Wiener process $B$ as 
	$$Z_t=b_t+\int_0^tu_sdM_s+\int_0^tw_sdB_s.$$
\end{theorem}

\begin{remark}
	If in the theorem above, every $\chi_i^n$, $i\in\{1,\dots,n\}$ have moments of order $p>2$, then the Lindeberg's condition can be obtained by the Lyapunov condition:
	$$\sum_{i=1}^{n}\mathbb{E}(|\chi_i^n|^p|\mathcal{F}_{\frac{i-1}{n}})\stackrel{\mathbb{P}}{\rightarrow}0.$$
\end{remark}

Now, according to Section 2 of Jacod \cite{c} and Lemma 2.1 of Jacod and Protter \cite{d}, we have the following result
\begin{lemma}\label{lm:d5}
	Let $V_n$ and $V$ be defined on $({\Omega},\mathcal{F})$ with values in another metric space $E$. If $V_n\stackrel{\mathbb{P}}{\rightarrow}V$, $X_n\stackrel{\rm stably}{\Rightarrow} X$  then $(V_n,X_n)\stackrel{\rm stably}{\Rightarrow} (V,X)$.\\
	Conversely, if $(V,X_n)\Rightarrow(V,X)$ and $V$ generates the $\sigma$-field $\mathcal{F}$, we can realize this limit as $(V,X)$ with $X$ defined on an extension of $({\Omega},
	\mathcal{F},\mathbb{P})$ and $X_n\stackrel{\rm stably}{\Rightarrow}X$.
\end{lemma}
Now, we recall a result on the convergence of stochastic integrals formulated from Theorem 2.3 in Jacod and Protter \cite{d}. 
\begin{theorem}\label{thm:B5}
	Assume that the sequence $(X^n)$ has \eqref{ut}. Let $H^n$ and $H$ be a sequence of adapted, right-continuous and left-hand side limited processes all defined on the same filtered probability space. If $(H^n,X^n) \stackrel{\rm stably}{\Rightarrow}(H,X)$ then $X$ is a semimartingale with respect to the filtration generated by the limit process $(H,X)$, and we have $(H^n,X^n,\int H^ndX^n)\stackrel{\rm stably}{\Rightarrow}(H,X,\int H dX)$. 	
\end{theorem} 
Now, we recall the Theorem 2.5c in \cite{d}.
\begin{theorem}\label{thm:2.5 1998}
We consider a sequence of SDE's like 
$$X^n_t=J^n_t+\int_0^tX^n_{s-}H^n_sdY_s,$$
all defined on the same filtered probability space and with the same dimensions. Also let $\rho^n$ be an auxiliary sequence of random variables with values in some Polish space $E$, all defined on the same space again.\\
Let $V^n_t=\int_0^tH^n_sdY_s$. Suppose the sequence $\sup_{t\le 1}\|H^n_t\|$ is tight and the sequence $(J^n,V^n,\rho^n)$ stably converges to the limit $(J, V,  \rho)$ defined on some extension of the space. Then $V$ is a semimartingale on some extension and $(J^n,V^n,X^n,\rho^n)$ stably converges to the limit $(J, V,X,  \rho)$ where $X$ is a solution of 
$$X_t=J_t+\int_0^tX_{s-}H_sdY_s.$$
\end{theorem}
\subsection{Lindeberg-Feller central limit theorem}
We recall also the following central limit theorem for triangular array (see, e.g., Theorem 7.2 and 7.3 in \cite{l}).
\begin{theorem}\label{central}
	Let $(k_n)_{n\in\mathbb{N}}$ be a sequence such that $k_n\rightarrow\infty$ as $n\rightarrow\infty$. For each $n$, let $X_{n,1}$,$\hdots$,$X_{n,k_n}$ be $k_n$ independent random variables with finite variance such that $\mathbb{E}(X_{n,k})=0$ for all $k\in\{1,\hdots,k_n\}$. Suppose that the following conditions hold:
	\begin{enumerate}[$(1)$]
		\item $\lim_{n\rightarrow\infty}\sum_{k=1}^{k_n}\mathbb{E}|X_{n,k}|^2=\vartheta$, $\vartheta>0$.
		\item Lindeberg's condition: For all $\epsilon>0$, $\lim_{n\rightarrow\infty}\sum_{k=1}^{k_n}\mathbb{E}(|X_{n,k}|^2\mathbbm{1}_{|X_{n,k}|>\epsilon})=0$. Then 
		\begin{align*}
			\sum_{k=1}^{k_n}X_{n,k}\Rightarrow\mathcal{N}(0,\vartheta),\hskip 1cm \textrm{ as }n\rightarrow\infty.
		\end{align*}
		Moreover, if the $X_{n,k}$ have moments of order $p>2$, then the Lindeberg's condition can be obtained by the following one:
		\item Lyapunov's condition: $\lim_{n\rightarrow\infty}\sum_{k=1}^{k_n}\mathbb{E}|X_{n,k}|^p=0.$
	\end{enumerate}
\end{theorem}

%%%%%%%%%%%%%%  Bibliography  %%%%%%%%%%%%%%%%%

\end{document}